%% file: ms.tex
\documentclass[preprint,pagebackref=true]{elsarticle} % pagebackref=true added, see below
\usepackage{hyperref}
\journal{Journal of Computational and Applied Mathematics}
\usepackage[utf8]{inputenc}
\usepackage{amssymb,upgreek,amsmath,bm,bbm, algorithm2e, graphicx, epstopdf, tikz}
\usepackage[justification=centering]{subfig}
\usepackage[export]{adjustbox}
\usepackage{mystyle} % contains notations and macros
\graphicspath{{./Figures/}}
\usepackage{geometry}
\usepackage[page]{appendix}

\begin{document}

\begin{frontmatter}
\title{Sparsest Piecewise-Linear Regression of One-Dimensional Data}

\author[1]{Thomas Debarre\corref{cor1}}
\cortext[cor1]{Corresponding author}
\ead{thomas.debarre@gmail.com}
\author[1]{Quentin Denoyelle}
\ead{quentin.denoyelle@epfl.ch}
\author[1]{Michael Unser}
\ead{michael.unser@epfl.ch}
\author[2]{Julien Fageot}
\ead{julien.fageot@epfl.ch}
\address[1]{EPFL, Biomedical Imaging Group, Lausanne, Switzerland}
\address[2]{EPFL, AudioVisual Communications Laboratory, Lausanne, Switzerland}

\begin{abstract}
We study the problem of one-dimensional regression of data points with total-variation (TV) regularization (in the sense of measures) on the second derivative, which is known to promote piecewise-linear solutions with few knots. While there are efficient algorithms for determining such adaptive splines, the difficulty with TV regularization is that the solution is generally non-unique, an aspect that is often ignored in practice.
In this paper, we present a systematic analysis that results in a complete description of the solution set with a clear distinction between the cases where the solution is unique and those, much more frequent, where it is not. For the latter scenario, we identify the sparsest solutions, i.e., those with the minimum number of knots, and we derive a formula to compute the minimum number of knots based solely on the data points. To achieve this, we first consider the problem of exact interpolation which leads to an easier theoretical analysis.
%We thoroughly describe the form of the solutions of the underlying constrained optimization problem. In particular, we show that in general, this problems admits uncountably many solutions, and we identify the sparsest piecewise-linear solutions, \ie those with the minimum number of knots.
Next, we relax the exact interpolation requirement to a regression setting, and we consider a penalized optimization problem with a strictly convex data-fidelity cost function. We show that the underlying penalized problem can be reformulated as a constrained problem, and thus that all our previous results still apply. Based on our theoretical analysis, we propose a simple and fast two-step algorithm, agnostic to uniqueness, to reach a sparsest solution of this penalized problem.
%Our work is motivated in part by recent results in the field of machine learning, more precisely for the study of popular ReLU neural networks.
%Indeed, it is well known that such networks produce an input-output relation that is a continuous piecewise-linear function, as does our interpolation algorithm.
\end{abstract}
\end{frontmatter}

\input{sections/sec-intro}
\input{sections/sec-maths-prelim}

\input{sections/sec-gBPC}

\input{sections/sec-sparsest-sol}
\input{sections/sec-gBLASSO}
\input{sections/sec-experiments}
\input{sections/sec-conclusion}
\input{sections/appendix}

\section*{Acknowledgements}
The authors are thankful to Shayan Aziznejad for many discussions related to this work and for his elegant connection between the (\gBLASSO) problem and its discrete counterpart (see~\eqref{eq:optizlambda}). Julien Fageot was supported by the Swiss National Science Foundation (SNSF) under Grants P2ELP2\_181759 and P400P2\_194364. The work of Thomas Debarre, Quentin Denoyelle, and Michael Unser is supported by the SNSF under Grant 200020\_184646/1 and the European Research Council (ERC) under Grant 692726-GlobalBioIm.

%\section*{References}
\bibliography{ms}
\end{document}

%% file: sections/sec-intro.tex
% !TEX root = ../main.tex
%
\section{Introduction}
\label{sec:intro}

%Many machine learning problems consist in learning a function $f:\R^d \to \R$ that maps an input vector $\V{x} \in \R^d$ to an output value $f(\V x) = y$.
Regression problems consist in learning a function $f$ that best approximates some data $(x_m, y_m)_{m=1}^M$, where $M$ is the number of data points, in the sense that $f(x_m) \approx y_m$. This is typically achieved by parametrizing $f$ with a vector of parameters $\V{\theta}$, and minimizing some objective function with respect to $\V{\theta}$. The oldest and most basic form or regression is linear regression: $f$ is parametrized as a linear (or affine) function. Although this model has the advantage of being very simple, it is very limited due to the fact that many data distributions are poorly approximated by linear functions, as illustrated by the dotted line example in Figure~\ref{fig:regression_vs_NN}. The choice of parametrization $\V\theta$ is therefore crucial, as it must strike an appropriate balance between two conflicting desirable properties. Firstly, in order to be suitable for a variety of problems, the parametric model should be flexible enough to represent a large class of functions. In the field of machine learning, where regression is known as \emph{supervised learning}, this quest for universality is for instance highlighted by several universal approximation theorems for artificial neural networks \cite{cybenko1989approximation, hornik1991approximation,leshno1993multilayer}. Next, the model should be simple enough so that it generalizes well to input vectors $\V{x}$ that are outside of the training set. Indeed, a known pitfall of machine learning algorithms is overfitting, which happens when the model is unduly complex and fits too closely to the training data~\cite[Chapter 3]{mitchell1997machine}. This leads to poor generalization abilities for out-of-sample data. This pitfall is often dealt with by adding some regularization to the objective function, which tends to simplify the model. The overarching guiding principle to avoid overfitting is Occam's razor: the simplest model that explains the data well will generalize better and should thus be selected.

\begin{figure}[t]
\centering
\includegraphics[width=0.5\linewidth]{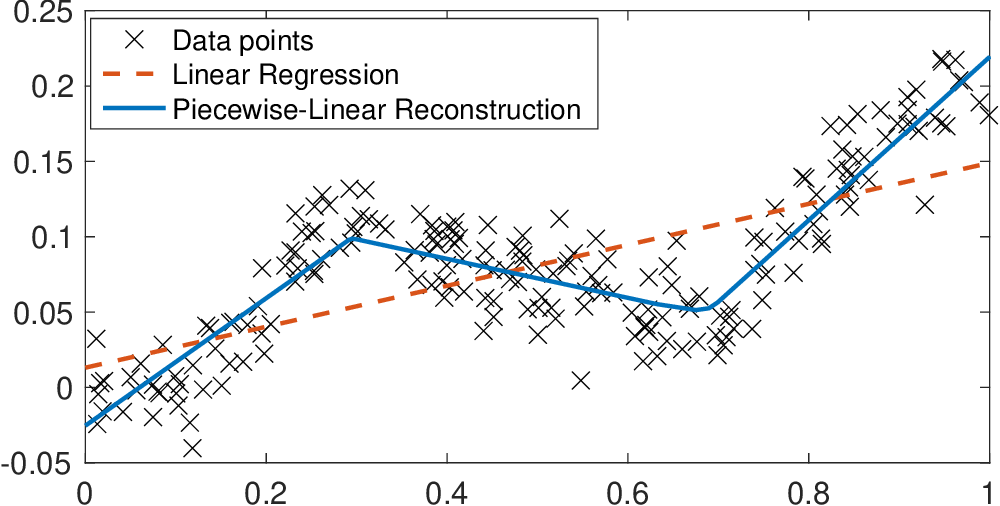}
\caption{Examples of reconstructions}
\label{fig:regression_vs_NN}
\end{figure}

\subsection{Problem Formulation}
In this paper, we study the regression (or supervised learning) problem in one dimension, \ie $f:\R \to \R$ and $x_m, y_m \in \R$. However, instead of parametrizing the reconstructed function, we formulate the learning problem as a regularized inverse problem in a continuous-domain framework. Inspired by their connection (that we discuss later on) to popular ReLU (rectified linear unit) neural networks, we focus on reconstructing piecewise-linear splines. Our metric for model simplicity is sparsity, \ie the number of spline knots.
%Although this sparsity cannot be equated to the number of parameters of a multi-layer ReLU network, it is the only one available to us, since our learned function lies in an infinite-dimensional space.
%Although this sparsity cannot be equated to the number of parameters of a multi-layer ReLU network, it is a natural metric.
For regularization purposes, we therefore use the total variation (TV) norm for measures $\Vert \cdot \Vert_\Spc{M}$, which is defined over the space of bounded Radon measures $\Spc{M}(\R)$. This norm is known to promote sparse solutions in the desired sense, as will be clarified in~\eqref{eq:formsolution}. We formulate the following optimization problem, which we refer to as the \emph{generalized Beurling LASSO}~\eqref{eq:noisyproblem}
\begin{align}\label{eq:noisyproblem}
    \argmin_f  \sum_{m=1}^M E(f(x_m), y_m) + \lambda \Vert\D^2 f \Vert_{\Spc{M}}, \tag{\gBLASSO}
\end{align}
where $E$ is a cost function that penalizes the fidelity of $f(x_m)$ to the data $y_m \in \R$ (\eg a quadratic loss $E(z, y) = \frac{1}{2} (z - y)^2$). We assume that the sampling locations are ordered, \ie $x_1 < \cdots < x_M$. The parameter $\lambda > 0$ balances the contribution of the data fidelity and the regularization, and $\D^2$ is the second-derivative operator. The terminology generalized Beurling LASSO comes from the Beurling LASSO (BLASSO) which is used in the Dirac recovery literature \cite{de2012exact}. Indeed, the~\eqref{eq:noisyproblem} problem is a generalization of the BLASSO due to the presence of a regularization operator $\D^2$, which is not present in the latter problem. It is known~\cite{Unser2017splines,gupta2018continuous,boyer2019representer} that the extreme points solutions to the~\eqref{eq:noisyproblem} are piecewise-linear splines of the form
\begin{align}
\label{eq:formsolution}
    \fopt(x) = b_0 + b_1 x + \sum_{k=1}^K a_k (x-\tau_k)_+,
\end{align}
where $x_+ = \max(0, x)$ is the ReLU, $b_0, b_1, a_k, \tau_k \in \R$, and the number of spline knots $K$ is bounded by $K \leq M-2$. This representer theorem has two important components:
\begin{itemize}
    \item the \eqref{eq:noisyproblem} has solutions of the prescribed form, \ie piecewise-linear splines. This stems from the choice of the regularization, \ie the TV norm of the second derivative;
    \item the sparsity is bounded by the number of training data by $K \leq M-2$. 
\end{itemize}
%Our continuous domain framework is inherently universal, since \eqref{eq:noisyproblem} is formulated over an infinite-dimensional space that thus contains a huge variety of functions.
In terms of model simplicity, the bound $K \leq (M-2)$ is typically uninformative in machine learning problems: in Figure \ref{fig:regression_vs_NN}, it yields $K \leq M-2 = 198$, which is clearly much higher than the desired sparsity. However, this bound does not take the effect of the regularization parameter $\lambda$ into account. Indeed, $\lambda \to 0$ will roughly lead to a learned function $f$ that interpolates all the data points, with typically close to $K = M-2$ knots. At the other extreme, the limit $\lambda \to +\infty$ leads to linear regression and thus sparsity $K=0$ due to the fact that linear functions are not penalized by the regularization. Therefore, the interesting case is the intermediate regime (as illustrated by the solid curve in Figure~\ref{fig:regression_vs_NN}), in which the overall trend is that the sparsity $K$ decreases as $\lambda$ increases. Hence, $\lambda$ controls the universality \vs simplicity trade-off.

\subsection{Summary of Contributions and Outline}
The above purely qualitative observation is far from telling the whole story. In particular, it does not prescribe how $\lambda$ should be chosen in practice. We attempt to overcome this impediment by giving a full description of the solution set of the~\eqref{eq:noisyproblem}.
% To this end, in this paper, we answer the following theoretical questions:
% \begin{itemize}
%     \item Is there a unique solution to problem \eqref{eq:noisyproblem}?
%     \item If not, what is the form of its solution set?
%     \item What is the sparsest solution?
% \end{itemize}
The basis of our analysis is the classical observation (see for instance \cite[Theorem 5]{gupta2018continuous}) that when $E$ is strictly convex, there exists a unique vector $\V{y}_\lambda = (y_{\lambda, 1}, \ldots, y_{\lambda, M}) \in \R^M$ such that the \eqref{eq:noisyproblem} is equivalent to the constrained problem
\begin{align} \label{eq:noiselessproblem}
    \argmin_{\substack{f: f(x_m) = y_{\lambda,m}, \\ m\in\{1,\ldots,M\}}} \Vert\D^2 f \Vert_{\Spc{M}} \tag{\gBPC},
\end{align}
which we refer to as the \emph{generalized basis pursuit in the continuum}~\eqref{eq:noiselessproblem}\footnote{A similar terminology, the ``continuous basis pursuit'', is used in a different context in \cite{ekanadham2011recovery, duval2017sparseII}.}. Our terminology is inspired by the (discrete) basis pursuit (BP)~\cite{chen2001atomic}, which is also a constrained problem; as for the \eqref{eq:noisyproblem}, the ``generalization'' is due to the presence of a regularization operator $\Op D^2$ which is absent in the BP. We therefore carry out our theoretical analysis on the more straightforward~\eqref{eq:noiselessproblem} problem, and we attest that these results apply to the~\eqref{eq:noisyproblem} as well, provided that $\V{y}_\lambda$ is known. For this analysis, we use mathematical tools based on duality theory, and we exploit the very specific form of the so-called dual certificate for our regularization operator $\Op{D}^2$. We describe in a systematic way the form of the solution set and identify the set of sparsest solutions. The fact the optimization problems with sparsity-promoting regularization sometimes have multiple solutions is often sidestepped in the literature by identifying specific cases of uniqueness \cite{candes2014towards,Duval2015exact, Fernandez-Granda2016Super}.
%Although this is known (despite being rarely studied) in discrete settings \cite{Tibshirani2013lasso}, to the best of our knowledge, our \emph{complete} description of the solution set is a first in the TV-based inverse-problem literature. 
When it is not, existing works typically provide the form of certain solutions \cite{fisher1975spline, Unser2017splines}, but they do not characterize cases of uniqueness nor do they give a \emph{complete} description of the solution set as we do here. % A notable exception is \cite{boyer2019representer}, whose setting is much more general at the expense of a considerably less complete description of the solution set. 
Concerning our specific problem, it is known that the function that simply connects the points $(x_1, y_{0, 1}), \ldots, (x_M, y_{0, M})$ is always a solution to the~\eqref{eq:noiselessproblem} (see~\cite[Theorem 1]{koenker1994quantile} and~\cite[Proposition 7]{mammen1997locally}). We refer to it as the \emph{canonical solution}. Building on this result, our contributions on the theoretical and algorithmic sides concerning the~\eqref{eq:noisyproblem} are summarized below.
\begin{enumerate}
\item { \textbf{Theory}

Our main theoretical contributions are the following.
\begin{itemize}
    %\item similarly to~\eqref{eq:noiselessproblem}, the~\eqref{eq:noisyproblem} also has a canonical solution that connects the points $(x_m, y_{0, m})$ for $m=1, \hdots, M$, where the vector $\V{y}_0 \in \R^M$ depends on the regularization parameter $\lambda$; 
    \item In Section~\ref{sec:noiseless}, we fully describe the solution set of the~\eqref{eq:noiselessproblem} by specifying the intervals in which all solutions follow the canonical solution, and those in which they do not (Theorem~\ref{theo:sol_set}). This allows us to characterize the cases where the~\eqref{eq:noiselessproblem} admits a unique solution. When they differ, we give a geometrical description of the set in which the graph of all solutions lies in Theorem~\ref{thm:limitdomain}.
    \item When there are multiple solutions, the canonical solution can be made sparser in certain regions, which is the topic of Section~\ref{sec:sparsestsolutions}. More precisely, in Theorem~\ref{thm:sparsest_sol}, we express the minimum achievable sparsity of a solution to the~\eqref{eq:noiselessproblem} as a simple function of $\x \eqdef (x_1, \ldots, x_M)$ and $\V{y}_0$, which we denote by $\sparsity{\x}{\V{y}_0}$. Concerning the solution set, we fully describe the set of sparsest solutions of the~\eqref{eq:noiselessproblem}. In particular, we characterize the cases of uniqueness, and provide a description of the sparsest solutions together with the number of degrees of freedom $n_{\mathrm{free}} (\x, \obs)$, that we characterize and show to be finite. 
    \item In Section~\ref{sec:solutionnoisy}, we extend the results of the first two items to the~\eqref{eq:noisyproblem}.
    %that is fully determined by the input data $\V{y} =(y_1,\cdots , y_M) \in \R^M$, $\lambda$, and the choice of $E$.
    This is a consequence of the aforementioned equivalence between the~\eqref{eq:noisyproblem} and the~\eqref{eq:noiselessproblem} problems, given in Proposition~\ref{prop:penalized_to_constrained}. We also specify the limit value $\lambda_{\text{max}}$, for which any $\lambda \geq \lambda_{\text{max}}$ amounts to linear regression in Proposition~\ref{prop:linear_regression}.
%    , in particular by specifying the cases of uniqueness, 
%    and giving the number of degrees of freedom within this set.
\end{itemize} }
    \item {\textbf{Algorithm}
    
    These theoretical findings warrant our simple and fast algorithm, presented in Section~\ref{sec:algonoisy}, for reaching (one of) the sparsest solution(s) to the~\eqref{eq:noisyproblem}. The algorithm, which is agnostic to uniqueness, is divided in two parts: first, we compute the $\V{y}_\lambda$ vector for the~\eqref{eq:noiselessproblem} problem by solving a standard discrete $\ell_1$-regularized problem. Next, we find a sparsest solution to the~\eqref{eq:noisyproblem} (with sparsity $\sparsity{\x}{\V{y}_\lambda}$) by optimally sparsifying its canonical solution in some prescribed regions that are determined by our theoretical results. This sparsification step is detailed in Algorithm \ref{alg:algo_constrained} and has complexity $\mathcal{O}(M)$.

    This complete algorithm provides a simple and fast way for the user to judiciously choose $\lambda$ by evaluating the data fidelity loss $\sum_{m=1}^M E(f(x_m), y_m)$ \vs the optimal sparsity $\sparsity{\x}{\V{y}_0}$ --- which depends on $\lambda$ --- as a proxy for the universality \vs simplicity trade-off. We illustrate this in our experiments in Section~\ref{sec:experiments}. The value of $\lambda$ may vary between $\lambda \to 0$ (which at the limit amounts to the~\eqref{eq:noiselessproblem} problem) and an upper bound $\lambda = \lambda_{\text{max}}$ mentioned above. Note that existing algorithms that solve the~\eqref{eq:noisyproblem} such as that introduced in \cite{mammen1997locally} are a lot more complex and computationally expensive. Moreover, to the best of our knowledge, no existing algorithm has the guarantee of reaching a sparsest solution of the~\eqref{eq:noiselessproblem} or the~\eqref{eq:noisyproblem}.
    }
\end{enumerate}

%On the practical side, we propose a simple and fast algorithm --- with complexity $\mathcal{O}(M)$ ---  for reaching (one of) the sparsest solution(s) to~\eqref{eq:noiselessproblem}. We then take advantage of the equivalence between~\eqref{eq:noisyproblem} and~\eqref{eq:noiselessproblem} to design an algorithm that computes a sparsest solution to~\eqref{eq:noisyproblem}.

%The crucial feature of our analysis and algorithm is that they apply not only to~\eqref{eq:noiselessproblem} but also to~\eqref{eq:noisyproblem}, which is an extremely interesting problem from a machine learning perspective. This work can thus be seen as an attempt to understand neural networks viewed as the solutions to variational problems. 

%The one-dimensional is now fully understood, which is a important first step. How to extend this framework to higher dimensions is an open and exciting, albeit daunting, question.
%even the seemingly simple fact of formulating a variational problem is far from being evident, let alone solving it.

\subsection{Related Works}

\paragraph{Discrete $\ell_1$ Optimization}
Putting aside for now the regularization operator $\Op{D}^2$, the optimization problems \eqref{eq:noiselessproblem} and~\eqref{eq:noisyproblem} are the continuous-domain counterparts of the basis pursuit~\cite{chen2001atomic} and the LASSO~\cite{tibshirani1996regression}, which were introduced in the late 90's. These problems are the precursors of the type of $\ell_1$-recovery techniques used in compressed sensing~\cite{Donoho2006,Candes2006sparse,eldar2012compressed,Foucart2013mathematical,Unser2016representer}. These approaches provide  solutions with only few nonzero coefficients. They are at the cornerstone of sparse statistical learning~\cite{hastie2015statistical} and sparse signal processing~\cite{rish2014sparse}.%, and have many applications for example in geophysics~\cite{Claerbout1973Robust,Levy1981Reconstruction,Santosa1986Linear}. 
Theoretical recovery guarantees have been proved, see for example~\cite{Donoho1992Superresolution}; however it is worth noting that in their initial formulations, these methods are inherently discrete and therefore adapted to recover finite-dimensional physical quantities.

\paragraph{Reconstruction in Infinite-Dimensional Spaces}
In our context, we aim at learning a continuous-domain function $f : \mathbb{R} \rightarrow \R$ from finite-dimensional data (the values $y_m =f(x_m)$ for $m\in\{1,\ldots,M\}$). It is therefore natural to formulate the optimization task in infinite dimension to perform the reconstruction. 
The problem is then inherently ill-posed: not only is the system undetermined, as it is also the case in compressed sensing, but we have infinitely many degrees of freedom with finitely many constraints for the reconstruction. Kernel methods based on quadratic regularization are an elegant way of removing this ill-posedness~\cite{Scholkopf2001generalized}, with the effect of restricting the approximation to a finite-dimensional subset of a Hilbert space~\cite{Wahba1990spline,Berlinet2011reproducing,badoual2018periodic}. The challenge is then to choose this Hilbert space adequately. These approaches are fruitful, but they still ultimately revert to the finite-dimensional setting. Taking inspiration from $\ell_1$-based methods for sparse vectors, new approaches have been proposed that go beyond the Hilbert space setting, such as~\cite{Adcock2015generalized,adcock2017breaking,bhandari2018sampling,Bodmann2018compressed}.

\paragraph{Reconstruction in Measure Spaces}
A fertile continuous-domain problem to which discrete $\ell_1$ methods were recently adapted is sparse spikes deconvolution~\cite{de2012exact,bhaskar2013atomic,candes2014towards,Bredies2013inverse,Duval2015exact}. The aim is to recover sums of Dirac masses (point sources signals) over a continuous domain by extending the $\ell_1$ regularization to a gridless setup thanks to the total variation norm $\mnorm{\cdot}$, which is defined over the space of Radon measures $\Spc{M}(\R)$. The underlying optimization problems, either formulated in a constrained form in the noiseless case \cite{candes2014towards} or in a penalized form known as the BLASSO~\cite{de2012exact} in the presence of noise,
are thus solved over a nonreflexive Banach space. The role of the total variation norm in variational methods has a rich history~\cite{Zuhovickii1962,krein1977markov} (see \cite[Section 1]{boyer2019representer} for additional references). From a theoretical standpoint, many reconstruction guarantees are proved, such as exact recovery of discrete measures (sums of Dirac masses) in the noiseless case~\cite{candes2014towards,Fernandez-Granda2016Super}, robustness to noise~\cite{Bredies2013inverse,Candes2013Super,Azais2015Spike,Bhaskar2014minimax}, support recovery~\cite{duval2017sparseII,Duval2015exact,duval2017sparseI,Poon2019Support} and super-resolution for positive discrete measures~\cite{de2012exact,Denoyelle2017support,Poon2019Multidimensional,Schiebinger2018Superresolution,Duval2019characterization,garcia2020approximate,chi2020harnessing}. 

From a numerical standpoint, there exist several different strategies to solve these problems. A first one is based on spatial discretization which leads back to the LASSO and algorithms such as FISTA~\cite{Beck2009fast}. There are also greedy algorithms such as continuous-domain Orthogonal Matching Pursuit (OMP)~\cite{Elvira2019Omp}. In special setups (typically Fourier measurements), it is possible to reformulate the optimization problems as semidefinite programs~\cite{candes2014towards,de2016exact,Catala2017low}. Finally, recent developments based on the Frank-Wolfe (FW) algorithm~\cite{Frank1956algorithm} solve the BLASSO directly over the space of Radon measures~\cite{Bredies2013inverse}. These FW-based methods improve on the traditional FW algorithm due to the possibility of moving the spikes in the continuous domain to further decrease the objective function~\cite{boyd2017alternating,denoyelle2019sliding,courbot2019sparse,flinth2019linear}.

\paragraph{From Dirac Masses Recovery to Spline Reconstruction}
More generally, Dirac masses recovery is part of a trend that promotes continuous-domain formalisms for signal reconstruction. By adding a differential operator to the total variation regularization, one allows for more diverse reconstructions than the recovery of sums of Dirac impulses, while keeping the sparsity-promoting effect of the total variation norm.
Even predating the era of ReLU networks, the~\eqref{eq:noisyproblem} and, to a greater degree, the~\eqref{eq:noiselessproblem}---or variations thereof---have been of keen interest to the signal processing and statistics communities.
Adding a differential operator leads to spline reconstructions, a result that can be traced back to~\cite{fisher1975spline,boor1976best} in the 70's. In \cite{pinkus1988smoothest}, Pinkus proved that the canonical solution---that simply connects the data points---is the unique solution to the~\eqref{eq:noiselessproblem} in some special cases, a result that we recover in our analysis. Later, Koenker \etal ~\cite[Theorem 1]{koenker1994quantile} and Mammen and Van de Geer~\cite[Proposition 7]{mammen1997locally} proved that the canonical solution is indeed a solution to the~\eqref{eq:noiselessproblem}. These works also propose algorithms to solve the~\eqref{eq:noisyproblem} for any value of $\lambda$. However, contrary to this paper, none of the aforementioned works describe the full solution set of the~\eqref{eq:noiselessproblem}, nor identifies its sparsest solutions. There has been a promising new surge of very recent works on related problems, both on the theoretical and the algorithmic sides~\cite{unser2018representer,boyer2019representer,duval2019epigraphical,flinth2019exact,bredies2019sparsity,simeoni2020functional,simeoni2021functional,debarre2019hybrid}. Several very general theories, that incorporate the~\eqref{eq:noisyproblem} and the~\eqref{eq:noiselessproblem}, and that deal with optimization in Banach spaces with various differential regularization operators, have also been recently developed~\cite{boyer2019representer,bredies2019sparsity,fageot2020tv,Unser2019native}.

\paragraph{ReLU Networks, Piecewise-Linear Splines, and the~\eqref{eq:noisyproblem}}
%Perhaps the most basic and oldest form of supervised learning is linear regression: in dimension one ($d=1$), $f$ is parametrized as a linear function by $f(x) = \theta_1 + \theta_2 x$. Although this model certainly follows Occam's razor, it has weak universality properties: many functions are far from being linear.
A modern approach to supervised learning is neural networks, which in recent years have become the gold standard for an impressive number of applications~\cite{goodfellow2016deep}. Many recent papers have highlighted the property that today's state-of-the-art convolutional neural networks (CNNs) with rectified linear unit (ReLU) activations specify an input-output relation $f: \R^d \to \R$, where $d$ is the number of dimensions, that is continuous and piecewise-linear (CPWL) \cite{pascanu2014number, montufar2014number, balestriero2018mad}. This result stems from the fact that the ReLU nonlinearity is itself a CPWL function, as well as, for instance, the widespread max-pooling operation. In fact, there are indications that using more general piecewise-linear splines as activation functions could be more effective than restricting to the ReLU or leaky ReLU \cite{agostinelli2015learning, unser2018representer,aziznejad2020deep}. In the one-dimensional case $d=1$, it follows that the learned function of a ReLU network is a piecewise-linear spline~\cite{daubechies2019nonlinear}, just like the solutions to the~\eqref{eq:noisyproblem} given by \eqref{eq:formsolution}. The trade-off between universality and Occam's razor is then determined by the network size and architecture.
%An illustration is given in Figure \ref{fig:regression_vs_NN}, which compares linear regression with a ReLU network-like learned function. It is clear that the input-output relation in the training data is too complex for linear regression to succeed, whereas a piecewise-linear spline with few knots is able to model this relation adequately.
Many recent papers in the literature have investigated this connection between ReLU networks and piecewise-linear splines \cite{poggio2015notes, boelcskei2019optimal}, including universality properties~\cite{daubechies2019nonlinear, yarotsky2017error, petersen2018optimal}.
We also mention \cite{gribonval2019approximation}, which considers more general spline activation functions.

Moreover, several works have specifically underscored the relevance of the~\eqref{eq:noisyproblem} --- or related problems~\cite{dios2020sparsity} --- in machine learning by showing that it is equivalent to the training of a one-dimensional ReLU network with standard weight decay \cite{savarese2019how, parhi2020role}. Therefore, although the current trend of overparametrizing neural networks is somewhat antagonistic to our paradigm of sparsity, our full description of the solution set of the~\eqref{eq:noisyproblem} (including its non-sparse solutions) could be relevant to the neural network community. Others recent works have designed multidimensional ($d>1$) equivalents of the regularization term $\Vert\D^2 f \Vert_{\Spc{M}}$ and derive similar connections to neural networks \cite{parhi2021banach, ongie2020function}.

%% file: sections/sec-maths-prelim.tex
\section{Mathematical Preliminaries}
\label{sec:math_bckgrnd}
    The task of recovering a continuous-domain function from finitely many samples is obviously ill-posed; this issue is commonly addressed by adding a regularization term. As a regularization norm, we consider  $\Vert \cdot \Vert_\Spc{M}$, which is the continuous-domain counterpart of the $\ell_1$-norm, and is known to promote sparse solutions~\cite{Foucart2013mathematical}.
    Some of the results of this section (in Sections~\ref{sec:BVspace} and~\ref{sec:BVRT}) are not new, as they can be seen as a special case of the general framework developed in previous works~\cite{Unser2017splines,gupta2018continuous,Unser2019native} to the case of the second-derivative operator $\mathrm{L} = \mathrm{D}^2$. 
    Nevertheless, we provide a self-contained treatment, for the benefit of readers who are unfamiliar with the general theory.

% \texorpdfstring{tex}{pdfbookmark} removes the warning about the maths symbol in the title that will not show up in in the PDF bookmarks.
\subsection[\texorpdfstring{tex}{pdfbookmark}]{The Measure Space $\Radon$}

    We denote by $\Radon$, the space of bounded Radon measures on $\RR$. It is a nonreflexive Banach space and is defined as the topological dual of the space $\Co$ of continuous functions that vanish at $\pm\infty$ endowed with the supremum norm $\normi{\cdot}$. 
    The duality product between a measure $w \in \Radon$ and a function $f \in \Co$ is denoted by $\langle w, f \rangle \eqdef \int_{\R} f \mathrm{d} w$.
    The norm on $\Radon$ is called the total-variation norm and is given by
    \begin{equation}
        \label{eq:Mnorm}
        \forall w\in\Radon, \quad \mnorm{w} \eqdef \sup_{f \in \Co, \ \lVert f \rVert_{\infty} \leq 1} \langle w, f \rangle. 
    \end{equation}
    Moreover, we have the continuous embeddings 
    \begin{equation}
    \Sch \subseteq \Radon \subseteq \Schp,    
    \end{equation} 
    where $\Sch$ is the Schwartz space of smooth and rapidly decaying functions and $\Schp$ is its topological dual, the space of tempered distributions~\cite{Schwartz1966distributions}.
    We observe that we can replace $\Co$ by $\Sch$ in~\eqref{eq:Mnorm}, by invoking the denseness of $\Sch$ in $\Co$, and then characterize the bounded Radon measures among $\Schp$ as
    \begin{equation}
        \Radon = \{ w \in \Schp: \sup_{f \in \Sch, \ \lVert f \rVert_{\infty} \leq 1} \langle w, f \rangle < \infty \}.
    \end{equation}

%%
% \texorpdfstring{tex}{pdfbookmark} removes the warning about the maths symbol in the title that will not show up in in the PDF bookmarks.
\subsection[\texorpdfstring{tex}{pdfbookmark}]{The Native Space $\BV$}
\label{sec:BVspace}
    
   Motivated by the form of the regularization in the~\eqref{eq:noiselessproblem} and the~\eqref{eq:noisyproblem}, we introduce the space over which we shall optimize both problems. It is defined as
    \begin{equation}
        \label{eq:BV2}
        \BV \eqdef \{ f \in \mathcal{S}'(\R): \D^2 f \in \Radon\},
    \end{equation}
    with $\D^2 : \Schp \rightarrow \Schp$ the second-derivative operator.
    The space $\BV$ has been considered and studied in~\cite[Section 2.2]{unser2018representer}. It is the second-order generalization of the well-known space of functions with bounded variation. 
    For the sake of completeness, a detailed presentation of the mathematical properties of $\BV$ is provided in~\ref{app:1}.
    For now, it is important to remember that $\BV$ is a Banach space equipped with the norm
         \begin{equation} \label{eq:normBVfirst}
        \lVert f \rVert_{\mathrm{BV}^{(2)}} \eqdef \lVert \mathrm{D}^2 f \rVert_{\Radon} + \sqrt{f(0)^2 + (f(1) - f(0))^2}.
    \end{equation}
    Moreover, any function $f \in \BV$ is continuous and such that $f(x) = \mathcal{O}(x)$ at infinity (see Proposition~\ref{prop:BV2} in~\ref{app:1}).
    
    For any $w \in \Radon$, we denote by $\Itwo \{ w\}$ the unique function $f \in \BV$ such that $\D^2 f = w$ and $f(0) = f(1) = 0$, 
    according to the last point of Proposition~\ref{prop:BV2}. 
    Then, $\Itwo$ is a continuous operator from $\Radon$ to $\BV$, whose main properties are summarized in Proposition~\ref{prop:Itwo} in~\ref{app:1}. Its effect is to doubly integrate the measure on which it operates\footnote{The notation $\Itwo$ has two justifications. First, it recalls that this operator is a right-inverse of the second derivative $\mathrm{D}^2$. However, the index 0 indicates that $\Itwo$ is \emph{not} a left-inverse, as revealed by  Proposition~\ref{prop:Itwo} in~\ref{app:1}.}.
   Moreover,  any $f \in \BV$ can be uniquely decomposed as 
        \begin{align} \label{eq:fdecompose}
            \forall x \in \R, \quad f(x) = \Itwo \{w\} (x) + \beta_0 + \beta_1 x,
        \end{align}
        where $w\in \Radon$ and $\beta_0,\beta_1 \in \R$ satisfy
        \begin{align}
            w = \D^2 f, \quad \beta_0 = f(0), \qandq \beta_1 = f(1) - f(0) .
        \end{align}
        We call the measure $w$ the \emph{innovation} of $f$. The key elements of $\BV$ we are interested in are piecewise-linear splines, which are defined as follows.
        \begin{definition}[Piecewise-Linear Spline]
        \label{def:splines}
        A piecewise-linear spline is a function $f\in \BV$ whose innovation $w = \D^2 f \in \Radon$ is a weighted sum of Dirac masses $w = \sum_{k=1}^K a_k \delta(\cdot - \tau_k)$, where $K \in \N$ is the number of knots (\ie singularities), called the \emph{sparsity} of the spline and $a_k, \tau_k \in \R$.
        \end{definition}
        
        It follows from Definition~\ref{def:splines} that a piecewise-linear spline $f$ can equivalently be written as
        \begin{align}
        \label{eq:splines_def}
            f(x) = b_0 + b_1 x + \sum_{k=1}^K a_k (x - \tau_k)_+,
        \end{align}
        where $b_0, b_1 \in \R$. Note that this representation is different from that of~\eqref{eq:fdecompose} (in general, $(\beta_0, \beta_1) \neq (b_0, b_1)$); however we favor the representation \eqref{eq:splines_def} for splines due to its simplicity.

% \texorpdfstring{tex}{pdfbookmark} removes the warning about the maths symbol in the title that will not show up in in the PDF bookmarks.
\subsection[\texorpdfstring{tex}{pdfbookmark}]{Representer Theorem for $\BV$}
    \label{sec:BVRT}
    
    The native space $\BV$ allows us to precisely define the optimization problems we are interested in. Indeed, it is the largest space for which the regularization  $\lVert \D^2 f \rVert_{\Spc{M}}$ is well-defined and finite. The following result is a special case of a more general theory, which is now well established.
    %Following the seminal work of Fisher and Jerome~\cite{fisher1975spline}, it has been been proved first in~\cite{Unser2017splines,gupta2018continuous}, and has been recently revisited and/or extended by several authors~\cite{boyer2019representer,flinth2019exact,bredies2019sparsity}. 
    
    \begin{theorem}[Representer Theorem for $\BV$]
    \label{theo:RTweneed}
        Let $\x=(x_1,\ldots,x_M) \in \R^M$ be a collection of distinct $M \geq 2$ ordered sampling locations and $\obs \in \R^M$. 
       We consider the set of solutions
        \begin{equation}\label{eq:noiseless}
      \mathcal{V}_0 \eqdef \underset{\substack{f \in \BV \\ f(x_m) = y_{0,m}, \ m=1, \ldots, M}}{\arg \min} \lVert \Op{D}^2 f \rVert_{\mathcal{M}}. \tag{\gBPC}
    \end{equation}
        Moreover, we  fix $\lambda > 0$ and $\V{y} \in \R^M$, together with a cost function $E : \R \times \R \rightarrow \R^+$ such that $E(\cdot, y)$ is strictly convex, coercive, and differentiable for any $y \in \R$ and $\lambda > 0$.
        We also consider the set of solutions
        \begin{align}\label{eq:noisy}
     \mathcal{V}_\lambda \eqdef \argmin_{f \in \BV}  \sum_{m=1}^M E(f(x_m), y_m) + \lambda \Vert\D^2 f \Vert_{\Spc{M}}. \tag{\gBLASSO}
        \end{align}
    Then, for any $\lambda \geq 0$ (including $0$), $\mathcal{V}_\lambda$ is nonemtpy, convex, and weak-* compact in $\BV$, and is the weak-* closure of the convex hull of its extreme points. The latter are all piecewise-linear splines of the form 
       \begin{equation}
           \label{eq:extremepoints}
           f_{\mathrm{extreme}} (x) = b_0 + b_1 x + \sum_{k=1}^{K} a_k (x - \tau_k)_+,
       \end{equation}
       where $b_0, b_1 \in \R$, the weights $a_k$ are nonzero, the knots locations $\tau_k \in \R$ are distinct, and $K \leq M-2$.
    \end{theorem}
   
    Following the seminal work of Fisher and Jerome~\cite{fisher1975spline}, this result was proved in~\cite[Theorem 2]{Unser2017splines} for $\lambda = 0$ and for a general spline-admissible operator $\mathrm{L}$ in the regularization term $\lVert \mathrm{L} \cdot \rVert_{\Radon}$.
    The case $\lambda > 0$ is proved in \cite[Theorem 4]{gupta2018continuous} for a general cost function $E$, by reducing the analysis to the optimization problem~\eqref{eq:noiseless} (as we shall do in Section~\ref{sec:noisy}). Theorem~\ref{theo:RTweneed} is then a particular case of these two works for the regularization operator $\mathrm{L} = \D^2$, whose null space is generated by $x\mapsto 1$ and $x \mapsto x$, and for sampling measurements. Note that the application of these known theorems requires to prove that the point evaluation $f\mapsto f(x_0)$ is weak-* continuous on $\BV$ for any $x_0 \in \R$, which has been shown in~\cite[Theorem 1]{unser2018representer}. These theorems has been recently revisited and/or extended by several authors~\cite{boyer2019representer,flinth2019exact,bredies2019sparsity}.
  
    Theorem~\ref{theo:RTweneed} is called a ``representer theorem", as initially proposed in~\cite{Unser2017splines}, because it specifies the form of the extreme-point solutions of the optimization problem. It is then possible to reduce the optimization task to functions of the form~\eqref{eq:extremepoints}, which considerably simplifies the analysis~\cite{gupta2018continuous,Debarre2019}.
    Theorem~\ref{theo:RTweneed} is also an existence result. It guarantees that the minimization problem~\eqref{eq:noiselessproblem} admits at least one piecewise-linear solution. In particular, if the solution is unique, then it is a piecewise-linear spline. However, Theorem~\ref{theo:RTweneed} is not informative regarding the knots locations $\tau_k$, which may be distinct from the sampling locations $x_m$.
    
    To the best of our knowledge, very few attempts have been made to characterize the cases where gTV optimization problems admit a unique solution, and to describe the solution set when the solution is not unique. In this paper, we provide complete answers to these questions for the reconstruction of functions via sampling measurements and with $\mathrm{BV}^{(2)}$-type regularization. 

\subsection{Dual Certificates}

This section presents the main tools for the study of the~\eqref{eq:noiseless} problem (with $\x\in\RR^M$ the ordered distinct sampling locations and $\V y_0\in\RR^M$ the measurements), coming from the duality theory, which are at the core of our contributions. Our strategy consists in studying a particular class of continuous functions, called \emph{dual certificates}, which can be used individually to certify that an element $f\in\BV$ is a solution of the optimization problem~\eqref{eq:noiseless}. More interestingly, from the properties of a given dual certificate, it is possible to precisely describe the whole structure of the set of solutions (see Theorem~\ref{theo:sol_set}) and, in particular, to determine whether or not the sparse solution given by Theorem~\ref{theo:RTweneed} is the unique solution of the problem (see Proposition~\ref{prop:uniqueness}).
%For a discussion on this matter involving a general differential operator $\Op L$ (instead of $\Op D^2$) and a general forward operator $\nuf$ (instead of sampling), see~\cite{denoyelle2020}.

Before giving the main results of this section (Propositions~\ref{prop:cns-sol-constrained} and \ref{prop:cns-sol-constrained-fixed-dual-certif}), let us first introduce the definition of a dual pre-certificate.

\begin{definition}[Dual Pre-Certificate]\label{def:dual-certif}
We say that a function $\eta\in\Co$ is a \emph{dual pre-certificate} (for the problem~\eqref{eq:noiseless}) 
if its norm satisfies $\normi{\eta} \leq 1$ and if $\eta$ is of the form
\begin{align}
\eta = \sum_{m=1}^M \dualvarm \green{x_m - \cdot}
\end{align}
for some vector $\dualvar = (c_1 , \ldots, c_M) \in \R^M$ such that $\dotp{\dualvar}{\un} = \dotp{\dualvar}{\x} = 0$ (with $\un\eqdef(1,\ldots,1)\in\RR^M$).
%if and only if the following conditions are satisfied
%\begin{align*}
%&\exists\dualvar\in\RR^M, \ \eta = \sum_{m=1}^M \dualvarm \green{x_m - \cdot} \qwithq \dotp{\dualvar}{\un} = \dotp{\dualvar}{\x} = 0 \qandq \normi{\eta} \leq 1.
%\end{align*}
\end{definition}

A dual pre-certificate is therefore a piecewise-linear spline. 
The conditions $\dotp{\dualvar}{\un} = \dotp{\dualvar}{\x} = 0$ ensure that $\eta$ is compactly supported, and is thus an element of $\Co$ (indeed, we have $\eta(x) = - \langle \V{c}, \V{1} \rangle x + \langle \V{c}, \V{x} \rangle = 0$ for any $x \leq x_1$).
We shall present an explicit construction of such a pre-certificate in Proposition~\ref{prop:canocertif} with the piecewise-linear spline $\etacano$.
A dual certificate is a pre-certificate that satisfies an additional condition (see Proposition~\ref{prop:cns-sol-constrained}) that ensures that the vector $\V c\in\RR^M$ in Definition~\ref{def:dual-certif} is a solution of the dual problem of~\eqref{eq:noiseless}. 
%Only a dual certificate can be used to characterize the solution set of~\eqref{eq:noiseless}. 

%\begin{proposition}\label{prop:existence-dual-certificate}
%There exists at least one dual certificates for the problem~\eqref{eq:noiseless}.
%\end{proposition}

From~\eqref{eq:fdecompose}, we know we can parametrize any $f\in\BV$ with a unique element $(w,(\beta_0, \beta_1)) \in\Radon\times\RR^2$ through the relation
\begin{align}
 \forall x \in \R, \quad f(x) = \Itwo\{w\} (x) + \beta_0 + \beta_1 x.
\end{align}
Dual certificates determine the localization of the support of $w$ when $f$ is a solution of~\eqref{eq:noiseless}. %This property is at the core of our contributions. 
To formulate this property, we need the following definition which introduces the concepts of signed support of a measure (see Section 1.4 of~\cite{Duval2015exact}) and signed saturation set of a pre-certificate (see~\cite[Definition 3]{Duval2015exact}).

\begin{definition}[Signed Support and Signed Saturation Set]\label{def:supp-sat}
Let $w\in\Radon$ and $\eta\in\Co$ be a dual pre-certificate in the sense of Definition~\ref{def:dual-certif}. We define the \emph{signed support} of $w$ by
\begin{align}
	\ssupp w \eqdef \supp(w_+) \times \{1\} \cup \supp(w_-) \times \{-1\},
\end{align}
where $w_+$ and $w_-$ are positive measures coming from the Jordan decomposition of $w=w_+ - w_-$.
Moreover from the positive and negative saturation sets of $\eta$, defined as
\begin{align}
	\satp \eta \eqdef \{x\in\RR: \eta(x)=1\} \qandq \satn \eta \eqdef \{x\in\RR: \eta(x)=-1\}
\end{align}
 respectively, we define the \emph{signed saturation set} of $\eta$ by
\begin{align}
	\ssat \eta \eqdef \satp{\eta} \times \{1\} \cup \satn{\eta} \times \{-1\}.
\end{align}
\end{definition}
Note that the sets $\ssupp w$, $\satp \eta$, $\satn \eta$, $\ssat \eta$ are all closed.
A dual pre-certificate $\eta$ is a piecewise-linear spline in $\Co$ with norm $\lVert \eta \rVert_\infty \leq 1$. Hence, its signed saturation set is necessarily a union of closed intervals (that can be singletons).

We can now state the first main result of this section, the proof of which can be found in~\ref{app:2}. 
It characterizes the solutions of~\eqref{eq:noiseless} via the signed support of their innovation using the signed saturation set of some dual pre-certificate. 
%It shows that a function of $\BV$ satisfying the interpolation conditions is a solution of~\eqref{eq:noiseless} if and only if the signed support of its innovation is a subset of the signed saturation set of some dual pre-certificate.

\begin{proposition}\label{prop:cns-sol-constrained}
Let $\x\in \R^M$ be the ordered sampling locations, and $\V{y}_0 \in \R^M$. An element $\fopt \in\BV$ is a solution of~\eqref{eq:noiseless} if and only if $\fopt$ satisfies the interpolation conditions $\fopt(x_m)=y_{0,m}$ for all $m\in\{1,\ldots,M\}$ and one can find a dual pre-certificate $\eta$ (Definition~\ref{def:dual-certif}) such that
\begin{align}
	\mnorm{w} = \dotp{w}{\eta}, \label{eq:duality-interpolation-general}
\end{align}
where $w\eqdef \mathrm{D}^2  \{\fopt\}$ is the innovation of $\fopt$.
The condition \eqref{eq:duality-interpolation-general} 
is moreover equivalent to the inclusion
%can be equivalently reformulated as%expanded as
%\begin{align}
%	\mnorm{w_{|\satp \eta}} = \dotp{w_{|\satp \eta}}{1}, \quad \mnorm{w_{|\satn \eta}} = \dotp{w_{|\satn \eta}}{-1} \qandq w_{|\satp{\eta}^c\cap\satn{\eta}^c} = 0. \label{eq:duality-interpolation-expanded}
%\end{align}
%\begin{align}
%	w_{|\satp \eta}\geq 0, \quad w_{|\satn \eta}\leq 0 \qandq w_{|\satp{\eta}^c\cap\satn{\eta}^c} = 0. \label{eq:duality-interpolation-expanded}
%\end{align}
\begin{align}\label{eq:duality-interpolation-expanded}
	\ssupp {w} \subset \ssat \eta.
\end{align}
The dual pre-certificate $\eta$ is then called a \emph{dual certificate} (for problem~\eqref{eq:noiseless}).
\end{proposition}

\begin{remark}
When $\fopt\in\BV$ is a piecewise-linear spline, \ie $\fopt(x)= \sum_{k=1}^K a_k\green{x -\tau_k}+b_0+b_1 x$ for all $x\in \R$ (see~\eqref{eq:splines_def}), the condition~\eqref{eq:duality-interpolation-expanded} is equivalent to the following interpolation requirements on the dual pre-certificate $\eta$
\begin{align}
	\forall k\in\{1,\ldots,K\}, \ \eta(\tau_k) = \sign(a_k).% \qandq \normi{\eta}\leq1.\label{eq:subdiff-sparse}
\end{align}
\end{remark}

From Proposition~\ref{prop:cns-sol-constrained}, a dual certificate $\eta$ is thus a dual pre-certificate that certifies that a given $\fopt\in\BV$ is a solution of~\eqref{eq:noiseless}, \ie $\fopt$ satisfies $\fopt(x_m)=y_{0,m}$ for all $m\in\{1,\ldots,M\}$ and $\ssupp {\Op D^2 \fopt} \subset \ssat \eta$ (or equivalently $\mnorm{\Op D^2 \fopt} = \dotp{\Op D^2 \fopt}{\eta}$). Once we know that some $\eta$ is a dual certificate, it can be used to check whether \emph{any} $f\in\BV$ is a solution of~\eqref{eq:noiseless}. In other words, contrary to what is seemingly implied in Proposition~\ref{prop:cns-sol-constrained}, there is no need to find a new dual pre-certificate for each candidate solution $f$. This is formulated in the following proposition, the proof of which can be found in~\ref{app:2bis}.

\begin{proposition}\label{prop:cns-sol-constrained-fixed-dual-certif}
Let $\x\in \R^M$ be the ordered sampling locations, $\V{y}_0 \in \R^M$, and let $\eta\in\Co$ be dual certificate as defined in Proposition~\ref{prop:cns-sol-constrained} for the problem~\eqref{eq:noiseless}. Then, an element $\fopt \in\BV$ is a solution of~\eqref{eq:noiseless} if and only if $\fopt$ satisfies the interpolation conditions $\fopt(x_m)=y_{0,m}$ for all $m\in\{1,\ldots,M\}$ and 
\begin{align}\label{eq:duality-interpolation-expanded-fixed-dual-certif}
	\ssupp {w} \subset \ssat \eta.
\end{align}
or equivalently $\mnorm{w} = \dotp{w}{\eta}$, where $w\eqdef\Op D^2 \fopt$ is the innovation of $\fopt$.
\end{proposition}

To end this section, let us illustrate how the concept of dual certificates can be used to describe the solution set of~\eqref{eq:noiseless}. Suppose that we know that some $\eta$ is a dual certificate (we prove in Proposition~\ref{prop:optimality_cano} that this is the case of the dual pre-certificate $\etacano$ introduced in Proposition~\ref{prop:canocertif}), then the condition $\ssupp w \subset \ssat \eta$ of Proposition~\ref{prop:cns-sol-constrained-fixed-dual-certif} enforces strong constraints on any candidate solution of~\eqref{eq:noiseless}. This is all the more true when $\ssat \eta$ is a discrete set, which we consider in the next definition and proposition.

\begin{definition}[Nondegeneracy]\label{def:non-degen}
Let $\x\in \R^M$ be the ordered sampling locations, $\V{y}_0 \in \R^M$ and let $\eta \in \Co$ be any dual certificate as defined in Proposition~\ref{prop:cns-sol-constrained}. We say that $\eta$ is \emph{nondegenerate} if its signed saturation set $\ssat \eta$ defined in Definition~\ref{def:supp-sat} is a discrete set. Otherwise, we say that it is \emph{degenerate}.
\end{definition}

\begin{proposition}[General Uniqueness Result for~\eqref{eq:noiseless}]\label{prop:general-uniqueness}
Let $\x\in \R^M$ be the ordered sampling locations and $\V{y}_0 \in \R^M$. If there exists a nondegenerate dual certificate in the sense of Definition~\ref{def:dual-certif}, then the optimization problem~\eqref{eq:noiseless} has a unique solution, which is a piecewise-linear spline in the sense of Definition~\ref{def:splines} with $K \leq M-2$ knots $\tau_k$ that form a subset of the sampling points $\{x_2,\ldots,x_{M-1}\}$.
%the sparse solution given by Theorem~\ref{theo:RTweneed} is the unique solution of~\eqref{eq:noiseless}.
\end{proposition}
The proof of Proposition~\ref{prop:general-uniqueness} is given in~\ref{app:3}.

%% file: sections/sec-gBPC.tex
% !TEX root = ../main.tex
%
\section{The Solutions of the~\eqref{eq:noiselessclean}} \label{sec:noiseless}
    
    In this section, we consider the optimization problem~\eqref{eq:noiselessclean} where the $x_m$ for $m\in\{1,\ldots,M\}$ are distinct and ordered sampling locations and $\obs \in \R^M$ is a fixed measurement vector.
    This setting is especially relevant when the measurements $y_{0,m}$ are exactly the values of the input signal at locations $x_m$ (noiseless case).
    The solution set is
    \begin{equation}\label{eq:noiselessclean}
      \mathcal{V}_0  \eqdef  \underset{\substack{f \in \BV \\ f(x_m) = y_{0,m}, \ m\in\{1,\ldots,M\}}}{\arg \min} \lVert \Op{D}^2 f \rVert_{\mathcal{M}}, \tag{\gBPC}
    \end{equation}
    and is known to admit at least one piecewise-linear solution due to Theorem~\ref{theo:RTweneed}.

    \subsection{Canonical Solution and Canonical Dual Certificate} \label{sec:canonicalstuff}

Thereafter, we identify the complete set of solutions~\eqref{eq:noiselessclean}. This allows us to fully determine in which cases this optimization problem admits a unique solution.
Our analysis is based on the construction of a pair $(\fcano, \etacano) \in \BV \times \Spc{C}_0(\R)$ that satisfies Proposition \ref{prop:cns-sol-constrained}, which we call the canonical solution and canonical dual certificate respectively. The former is simply the function that connects the points $\P0{m} = \Point{x_m}{y_{0,m}}$. 

\begin{definition}[Canonical Interpolant]
    \label{def:canonicalsolutiondef}
        Let $\x \in \R^M$ be the ordered sampling locations and $\obs \in \R^M$ with $M \geq 2$. We define $\fcano$ as the unique piecewise-linear spline that interpolates the data points with the minimum number of knots, \ie such that 
        \begin{itemize}
            \item $\fcano (x_m) = y_{0,m}$ for any $m \in \{1, \ldots , M \}$ and
            \item $\fcano$ has at most $M-2$ knots which form a subset of $\{x_m: 2\leq m\leq M-1\}$.
        \end{itemize}
        We refer to $\fcano$ as the \emph{canonical interpolant}.
\end{definition}

 %   From Definition \ref{def:canonicalsolutiondef}, one can deduce the explicit expression of $\fcano$.
    The existence and uniqueness of $\fcano$ in Definition \ref{def:canonicalsolutiondef} simply follows from the number of degrees of freedom of a piecewise-linear spline whose knots are known. The canonical interpolant is of the form
    \begin{equation}
    \label{eq:canonical_sol}
        \fcano(x) = a_1 x + a_M + \sum_{m=2}^{M-1} a_m (x - x_m)_+
    \end{equation}
    with $\V{a} = (a_1, \ldots , a_M )\in \R^M$. 
    By definition, $\fcano$ is linear on the interval $(x_m, x_{m+1})$ for $m \in \{2, \ldots , M-1\}$. The interpolatory conditions $\fcano(x_m) = y_m$ and $\fcano(x_{m+1}) = y_{m+1}$ then imply that its slope is $s_m = \frac{y_{0,m+1}-y_{0,m}}{x_{m+1} - x_m}$.
    Yet from~\eqref{eq:canonical_sol} we get that $s_m = a_1 + \cdots + a_m$. This implies that $a_1 = s_1$ and that $a_m = s_m - s_{m-1}$ for $m\in\{2, \ldots, M-1\}$. Finally, the equation $\fcano(x_1) = y_{0,1}$ yields $a_M = y_{0,1} - a_1 x_1$. Consequently, the vector $\V{a} \in \R^M$ in~\eqref{eq:canonical_sol} is given by 
\begin{align}
\label{eq:a_coefs}
    \begin{cases}
    a_1 = \frac{y_{0,2}-y_{0,1}}{x_2 - x_1}, \\
    a_m = \frac{y_{0,m+1}-y_{0,m}}{x_{m+1} - x_m} - \frac{y_{0,m}-y_{0,m-1}}{x_m - x_{m-1}}, \quad \forall m\in\{2, \ldots, M-1\}, \\
    a_M = y_{0,1} - \frac{y_{0,2}-y_{0,1}}{x_2 - x_1} x_1.
    \end{cases}
\end{align}

In order to prove that $\fcano$ is always a solution of~\eqref{eq:noiselessclean}, we construct a particular dual pre-certificate $\etacano$.
\begin{proposition}[Canonical Pre-Certificate]
\label{prop:canocertif}
Let $\x\in \R^M$ be the ordered sampling locations and $\obs \in \R^M$. Let $\V{a} \in \R^M$ be the vector defined by~\eqref{eq:a_coefs}.
There exists a unique piecewise-linear spline $\etacano$ given by
\begin{align}
	&\etacano \eqdef \sum_{m=1}^M \dualvarm \green{x_m-\cdot} \qwithq \dualvar=(c_1, \ldots, c_M) \in\RR^M, \label{eq:condition1etacano} \\
	&\dotp{\dualvar}{\un} = \dotp{\dualvar}{\x} = 0,\\
	&\forall m\in\{2,\ldots,M-1\}, \ \etacano(x_m)=\sign(a_m). \label{eq:condition3etacano}
\end{align}
with the convention $\sign(0)=0$. Moreover, since $\etacano(x) = 0$ for $x \leq x_1$ and $x \geq x_M$, we have $\etacano \in \Co$ and $\normi{\etacano} = 1$. Hence, $\etacano$ is a dual pre-certificate in the sense of Definition~\ref{def:dual-certif}.
\end{proposition}

\begin{figure}[t]
\subfloat[Canonical solution $\fcano$]{\includegraphics[width=0.5\linewidth]{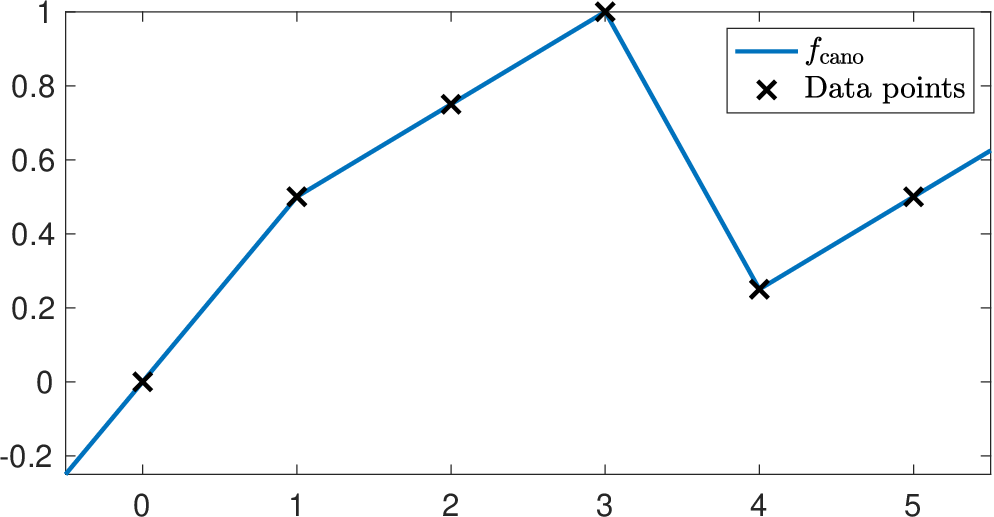}}
\subfloat[Canonical dual certificate $\etacano$]{\includegraphics[width=0.5\linewidth]{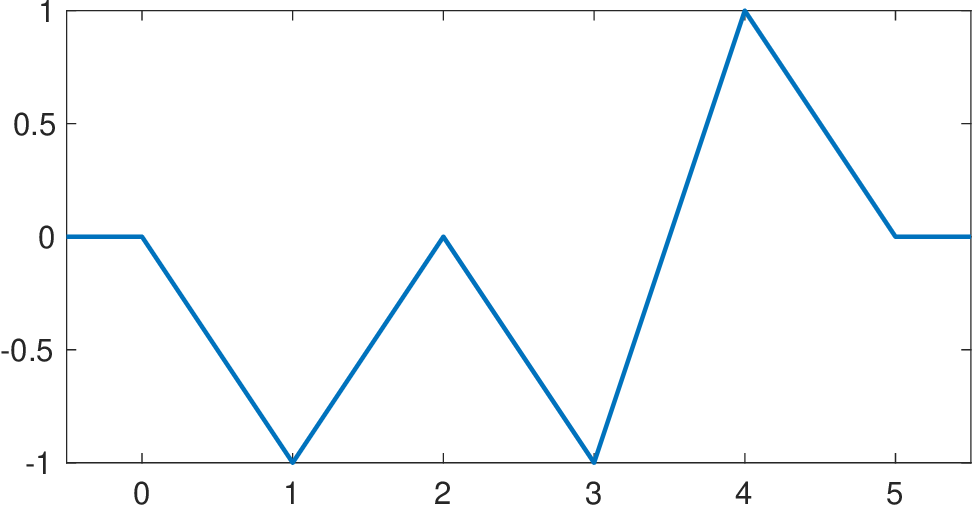}}
\caption{Example of a canonical solution and canonical dual certificate for $M=6$ with $x_m = m-1$. We have $a_2<0$, $a_3=0$, $a_4<0$, and $a_5>0$, where the $a_m$ are defined in \eqref{eq:a_coefs}. }
\label{fig:cano_example}
\end{figure}

\begin{proof}
The existence and uniqueness of such a spline follows the same argument as for $\fcano$,
%follows from Proposition~\ref{prop:canonicalsolutiondef} 
applied to the data points $(x_1-1, 0)$, $(x_1, 0)$, $(x_m, \sign(a_m))$ for $m\in\{2, \ldots, M-1\}$, $(x_M, 0)$ and $(x_M+1, 0)$. Note that the points $(x_1-1, 0)$ and $(x_M+1, 0)$ at the boundaries add two additional interpolation constraints to~\eqref{eq:condition3etacano}. Moreover, they imply that $\etacano$ does not have a linear term and is thus of the form~\eqref{eq:condition1etacano}.

Next, we notice that for $x\leq x_1$, we have $\etacano(x) = -\langle \V{c}, \x \rangle x + \langle \V{c}, \V{1} \rangle = 0$, due to  $\langle \V{c}, \x \rangle = \langle \V{c}, \V{1} \rangle = 0$. For $x\geq x_M$, $(x_m - x)_+ = 0$ for every $m \in\{1,\ldots , M\}$, hence $\etacano(x) = 0$. Then, as a piecewise-linear spline with compact support, $\etacano$ is of course in $\Co$. Being compactly supported, it is also clear that $\etacano$ attains its maximum and minimum values at its knots. In particular, $\lVert \etacano \rVert_{\infty} = \max_{m \in\{1,\ldots , M\}} \lvert \etacano(x_m) \rvert = 1$. 

\end{proof}
We now prove that the pair $(\fcano, \etacano) \in \BV \times \Spc{C}_0(\R)$ satisfies Proposition \ref{prop:cns-sol-constrained}. Although the fact that $\fcano$ is a solution to~\eqref{eq:noiselessclean} is known \cite{koenker1994quantile, mammen1997locally} and is significant in its own right, the key element of this result is the construction of the dual certificate $\etacano$. The latter will be essential to fully describe the solution set $\Spc{V}_0$.

\begin{proposition}
\label{prop:optimality_cano}
    Let $\x\in \R^M$ be the ordered sampling locations and $\obs \in \R^M$.
    The canonical interpolant $\fcano$ defined in Definition~\ref{def:canonicalsolutiondef} is a solution of~\eqref{eq:noiselessclean} and $\etacano$, defined in Proposition~\ref{prop:canocertif}, is a dual certificate as defined in Proposition \ref{prop:cns-sol-constrained}.
\end{proposition}
\begin{proof}
By construction, the interpolation conditions $\fcano(x_m) = y_{0,m}$ for all $m\in\{1,\ldots,M\}$ are satisfied. Moreover thanks to Proposition~\ref{prop:canocertif}, $\etacano$ is a dual pre-certificate. By Proposition~\ref{prop:cns-sol-constrained}, it remains to prove that
\begin{align}\label{eq:oc-etacano}
    \ssupp{\Op D^2 \fcano}\subset\ssat\etacano,
\end{align}
from which we deduce both that $\fcano$ is a solution of~\eqref{eq:noiselessclean} and that $\etacano$ is a dual certificate. Since, again by construction, $\etacano(x_m)=\sign(a_m)$ for all $m\in\{1,\ldots,M\}$ and $\Op D^2 \fcano=\sum_{m=2}^{M-1} a_m \dirac{x_m}$, this proves~\eqref{eq:oc-etacano}.
\end{proof}

Due to Proposition~\ref{prop:optimality_cano}, we call $\fcano$ the \emph{canonical solution} and $\etacano$ the \emph{canonical dual certificate} of the optimization problem~\eqref{eq:noiselessclean}. We show an example of such functions for given data points $(x_m, y_{0,m})_{m\in\{1,\ldots,6\}}$ in Figure~\ref{fig:cano_example}. Notice that the points $\P0{2}$, $\P0{3}$, and $\P0{4}$ are aligned, which implies that $a_3=0$ (defined in~\eqref{eq:a_coefs}).

\subsection{Characterization of the Solution Set} \label{sec:solutionnoiseless}

%In the previous section, we identified the canonical solution of~\eqref{eq:noiselessclean}. However, in general, the solution of the optimization problem is not unique.
Although identifying a solution $\fcano$ to~\eqref{eq:noiselessclean} is an important first step, this solution is not unique in general.
We characterize the case of uniqueness in Proposition~\ref{prop:uniqueness}, and then provide a complete description of the solution set when the solution is not unique in Theorem~\ref{theo:RTweneed}.
We shall see that the canonical dual certificate $\etacano$ plays an essential role regarding these issues.

\begin{proposition}[Uniqueness Result for~\eqref{eq:noiselessclean}]
\label{prop:uniqueness}
Let $\x\in \R^M$ be the ordered sampling locations and $\obs \in \R^M$.
    Then, the following conditions are equivalent.
    \begin{enumerate}
        \item~\eqref{eq:noiselessclean} has a unique solution.
        \item The canonical dual certificate $\etacano$ (defined in Proposition \ref{prop:canocertif}) is nondegenerate (see Definition \ref{def:non-degen}).
        \item For all $m \in\{2,\ldots, M-2\}$, $a_m a_{m+1} \leq 0$, where $\V{a}\in\RR^M$ is given by~\eqref{eq:a_coefs}. 
     \end{enumerate}
\end{proposition}
\begin{proof}
The equivalence $2. \Leftrightarrow 3.$ comes from the fact that $\etacano$ is nondegenerate if and only if it never saturates at $1$ or $-1$ between two consecutive knots. This is equivalent to item 3 because for all $m \in\{2,\ldots, M-1\}$, $\etacano(x_{m}) = \sign(a_{m})$.

The implication $2. \Rightarrow 1.$ is given by Proposition~\ref{prop:general-uniqueness}. We now prove the contraposition of the reverse implication $1. \Rightarrow 2$. We thus assume that $\etacano$ is degenerate, and wish to prove that~\eqref{eq:noiselessclean} has multiple solutions. Using item 3., there exists an index $m \in \{2, \ldots, M-2 \}$ such that $a_{m} a_{m+1} > 0$. We now invoke the following lemma (illustrated in Figure~\ref{fig:mountain}) that plays an important role throughout the paper.
\begin{lemma}
\label{lem:mountain}
Let $\x\in \R^M$ be the ordered sampling locations, and $\V{y}_0 \in \R^M$ with $M \geq 4$. Let $m \in \{ 2, \ldots, M-2 \}$ be an index such that $a_m a_{m+1} > 0$, where $\V{a}\in\RR^M$ is defined as in~\eqref{eq:a_coefs}. Then, the lines $(\P0{m-1}, \P0{m})$ and $(\P0{m+1}, \P0{m+2})$ are intersecting at a point $\mathrm{\widetilde{P}} = \Point{\tilde{\tau}}{\tilde{y}}$ such that $x_m < \tilde{\tau} < x_{m+1}$. Moreover, the piecewise-linear spline $\fopt$ defined by
\begin{align}
\label{eq:mountain}
\fopt(x) \eqdef 
\begin{cases}
\frac{y_{0,m} - y_{0,m-1}}{x_m - x_{m-1}} (x- x_{m-1})+ y_{0,m-1}, & \text{for }x_m < x \leq \tilde{\tau} \\
\frac{y_{0,m+2} - y_{0,m+1}}{x_{m+2} - x_{m+1}} (x- x_{m+1})+ y_{0,m+1}, & \text{for }\tilde{\tau} < x < x_{m+1} \\
\fcano(x) & \text{for }x\not\in (x_m, x_{m+1}),
\end{cases}
\end{align}
which has no knots at $x_m$ or $x_{m+1}$, is a solution of~\eqref{eq:noiselessclean}.
\end{lemma}
\begin{proof}
Let $I_0 = \{2, \ldots, M-1 \} \setminus \{ m, m + 1 \}$.
We then define
\begin{align}
\fopt(x) \eqdef a_1 x + a_M + \sum_{m' \in I_0} a_{m'} (x - x_{m'})_+ + \tilde{a} (x - \tilde{\tau})_+,
\end{align}
where $\tilde{a} = a_{m} + a_{m + 1}$ and $\tilde{\tau} = \frac{a_{m}x_{m} + a_{m+1}x_{m+1}}{\tilde{a}}$. By definition, $\tilde{\tau}$ is a barycenter of $x_{m}$ and $x_{m+1}$ with weights $\frac{a_{m}}{\tilde{a}}$ and $\frac{a_{m+1}}{\tilde{a}}$. Yet $a_{m}$ and $a_{m+1}$ have the same (nonzero) signs, which implies that these weights are in the interval $(0,1)$ and thus that $\tilde{\tau} \in (x_{m}, x_{m+1})$. Yet $\fopt$ has no knots at $x_m$ and $x_{m+1}$, so it must follow the line $(\P0{m-1}, \P0{m})$ in the interval $[x_m, \tilde{\tau}]$, and the line $(\P0{m+1}, \P0{m+2})$ in the interval $[\tilde{\tau}, x_{m+1}]$, which conforms with the first two first lines in~\eqref{eq:mountain}. Due to the continuity of $\fopt$, these lines are therefore intersecting at the point $\mathrm{\widetilde{P}} = \Point{\tilde{\tau}}{\tilde{y}} = \Point{\tilde{\tau}}{\fopt(\tilde{\tau})}$.

Next, for $x \leq x_{m}$, we have $a_{m} (x - x_{m})_+ + a_{m+1}(x - x_{m+1} )_+ = \tilde{a} (x - \tilde{\tau})_+ = 0$. Similarly, for $x \geq x_{m+1}$, we have $a_{m} (x - x_{m})_+ + a_{m+1}(x - x_{m+1} )_+ = \tilde{a} (x - \tilde{\tau})_+ = \tilde{a} (x - \tilde{\tau})$ since $x\geq \tilde{\tau}$. Therefore, for any $x \not\in (x_{m}, x_{m+1})$, we have $\fcano(x) = \fopt(x)$, which conforms with the third line in~\eqref{eq:mountain}. This also implies that $\fopt(x_m) = \fcano(x_m) = y_{0,m}$ for all $m\in\{1,\ldots,M\}$. Moreover, we have $\Vert \Op{D}^2 \fcano \Vert_\Spc{M} = \sum_{m=2}^{M-1} a_m = \sum_{m\in I_0} a_m + \tilde{a} = \Vert \Op{D}^2 \fopt \Vert_\Spc{M}$. Therefore, $\fopt$ has the same measurements and regularization cost as $\fcano$, which implies that it is also a solution of~\eqref{eq:noiselessclean}.
\end{proof}

Since $\fopt$ defined in Lemma~\ref{lem:mountain} is a solution to the \eqref{eq:noiselessclean} such that $\fopt \neq \fcano$, the \eqref{eq:noiselessclean} has multiple solutions, which concludes the proof.
\end{proof}

To the best of our knowledge, Proposition~\ref{prop:uniqueness} is a new result. A similar uniqueness result is presented in \cite[Theorem 4.2]{pinkus1988smoothest}, but with more restrictive conditions than item 3. It follows from Proposition~\ref{prop:uniqueness} that when $M =3$, the solution of the~\eqref{eq:noiselessproblem} is always unique because the certificate is always nondegenerate, and is given by $\fcano$. We go much further in Theorem~\ref{theo:sol_set} by providing the full characterization of the solution set when $M\geq 4$.

\begin{theorem}[Characterization of the Solution Set of the~\eqref{eq:noiselessclean}]
\label{theo:sol_set}
    Let $\x\in \R^M$ be the ordered sampling locations and $\obs \in \R^M$ with $M \geq 4$, and let $\fcano$ and $\etacano$ be the functions defined in Definition \ref{def:canonicalsolutiondef} and Proposition \ref{prop:canocertif} respectively.
   A function $\fopt \in \BV$ is a solution of the~\eqref{eq:noiselessclean} if and only if $\fopt(x_m) = y_{0,m}$ for $m\in\{1,\ldots,M\}$, and the following conditions are satisfied for $m\in\{2,\ldots,M-2\}$
    \begin{enumerate}
        \item $\fopt = \fcano$ in $[x_m,x_{m+1}]$ if $|\etacano| < 1$ in $(x_m,x_{m+1})$;
        \item $\fopt$ is convex in $[x_{m-1},x_{m+2}]$ if $\etacano = 1$ in $[x_m,x_{m+1}]$;
        \item $\fopt$ is concave in $[x_{m-1},x_{m+2}]$ if $\etacano = -1$ in $[x_m,x_{m+1}]$; 
        \item $\fopt = \fcano$ in $(-\infty, x_2)$ and $(x_{M-1}, + \infty)$.
       \end{enumerate}
\end{theorem}

%\begin{proofof}
\begin{proof}
Let $\fopt$ be a solution of the~\eqref{eq:noiselessclean}. According to Proposition~\ref{prop:optimality_cano}, $\etacano$ is a dual certificate. According to Proposition \ref{prop:cns-sol-constrained-fixed-dual-certif}, we therefore have that
$	\ssupp {\mathrm{D}^2 \fopt } \subset	\ssat \etacano$, meaning that $\mathrm{D}^2 \fopt  = 0$ on the complement $\ssat \etacano^c$ of $\ssat \etacano$. In particular, we have that $(-\infty, x_2] \subset \ssat \etacano^c$, hence $\fopt$ is linear on this interval. The interpolation constraints $\fopt(x_1) = \fcano (x_1)$ and $\fopt(x_2) = \fcano (x_2)$ then imply that $\fopt = \fcano$ on $(-\infty, x_2] $. The same argument holds for the interval $[x_{M-1}, + \infty)$ and any interval $(x_m, x_{m+1})$ on which $\etacano$ does not saturate.

Assume now that $[x_m,x_{m+1}]\subset \satp \etacano$; that is, $\etacano = 1$ on $[x_m,x_{m+1}]$. 
We use the Jordan decomposition of $\mathrm{D}^2 \fopt = w = w_+ - w_{-}$ where $w_+$ and $w_-$ are positive measures. 
By~\eqref{eq:duality-interpolation-expanded}, we know that $w_{-} = 0$ on  $[x_m,x_{m+1}]$ because its support is included in $\satn{\etacano}$. Hence, on this interval, $\mathrm{D}^2 \fopt = w = w_+$ is a positive measure, implying that $\mathrm{D} \fopt$ is increasing and therefore that $\fopt$ is convex on $[x_m,x_{m+1}]$. Now, if $(x_{m-1},x_{m})\subset \satp{\eta}^c\cap\satn{\eta}^c$ then, as above, $\D^2 \opt f_{|(x_{m-1},x_{m})} = 0$. Otherwise, by continuity of $\etacano$, we have $(x_{m-1},x_{m}) \subset \satp\etacano$ hence $\D^2 \opt f_{|(x_{m-1},x_{m})} \geq 0$. As a result $\fopt$ is convex on $(x_{m-1},x_{m+1}]$. The same argument proves that $\fopt$ is convex on $[x_m, x_{m+2})$, and therefore on the whole interval $(x_{m-1},x_{m+2})$.

Suppose conversely that $\fopt$ satisfies all the conditions of Theorem \ref{theo:sol_set}. Let us prove that it is a solution of the \eqref{eq:noiselessclean}. By Proposition~\ref{prop:cns-sol-constrained-fixed-dual-certif}, we just need to check that $\fopt$ satisfies $\ssupp{\D^2 \fopt}\subset\ssat\etacano$ since by construction, $\fopt(x_m)=y_{0,m}$ . By definition of $\etacano$, we have $\D^2 \fopt = 0$ on $\satp{\etacano}^c\cap\satn{\etacano}^c$ (because $\D^2 \fopt$ is equal to $\fcano$ which is linear on that set). Moreover, $\D^2 \fopt\geq 0$  on $\satp\etacano$ (because by assumption, $\fopt$ is convex on intervals where $\etacano = 1$) and $\D^2 \fopt\leq 0$  on $\satn\etacano$ (because $\fopt$ is concave on intervals where $\etacano = -1$).
This means that $\supp{w_+} \subset \satp{\etacano}$ and $\supp{w_-} \subset \satp{\etacano}$ where $\mathrm{D}^2 \fopt = w_+ - w_{-}$ is again the Jordan decomposition of $\mathrm{D}^2 \fopt$. Finally, as expected, we have that 
\begin{equation}
    \ssupp{\mathrm{D}^2 \fopt} = \supp {w_+} \times \{1\} \cup  \supp{w_-} \times \{-1\} \subset \satp{\etacano} \times \{1\} \cup \satn{\etacano} \times \{-1\} = \ssat \etacano,
\end{equation}
hence $\fopt$ is a solution of the \eqref{eq:noiselessclean}. 
\end{proof}

To illustrate Theorem~\ref{theo:sol_set}, a simple example with $M=4$ data points for which the solution is not unique is given in Figure~\ref{theo:sol_set}. Indeed, the canonical dual certificate saturates at -1 in the interval $[1, 2]$. Therefore, by Theorem~\ref{theo:sol_set}, any function that coincides with $\fcano$ in $\R\setminus [1, 2]$ and that is concave in the interval $[0, 3]$ is a solution. This includes the sparsest solution (with a single knot), as well as non-sparse solutions, \eg with a quadratic regime in $[1, 2]$ as in Figure~\ref{fig:mountain}.

\begin{figure}[t]
\centering
\subfloat[Various solutions]{\includegraphics[width=0.5\linewidth]{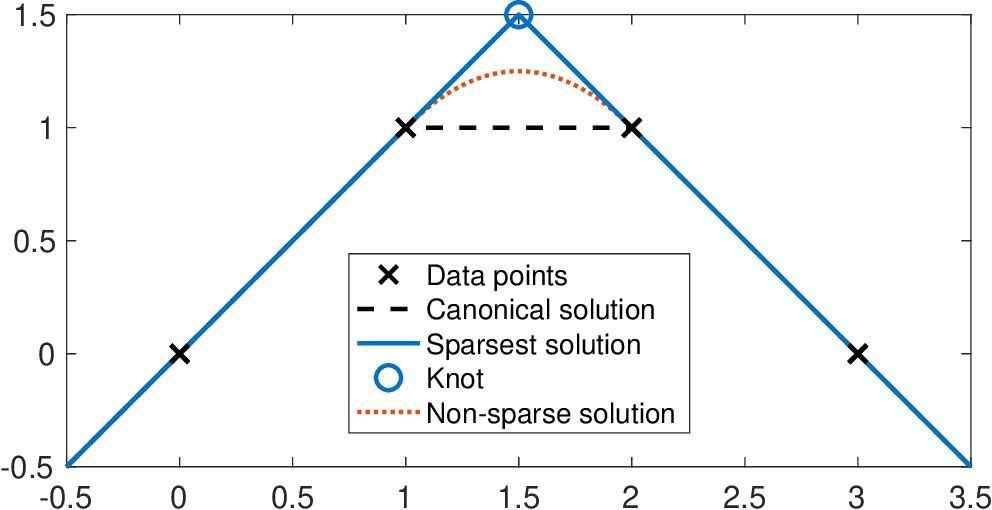}}
\subfloat[Canonical dual certificate]{\includegraphics[width=0.5\linewidth]{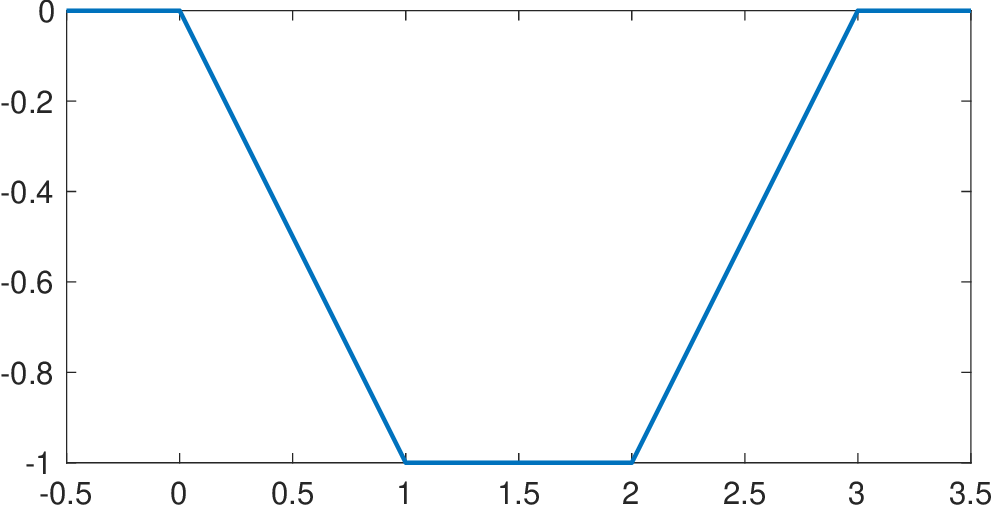}}
\caption{Example with $M=4$ of a non-unique solution ($\etacano$ saturates at -1). An example of a non-sparse solution with a quadratic regime in $[1, 2]$ is given.}
\label{fig:mountain}
\end{figure}

\begin{corollary} \label{coro:uncountable}
    If the~\eqref{eq:noiselessclean} has more than one solution, then it has an uncountable number of solutions.
\end{corollary}

\begin{proof}
If the solution is not unique, then the dual certificate $\etacano$ is degenerate, and therefore saturates over some interval $(x_m,x_{m+1})$. Then, Theorem~\ref{theo:sol_set} characterizes the whole set of solutions, which is clearly uncountably infinite. 
\end{proof}

Corollary~\ref{coro:uncountable} is the continuous counterpart of the well-known fact that the discrete LASSO either admits a unique solution or an uncountable number of solutions~\cite[Lemma 1]{Tibshirani2013lasso}. Even with infinitely many solutions, we are able to delimit the geometric domain that contains the graphs of all solutions by exploiting the local convex/concavity. We recall that $\mathrm{P}_{0,m} = [x_m \ y_{0,m}]^T$ for $m\in\{1,\ldots,M\}$, and that for $\mathrm{A},\mathrm{B}\in \R^2$, we denote by $(\mathrm{A},\mathrm{B})$ the line joining $\mathrm{A}$ and $\mathrm{B}$.
Then, for $M \geq 4$, we consider the set of indices
\begin{equation}
    \label{eq:indicesnoparallel}
    \mathcal{X} \eqdef \mathcal{X}(\x, \obs) \eqdef \left\{m \in \{2, \ldots, M-2\}; \ a_m a_{m+1} > 0 \right\},
\end{equation}
where we recall that $ a_m = \frac{y_{0,m+1}-y_{0,m}}{x_{m+1} - x_m} - \frac{y_{0,m}-y_{0,m-1}}{x_m - x_{m-1}}$ (see~\eqref{eq:a_coefs}). 
The slope condition $a_m a_{m+1}> 0$ in~\eqref{eq:indicesnoparallel} is equivalent to the fact that the lines $(\mathrm{P}_{0,m-1}, \mathrm{P}_{0,m})$ and  $(\mathrm{P}_{0,m+1}, \mathrm{P}_{0,m+2})$ are not parallel (otherwise we would have that $a_m = -a_{m+1}$, hence $a_m a_{m+1} \leq 0$) and that their intersection point, that we denote by $\widetilde{\mathrm{P}}_m= [\tilde{\tau}_m  \ \tilde{y}_m]^T$, is such that $x_m \leq \tilde{\tau}_m \leq x_{m+1}$ according to Lemma \ref{lem:mountain}. %This equivalence can be proved using a barycenter-based argument similar to that used in the proof of Proposition \ref{prop:uniqueness}.
We can thus introduce the triangles $\Delta_m$, whose vertices are the points $\mathrm{P}_{0,m}$, $\widetilde{\mathrm{P}}_m$, and $\mathrm{P}_{0,m+1}$. Theorem \ref{thm:limitdomain} makes the link between the graph of any solution $\fopt\in\BV$ of the~\eqref{eq:noiselessclean}, the graph of $\fcano$ and the triangles $\Delta_m$.

\begin{theorem}[Geometric Domain of the Graph of Solutions of the~\eqref{eq:noiselessclean}]
\label{thm:limitdomain}
Let $\x\in \R^M$ be the ordered sampling locations and $\obs \in \R^M$ with $M \geq 4$.
Then, we have
\begin{equation} \label{eq:domain}
    \cup_{\fopt \in \mathcal{V}_0}  \mathcal{G}(\fopt) = \mathcal{G}(\fcano) \cup \left( \cup_{m \in \mathcal{X}}  \Delta_m \right),
\end{equation}
where $\fcano$ is defined in Definition~\ref{def:canonicalsolutiondef}, $\Spc{X}$ is defined in~\eqref{eq:indicesnoparallel}, and the $\Delta_m$ triangles are defined in the above paragraph.
\end{theorem}

\begin{figure}[t]
\centering
\begin{tikzpicture}
    \node[anchor=south west,inner sep=0] (image) at (0,0,0) {\includegraphics[width=0.5\linewidth]{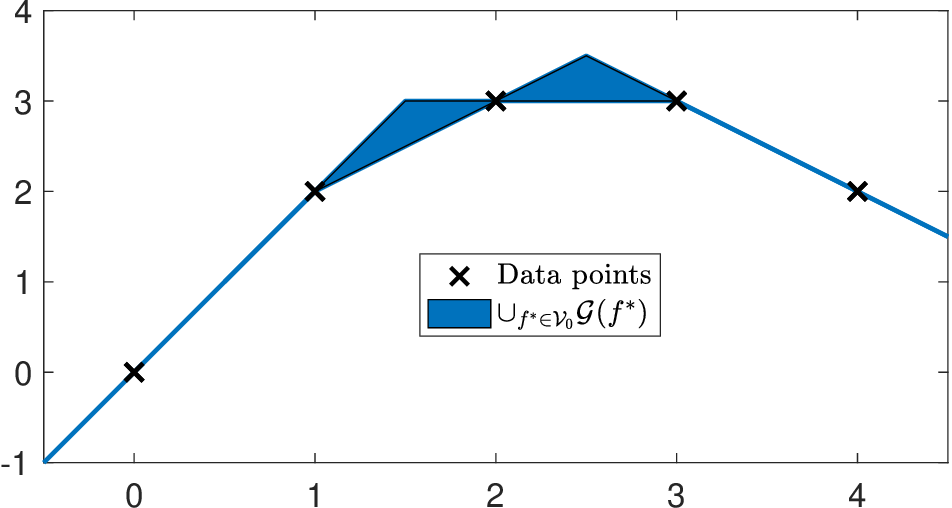}};
    \begin{scope}[x={(image.south east)},y={(image.north west)}]
        \draw (0.45,0.67) node {$\Delta_2$};
        \draw (0.62,0.75) node {$\Delta_3$};
        %\draw (0.31,0.74) node {$\widetilde{\mathrm{P}}_1$};
        %\draw (0.69,0.9) node {$\widetilde{\mathrm{P}}_2$};
    \end{scope}
\end{tikzpicture}
\caption{Example with $M=5$ of the geometric domain $\cup_{\fopt \in \mathcal{V}_0}  \mathcal{G}(\fopt)$ containing all the solutions to the~\eqref{eq:noiselessclean}. We have $\Spc{X} = \{ 2, 3\}$ and thus two triangles $\Delta_m$; all solutions follow $\fcano$ everywhere else.}
\label{fig:graph_location}
\end{figure}

The relation~\eqref{eq:domain} reveals the smallest possible geometric domain containing all the graphs of the solutions of the~\eqref{eq:noiselessclean}. To obtain a solution of the~\eqref{eq:noiselessclean}, one just needs to follow the graph of $\fcano$ outside the triangles $\Delta_m$ and take a convex or concave function inside them. An example of this domain is given in Figure~\ref{fig:graph_location} with $M=5$ and $\#\Spc{X} = 2$ triangles (this same example is treated further later in Figure~\ref{fig:2_sat}). The proof of Theorem~\ref{thm:limitdomain} is given in~\ref{app:limitdomain}. 
Next, Section~\ref{sec:sparsestsolutions} is dedicated to the study of the sparsest piecewise-linear solutions of the~\eqref{eq:noiselessclean}.

%% file: sections/sec-sparsest-sol.tex
% !TEX root = ../main.tex
%
\section{The Sparsest Solution(s) of the \eqref{eq:noiselessclean}}
\label{sec:sparsestsolutions}

    \subsection{Characterization of the Sparsest Solution(s)} \label{sec:sparsenoiseless}

We have already identified the situations where the \eqref{eq:noiselessclean} admits a unique solution, in which case it is the canonical solution introduced in Definition~\ref{def:canonicalsolutiondef}.
When the solution is not unique, Theorem \ref{theo:RTweneed} ensures that the extreme-point solutions are piecewise-linear functions with at most $K-2$ knots, and Theorem \ref{theo:sol_set} gives a complete description of the solution set. In this section, we go further by providing a complete answer to the following questions:
\begin{itemize}
    \item what is the minimal number of knots of a solution of the \eqref{eq:noiselessclean}? 
    \item what are the sparsest solutions, \ie the ones reaching this minimum number of knots?
\end{itemize}
These questions are addressed in Theorem~\ref{thm:sparsest_sol}.
%Let $N_s = \# \{ m \in \{ 2, \hdots, M-1 \}: \etacano(x_m)=0 \}$, and$
Let $\etacano$ be defined as in Proposition \ref{prop:canocertif} for fixed values of $\x, \obs \in \R^M$, and let
\begin{align}
\label{eq:saturation_indices}
\Isat &\eqdef \{ m \in \{ 2, \ldots, M-1 \}: \etacano(x_m)=\pm 1 \text{ and } \etacano(x_{m})\neq \etacano(x_{m-1}) \} \nonumber \\
&= \{ s_1, \ldots, s_{N_s} \} \qwithq s_1 < \cdots < s_{N_s}.
\end{align}
In other words, $N_s = \# \Isat$ corresponds to the number of times $\etacano$ reaches $\pm 1$. Next, let $\alpha_n\in\NN$ for $n\in\{1,\ldots,N_s\}$ be the number of consecutive intervals starting from $x_{s_n}$ in which $\etacano$ saturates at $\pm 1$, \ie
\begin{align}
\label{eq:num_saturations}
\alpha_n \eqdef \min \{ k \in \N: \etacano(x_{s_n + k +1}) \neq \etacano(x_{s_n}) \}.
\end{align}
In what follows, $\lceil x \rceil$ is the smallest integer larger or equal to $x\in \R$.
\begin{theorem}[Sparsest Solutions of the~\eqref{eq:noiselessclean}]
\label{thm:sparsest_sol}
Let $\x\in \R^M$ be the ordered sampling locations, $\obs \in \R^M$ with $M \geq 4$. Concerning the minimum sparsity of a solution of the~\eqref{eq:noiselessclean}, the following hold.
\begin{enumerate}
    \item The lowest possible sparsity (\emph{i.e.}, number of knots) of a piecewise-linear solution of the~\eqref{eq:noiselessclean} is
    \begin{align}
    \sparsity{\x}{\obs} = \sum_{n=1}^{N_s} \left\lceil\frac{\alpha_n+1}{2} \right\rceil,
\end{align} 
where the $\alpha_n$ are defined in~\eqref{eq:num_saturations}, and $N_s = \# \Isat$ where $\Isat$ is defined in~\eqref{eq:saturation_indices}.
    \item There is a unique sparsest solution of the~\eqref{eq:noiselessclean} if and only if none of the $\alpha_n$ are nonzero even numbers.
    \item If one or more $\alpha_n>0$ are even, then there are uncountably many sparsest solutions to the~\eqref{eq:noiselessclean}. The number of degrees of freedom $n_{\mathrm{free}} (\x, \obs)$ of the set of sparsest solutions is equal to the number of even $\alpha_n$ coefficients, that is,
    \begin{equation}
     n_{\mathrm{free}} (\x, \obs) = \sum_{n=1}^{N_s} \One_{2\N_{\geq 1}}(\alpha_n).
    \end{equation}
    More precisely, for each saturation region of $\etacano$, fixing a single knot within a certain admissible segment uniquely determines the other knots within the saturation region.
    %If one or more $\alpha_n>0$ are even, then there are uncountably many sparsest solutions to \eqref{eq:noiselessclean}. The number of degrees of freedom of the set of sparsest solutions is equal to the number of even $\alpha_n$ coefficients
    %More precisely, for each saturation region of $\etacano$, fixing a single knot within a certain admissible segment uniquely determines the other knots within the saturation region.
\end{enumerate}
\end{theorem}

The proof of Theorem \ref{thm:sparsest_sol} is given in \ref{sec:sparsest_solutions_proof}. Illustrations of its items 2. and 3. with a single saturation region (\ie $N_s=1$) are given in Figures \ref{fig:3_sat} and \ref{fig:2_sat} respectively. In Figure \ref{fig:3_sat}, the unique sparsest solution is shown. In Figure \ref{fig:2_sat}, any point $\widetilde{\mathrm{P}}_1$ in the segment that connects the points $\P0{2}$ and $\widetilde{\mathrm{P}}$ yields one of the sparsest solutions, with a uniquely determined second knot $\widetilde{\mathrm{P}}_2$. In the latter example, there is thus a single degree of freedom $n_{\mathrm{free}}(\x, \obs)$ in the set of sparsest solutions to the~\eqref{eq:noiselessclean}.

\begin{figure}[t]
\centering
\subfloat[Sparsest solution]{\includegraphics[width=0.5\linewidth]{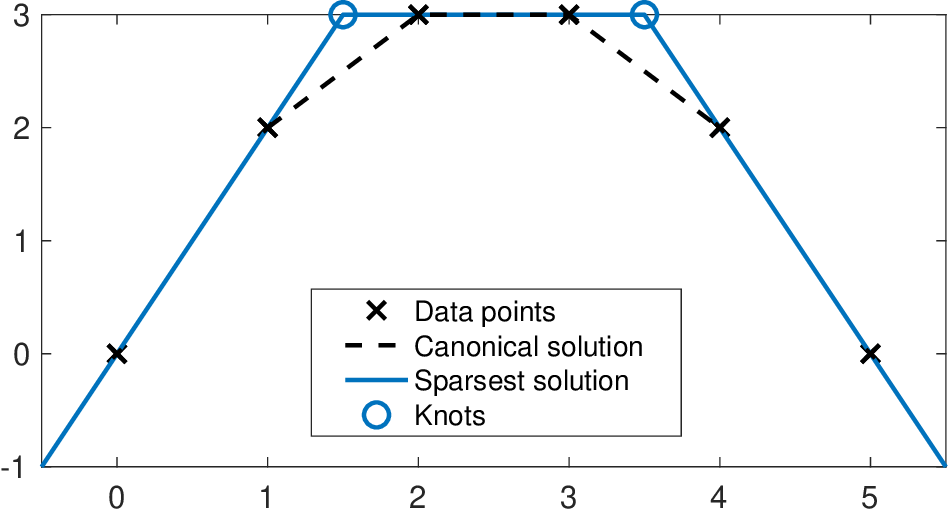}}
\subfloat[Canonical certificate]{\includegraphics[width=0.5\linewidth]{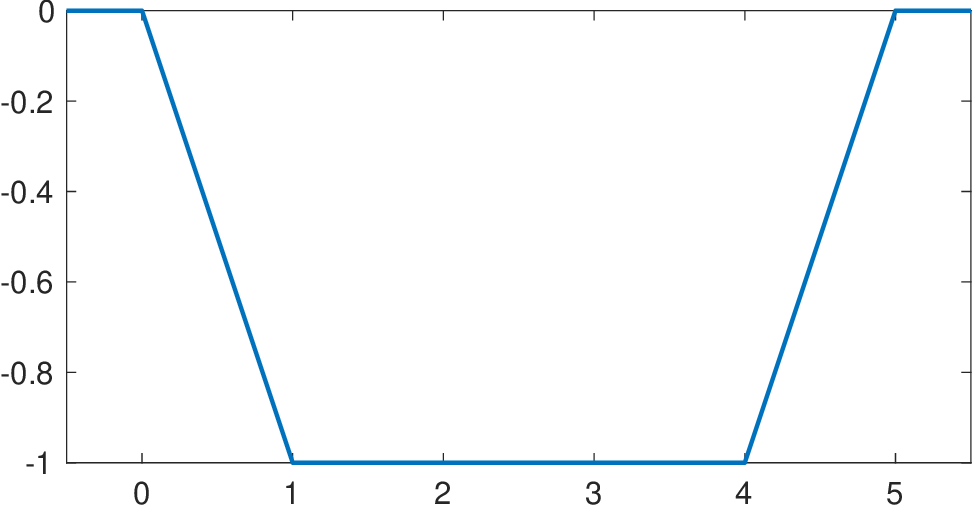}}
\caption{Example with $M=6$ and $\alpha = 3$ consecutive saturation intervals of $\etacano$ at -1. The unique sparsest solution has $P=2$ knots.}
\label{fig:3_sat}
\end{figure}

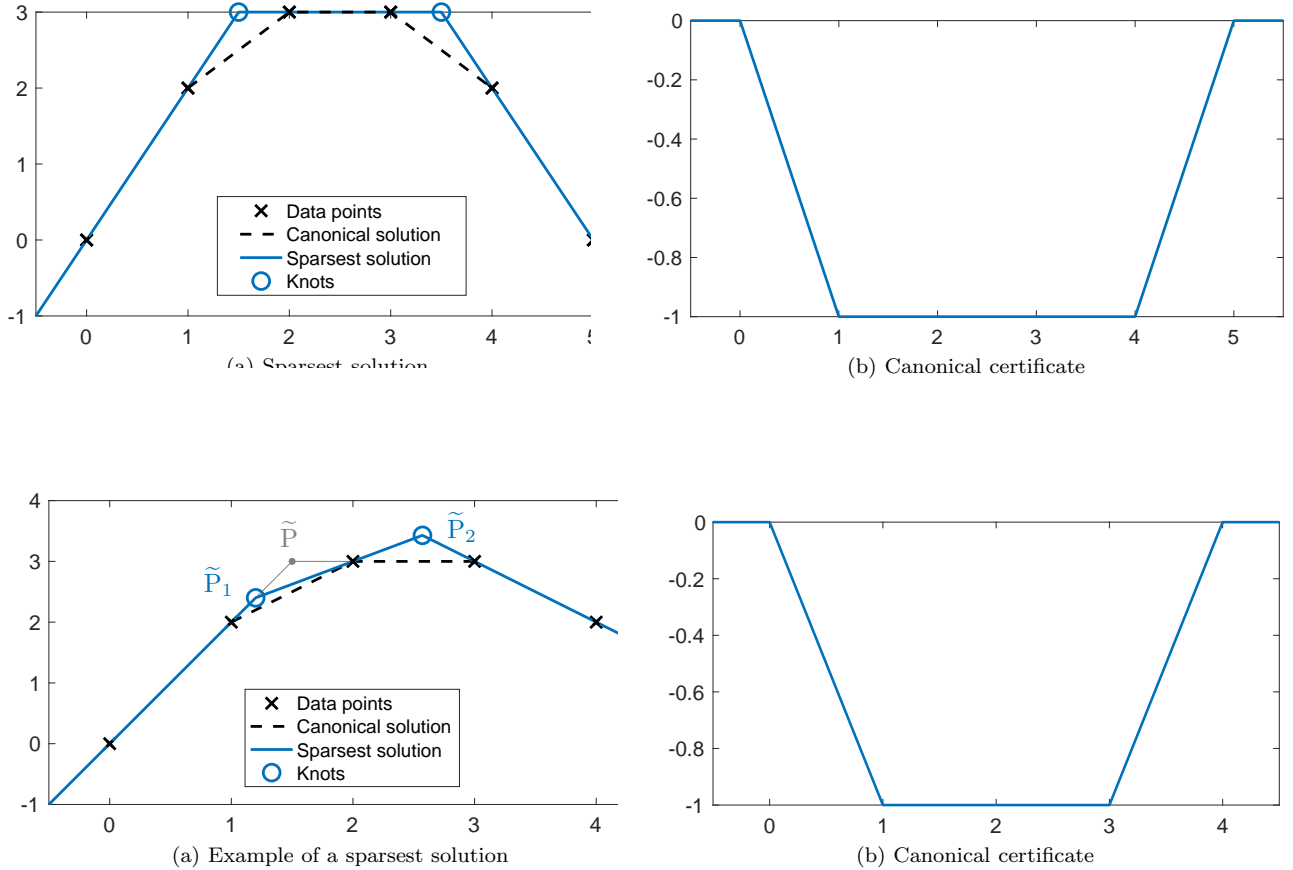
\begin{figure}[t]
\centering
\subfloat[Example of a sparsest solution]{
\begin{tikzpicture}
    \node[anchor=south west,inner sep=0] (image) at (0,0,0) {\includegraphics[width=0.50\linewidth]{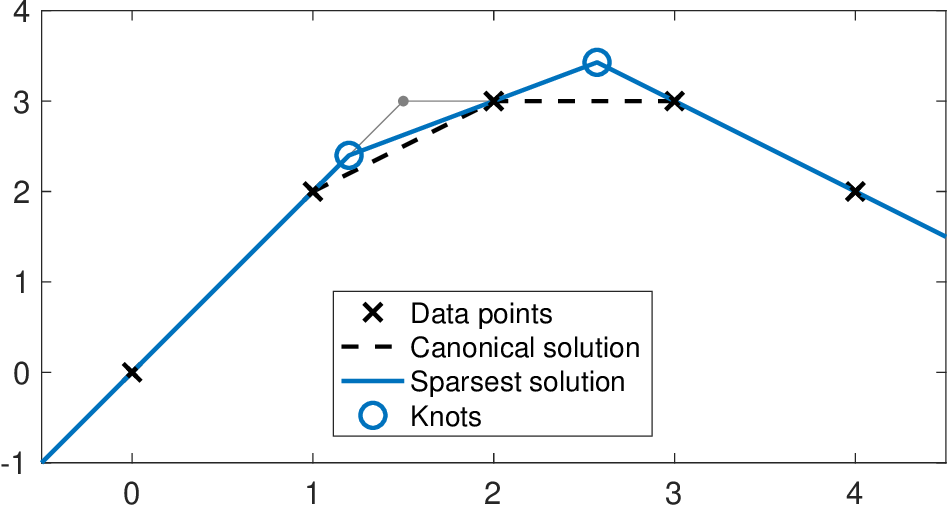}};
    \begin{scope}[x={(image.south east)},y={(image.north west)}]
        \draw[gray] (0.42,0.88) node {$\widetilde{\mathrm{P}}$};
        \draw[matblue] (0.31,0.74) node {$\widetilde{\mathrm{P}}_1$};
        %\draw[dashed,-latex, matblue] (0.65,0.875) -- +(0.7cm,0) node[anchor=west] {$\widetilde{\mathrm{P}}_2$};
        \draw[matblue] (0.69,0.9) node {$\widetilde{\mathrm{P}}_2$};
    \end{scope}
\end{tikzpicture}
}
\subfloat[Canonical certificate]{\includegraphics[width=0.48\linewidth]{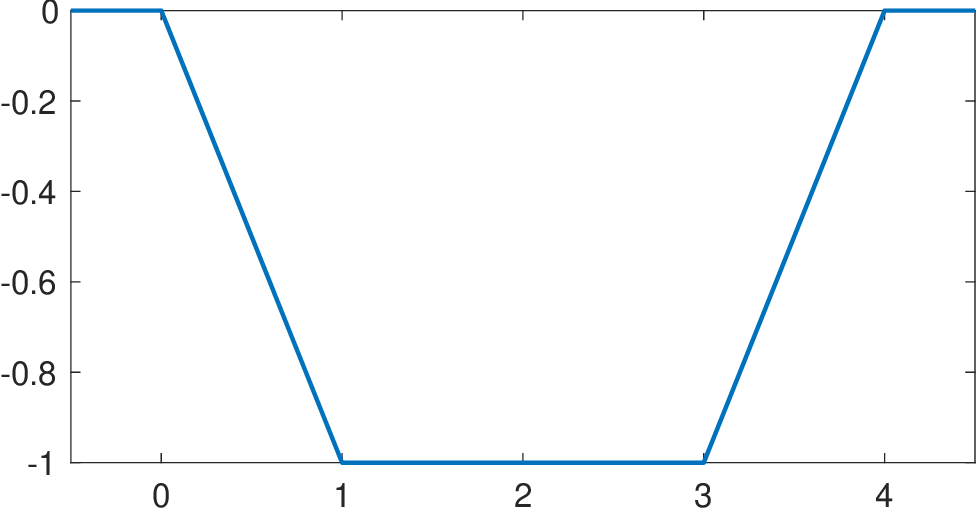}}
\caption{Example with $M=5$ and $\alpha = 2$ consecutive saturation intervals of $\etacano$ at -1. The sparsest solutions have $P=2$ knots.}
\label{fig:2_sat}
\end{figure}

\subsection{Algorithm for Reaching a Sparsest Solution} \label{sec:algonoiseless}
    The results of Theorem \ref{thm:sparsest_sol} suggest a simple yet elegant algorithm for constructing a sparsest solution of the~\eqref{eq:noiselessclean} for given sampling locations $\x = (x_1, \ldots , x_M)$ and data $\obs = (y_{0, 1}, \ldots , y_{0, M})$. The pseudocode is given in Algorithm \ref{alg:algo_constrained}, which applies the sparsifying procedure described in Lemma \ref{lem:construction_sparsest_sol} in every saturation interval. Since the latter is rather lengthy and technical, it is given in \ref{sec:sparsest_solutions_proof} for ease of reading.
    The proof of Theorem \ref{thm:sparsest_sol} guarantees that the output $f^\ast$ of Algorithm \ref{alg:algo_constrained} is indeed a sparsest solution to the~\eqref{eq:noiselessclean}, with sparsity $\sparsity{\x}{\obs}$ as defined in Theorem \ref{thm:sparsest_sol}.
    The following observations can be made concerning Algorithm \ref{alg:algo_constrained}.

\begin{algorithm}[t]
 \KwIn{$\x, \obs$}
 compute $a_1, \ldots a_M$ defined in \eqref{eq:a_coefs}; % \tcp*{Weights of $\fcano$}
 $[\etacano(x_1), \ldots, \etacano(x_M)] = [0, \sign(a_2), \ldots, \sign(a_{M-1}), 0]$;\\
compute $N_s$, $s_1$, $\ldots$, $s_{N_s}$ and $\alpha_1$, $\ldots$, $\alpha_{N_s}$ defined in \eqref{eq:saturation_indices} and \eqref{eq:num_saturations};\\
$\hat{\V \tau} = [\,]; \hat{\V a} = [\,]$; \\
\For{ $n \leftarrow 1$ \KwTo $N_s$}{
    $P \leftarrow \lceil\frac{\alpha_n+1}{2} \rceil$ ; \\% \tcp*{Minimum sparsity in $[x_{s_n}, x_{s_n+\alpha_n}]$}\\
  compute $\tilde{\tau}_1$, $\ldots$, $\tilde{\tau}_P$ and $\tilde{a}_1$, $\ldots$, $\tilde{a}_P$ using
 \eqref{eq:sparsest_sol_odd} or \eqref{eq:sparsest_sol_even};\\
  $\hat{\V \tau} \leftarrow  [\hat{\V \tau}, \tilde{\tau}_1, \ldots, \tilde{\tau}_P]$; \\ % \tcp*{Vector of knot locations}
  $\hat{\V a} \leftarrow  [\hat{\V a}, \tilde{a}_1, \ldots, \tilde{a}_P]$; %\tcp*{Vector of weights}
}
\Return $\fopt \leftarrow \sum_{k=1}^{K} \hat{a}_k (\cdot - \hat{\tau}_k)_+$
\vspace{2mm}
\caption{Pseudocode of our algorithm to find a sparsest solution of the~\eqref{eq:noiselessclean}.}
 \label{alg:algo_constrained}
\end{algorithm}
\begin{itemize}
    \item In the cases where the sparsest solution is not unique, the choice of solution specified by \eqref{eq:sparsest_sol_even} (which is \emph{not} the one shown in Figure~\ref{fig:2_sat}) is guided by simplicity. However, it is an arbitrary choice that can be adapted depending on the application.
    \item Notice that the $x_m$ such that $\etacano(x_m) = 0$ need not be included in the vector of knots $\x'$ built in the algorithm, since we have $a_m = 0$. Therefore, there is in fact no knot at $x_m$ in the canonical solution, which implies that the sparsity of $\fcano$ is strictly less than $M-2$. This corresponds to alignment cases of the data points, \ie the points $\P0{m-1}$, $\P0{m}$, and $\P0{m+1}$ are aligned, as illustrated in Figure~\ref{fig:cano_example}.
    %\item Algorithm \ref{alg:algo_constrained} is extremely fast: it takes linear time $\Spc{O}(M)$ with respect to the number of data points. This is a remarkable feature, since it is used to solve a continuous-domain optimization problem exactly. In fact, it is in the same complexity class as the simple computation of $\fcano$.
    \item Algorithm \ref{alg:algo_constrained} can be translated into an online algorithm, \ie an updated solution can be computed efficiently if a new input data point is added. More precisely, when a new data point $\P0{M+1}$ is added, the reconstructed signal is at worst only modified in the saturation interval $I = [x_{s_{n}-1}, x_{s_n + \alpha_n}]$ if $x_{M+1} \in I$. Since in practice, we usually have $\alpha_n \ll M$, the computational complexity of updating the solution is typically much smaller than rerunning the complete offline algorithm.
\end{itemize}

\subsection{Computational Complexity}
Algorithm \ref{alg:algo_constrained} is very fast and memory-efficient; it requires at most two passes through the data points, and thus has linear time and space complexity $\Spc{O}(M)$ with respect to the number of data points. More precisely, computing the canonical interpolant (\ie, the $a_m$ coefficients using~\eqref{eq:a_coefs}) requires about $3M$ subtractions and $M$ divisions, and storing two arrays of size $M$. Next, in the worst-case scenario where $\sign(a_2) = \ldots = \sign(a_{M-1})$, computing the sparsest interpolant (\ie the $\tilde{a}_k$ and $\tilde{x}_k$ coefficients using~\eqref{eq:sparsest_sol_odd} or \eqref{eq:sparsest_sol_even}) requires approximately $M$ multiplications, $M$ additions, $\frac{M}{2}$ divisions and storing two arrays of size $\frac{M}{2}$. Hence, the complete worst-case time complexity for Algorithm~\ref{alg:algo_constrained} requires $4M$ additions, $M$ multiplications and $\frac{3M}{2}$ divisions, and its space complexity is $3M$.

%% file: sections/sec-gBLASSO.tex
% !TEX root = ../main.tex
%
\section{The Solutions of the~\eqref{eq:noisyclean}}\label{sec:noisy}

We now focus on the~\eqref{eq:noisyclean} problem, in which the interpolation of the data is no longer required to be exact as in Section~\ref{sec:noiseless}, but is formulated as a penalized problem with a regularization parameter $\lambda > 0$. In practice, such problems are typically formulated when we have access to noise-corrupted measurements $\V{y} = \V{y}_0 + \V{n}$ where $\V{n} \in \R^M$ is a noise term. In this case, we solve the following optimization problem
    \begin{equation}\label{eq:noisyclean}
        \Spc{V}_\lambda \eqdef \argmin_{f \in \BV} \sum_{m=1}^M E(f(x_m), y_{m}) + \lambda \lVert \Op{D}^2 f \rVert_{\mathcal{M}}, \tag{\gBLASSO}
    \end{equation}
    where $E(\cdot, y)$ is a strictly convex, coercive, and differentiable cost function (typically quadratic, \ie $E(z, y) = \frac{1}{2} (z - y)^2$) for any $y \in \R$, and $\lambda > 0$ is a regularization parameter. The latter controls the weight between the data fidelity term $\sum_{m=1}^M E(f(x_m), y_{m})$ and the regularization term $\Vert \Op{D}^2 f \Vert_{\Spc{M}}$, and should therefore be adapted to the noise level.

    \subsection{From the~\eqref{eq:noiselessclean} to the~\eqref{eq:noisyclean}: Reduction to the Noiseless Case} \label{sec:solutionnoisy}
      
We now show that the~\eqref{eq:noisyclean} can be reduced to an optimization problem of the form~\eqref{eq:noiselessclean} (see~\cite[Theorem 5]{gupta2018continuous}), as is often done in finite-dimensional optimization problems~\cite[Lemma 1]{Tibshirani2013lasso}.

    \begin{proposition}[Reformulation of the~\eqref{eq:noisyclean} as a~\eqref{eq:noiselessclean} Problem]
    \label{prop:penalized_to_constrained}
        Let $\x\in \R^M$ be the ordered sampling locations, and $\V{y}\in \R^M$ with $M \geq 2$. Let $E : \R\times\R \rightarrow \R^+$ be a cost function such that $E(\cdot, y)$ is strictly convex, coercive, and differentiable for every $y \in \R$.
        Then, there exists a unique $\V{y}_\lambda \in \R^M$ such that, for any $\fopt \in \mathcal{V}_\lambda$, $\fopt(x_m) = y_{\lambda, m}$ for all $m\in\{1, \ldots, M\}$. Moreover, we have that the~\eqref{eq:noisyclean} is equivalent to the~\eqref{eq:noiselessclean} with the measurement vector $\V{y}_0 = \V{y}_\lambda$, \ie
        \begin{equation}
        \label{eq:noisy_constrained}
            \mathcal{V}_{\lambda} = 
             \argmin_{\substack{f \in \BV \\ f(x_m) = y_{\lambda, m}, \ m=1, \ldots, M}} \lVert \Op{D}^2 f \rVert_{\Spc{M}}.
        \end{equation}
    \end{proposition}
    
The proof of Proposition~\ref{prop:penalized_to_constrained} is provided in \ref{sec:penalized_to_constrained}. The implications of this result for our problem are huge: it implies that all the results of Section~\ref{sec:noiseless}---in particular, uniqueness, form the solutions, and sparsest solutions---can be applied to the penalized problem~\eqref{eq:noisyclean}. The only---but crucial---catch is that the samples $\V{y}_\lambda\in\RR^M$ are unknown. Fortunately, the following proposition enables us to compute them through a standard $\ell_1$-regularized discrete optimization.

\begin{proposition}
\label{prop:optizlambda}
Assume that the hypotheses of Proposition~\ref{prop:penalized_to_constrained} are met. Then, the vector $\V{y}_\lambda\in\RR^M$ defined in Proposition~\ref{prop:penalized_to_constrained} is the unique solution of the discrete minimization problem
        \begin{equation}
            \label{eq:optizlambda}
          \V{y}_\lambda =  \argmin_{\V{z}\in\R^M} \sum_{m=1}^M E(z_m, y_m) + \lambda \Vert \M{L} \V{z} \Vert_{1},
        \end{equation}
        where $\M{L} \in \R^{(M-2)\times M}$ is given by 
\begin{align}
\label{eq:matrixL}
\setlength\arraycolsep{2pt}
\M{L} \eqdef \begin{pmatrix} 
v_1 & - (v_1 + v_2) & v_2 & 0 & \cdots & 0 \\
0 & v_2 & - (v_2 + v_3) & v_3  & \ddots & \vdots \\
\vdots & \ddots & \ddots & \ddots  & \ddots & 0 \\
0 & \cdots & 0 & v_{M-2}& - (v_{M-2} + v_{M-1}) & v_{M-1}
\end{pmatrix},
\end{align}
and $\V{v} \eqdef (v_1, \ldots , v_{M-1}) \in \R^{M-1}$ is defined as $v_m \eqdef \frac{1}{x_{m+1} - x_m}$ for $m \in\{1,\ldots, M-1\}$.
%\begin{align}
%\label{eq:matrixL}
%\setlength\arraycolsep{2pt}
%\M{L} = \begin{pmatrix} 
%v_1 & w_1 & v_2 & 0 & \cdots & 0 \\
%0 & v_2 & w_2 & v_3  & \ddots & \vdots \\
%\vdots & \ddots & \ddots & \ddots  & \ddots & 0 \\
%0 & \cdots & 0 & v_{M-2}& w_{M-2} & v_{M-1}
%%\end{pmatrix}.
%\end{align}
%In~\eqref{eq:matrixL}, $\V{v} \in \R^{M-1}$ is defined as $v_m = \frac{1}{x_{m+1} - x_m}$, $m = 1, \ldots , M-1$, and $\V{w} \in \R^{M-2}$ is defined as $w_m = -(v_m + v_{m+1})$, $m = 1, \ldots , M-2$.
\end{proposition}
\begin{proof}
In this proof, we denote by $f_{\V{z}}$ the canonical solution (defined in Definition~\ref{def:canonicalsolutiondef}) of the~\eqref{eq:noiselessclean} with sampling locations $\V{x}$ and data point $\V{y}_0 = \V{z}$. Let us first prove that if $\zopt\in \R^M$ is a solution of problem~\eqref{eq:optizlambda}, then $f_{\zopt}\in\BV$ is a solution of the~\eqref{eq:noisyclean}. We then deduce that for all $m\in\{1,\ldots,M\}$, $z_m=f_{\zopt}(x_m)= y_{\la,m}$ (where the last equality is true thanks to Proposition~\ref{prop:penalized_to_constrained}), which proves the desired result, \ie $\V y_\la = \zopt$ is the unique solution of problem~\eqref{eq:optizlambda}.

Let $\V{z} \in \R^M$. Using Equations~\eqref{eq:canonical_sol} and \eqref{eq:a_coefs}, we have that $\Vert \Op{D}^2 f_{\V{z}} \Vert_\Spc{M} = \sum_{m=2}^{M-1} \vert a_m \vert$, where $a_m = \frac{z_{m+1} - z_m}{x_{m+1} - x_m} - \frac{z_{m} - z_{m-1}}{x_{m} - x_{m-1}}$. Therefore, we have $\Vert \Op{D}^2 f_{\V{z}} \Vert_\Spc{M} = \Vert \M{L} \V{z} \Vert_{1}$, where $\M{L}$ is given by Equation~\eqref{eq:matrixL}. This yields $\sum_{m=1}^M E(f_{\V{z}}(x_m), y_m) + \lambda \Vert f_{\V{z}} \Vert_\Spc{M} = \sum_{m=1}^M E(z_m, y_m) + \lambda \Vert \M{L} \V{z} \Vert_{1}$. Applied to the particular case $\V{z} = \V{y}_\lambda$, we obtain the equality $\sum_{m=1}^M E(y_{\lambda, m}, y_m) + \lambda \Vert \M{L} \V{y}_\lambda \Vert_{1} = \Spc{J}_\lambda$, where $\Spc{J}_\lambda$ is the optimal cost of the~\eqref{eq:noisyclean}, since by Proposition~\ref{prop:optimality_cano}, $f_{\V{y}_\lambda} \in \Spc{V}_\lambda$. This proves that the optimal value of problem~\eqref{eq:optizlambda} is lower or equal than $\Spc{J}_\lambda$.

Next, let $\zopt$ be a solution of problem~\eqref{eq:optizlambda} (which exists due to the coercivity of $E(\cdot, y)$ for any $y \in \R$). We thus have from before that
\begin{align}
\Spc{J}_\lambda \leq \sum_{m=1}^M E(f_{\zopt}(x_m), y_m) + \lambda \Vert \Op{D}^2 f_{\zopt} \Vert_\Spc{M} = \sum_{m=1}^M E(z_m, y_m) + \lambda \Vert \M{L} \zopt \Vert_{1} \leq \Spc{J}_\lambda,
\end{align}
which yields the desired result $f_{\zopt} \in \Spc{V}_\la$.
\end{proof}
\subsection{Algorithm for Reaching a Sparsest Solution of the~\eqref{eq:noisyclean}}
\label{sec:algonoisy}
By combining results from the previous sections, we now formulate the following simple algorithmic pipeline to reach a sparsest solution of the~\eqref{eq:noisyclean}.

\begin{proposition}
\label{prop:algo_complete}
Let $\x\in \R^M$ be the ordered sampling locations and $\V{y} \in \R^M$  with $M \geq 2$, and let $E: \R\times \R \to \R^+$ be a cost function such that $E(\cdot, y)$ is strictly convex, coercive, and differentiable for any $y \in \R$. Let the function $\fopt$ be obtained through the following two-step procedure:
\begin{enumerate}
    \item Compute $\V{y}_\lambda\in\RR^M$ (defined in Proposition~\ref{prop:penalized_to_constrained}) by solving problem~\eqref{eq:optizlambda};
    \item Apply Algorithm~\ref{alg:algo_constrained} with the measurement vector $\V{y}_0 = \V{y}_\lambda$ to compute a sparsest solution $\fopt$ of the~\eqref{eq:noiselessclean} given by Equation~\eqref{eq:noisy_constrained}.
\end{enumerate}
Then, $\fopt$ is one of the sparsest solutions to the~\eqref{eq:noisyclean}, with sparsity $\sparsity{\x}{\V{y}_\lambda}$ as defined in Theorem~\ref{thm:sparsest_sol}.
\end{proposition}
\begin{proof}
Proposition~\ref{prop:penalized_to_constrained} guarantees that the~\eqref{eq:noisyclean} is equivalent to the~\eqref{eq:noiselessclean} with the measurement vector $\V{y}_0 = \V{y}_\lambda$. Proposition~\ref{prop:optizlambda} then specifies that $\V{y}_\lambda$ can be computed by solving problem~\eqref{eq:optizlambda}. Finally, as demonstrated in the proof of Theorem~\ref{thm:sparsest_sol}, the output $\fopt$ of Algorithm~\ref{sec:algonoiseless} reaches a sparsest solution of the corresponding~\eqref{eq:noiselessclean} problem, which thus has sparsity $\sparsity{\x}{\V{y}_\lambda}$.
\end{proof}

Proposition~\ref{prop:algo_complete} proposes a simple but very powerful algorithm. It reaches a sparsest solution of the \eqref{eq:noisyclean} - a challenging task a priori - in two simple steps. The first consists in solving a standard $\ell_1$-regularized discrete problem, for which many off-the-shelf solvers such as ADMM~\cite{boyd2010distributed} are available. The second is our proposed sparsifying procedure, which converges in finite time. The following remarks can be made concerning Proposition~\ref{prop:algo_complete}.
% \begin{remark}
% The computational bottleneck of the pipeline described in Proposition~\ref{prop:algo_complete} is its item 1., as Problem~\eqref{eq:optizlambda} admits no closed-form solution due to the non-differentiable $\ell_1$ term. It is thus typically solved using an iterative procedure that does not converge in finite time, such as ADMM. By contrast, as explained in Section~\ref{sec:algonoiseless}, item 2. requires a finite number $\Spc{O}(M)$ of operations to reach an exact solution.
% \end{remark}

\begin{remark}
 Algorithm~\ref{alg:algo_constrained} still converges to a solution of the~\eqref{eq:noisyclean} when $E$ is only a convex function, and not strictly convex as assumed in Propositions~\ref{prop:penalized_to_constrained} and~\ref{prop:optizlambda}. The difference is that Proposition~\ref{prop:penalized_to_constrained} no longer holds true in that there is no unique vector of measurements $\V{y}_\lambda$. The solution set of the constrained problem~\eqref{eq:noisy_constrained} is thus in general a strict subset of $\Spc{V}_\lambda$. Hence, the obtained solution is not necessarily the sparsest solution of the full solution set $\Spc{V}_\lambda$, but only of this subset.
 
 As for the assumption that $E$ is differentiable, it is not a requirement for Proposition~\ref{prop:algo_complete}. However, as it is needed later on in Proposition~\ref{prop:linear_regression}, we include it in order to have consistent assumptions concerning $E$ throughout the paper.
  \end{remark}

\subsection{Computational Complexity}
The computational bottleneck of the pipeline described in Proposition~\ref{prop:algo_complete} is its item 1; as an illustration, for $M=50$ data points, item 1 runs in about 200ms on commodity hardware, compared to 2ms for item 2. This gap is due to the absence of a closed-form solution to Problem~\eqref{eq:optizlambda} owing to the non-differentiable $\ell_1$ term. The latter is thus typically solved using an iterative procedure that does not converge in finite time, such as ADMM. It is well known that ADMM has a $\mathcal{O}(1/k)$ convergence rate in general, where $k$ is the number of iterations~\cite{he2012o1/n}. In our case, when $E$ is strongly convex with Lipschitz-continuous gradient, \eg with a standard quadratic loss, ADMM achieves a linear convergence rate \cite{deng2015global}. In general, the cost per iteration of ADMM depends on how the $\V z$-minimization step is performed, which may depend on the choice of $E$. In the standard quadratic case, this step consists in applying the inverse of an $M \times M$ matrix, which is fixed across iterations, to an iteration-dependant vector. To achieve this, the inverse matrix must either be computed beforehand (which is our approach), or this inverse must by applied in a matrix-free fashion. In our approach, the computational bottleneck at each iteration being the storage of the inverse matrix and its application to a vector, the computational complexity per iteration of ADMM is $\Spc O (M^2)$ both in time and space.

% \texorpdfstring{tex}{pdfbookmark} removes the warning about the maths symbol in the title that will not show up in in the PDF bookmarks.
\subsection[\texorpdfstring{tex}{pdfbookmark}]{Range of the Regularization Parameter $\lambda$}
\label{sec:range_lambda}
In practice, the choice of the regularization parameter $\lambda$ is the critical element that determines the performance of our algorithm. Although this choice is highly data-dependant, in this section, we show that the search can be restricted to a bounded interval. The lower bound is $\lambda \to 0$, which corresponds at the limit to exact interpolation, that is the~\eqref{eq:noiselessclean}. The upper bound $\lambda \to +\infty$ corresponds to the linear regression regime, which is described in the following proposition.
\begin{proposition}[Linear Regression Regime of the~\eqref{eq:noisyclean}]
\label{prop:linear_regression}
Let $\x\in \R^M$ be the ordered sampling locations and $\V{y} \in \R^M$ with $M \geq 2$. Let $E: \R\times \R \to \R^+$ be a cost function such that $E(\cdot, y)$ is strictly convex, coercive, and differentiable for any $y \in \R$.  Then, the following properties hold.
\begin{enumerate}
    \item There is a unique solution $(\opt \beta_0, \opt \beta_1) \in \R^2$ to the linear regression problem
\begin{align}
\label{eq:linear_regression}
( \opt \beta_0, \opt \beta_1) \eqdef \argmin_{(\beta_0, \beta_1) \in \R^2} \sum_{m=1}^M E(\beta_0 + \beta_1 x_m, y_m).
\end{align}
\end{enumerate}
We can thus define the value
\begin{align}
\label{eq:lambda_max}
\lambda_{\text{max}} \eqdef \left\Vert {\M{L}^T}^\dagger \begin{pmatrix}\partial_1 E(\opt \beta_0 + \opt \beta_1 x_m, y_1) \\ \vdots \\ \partial_1 E(\opt \beta_0 + \opt \beta_1 x_M, y_M) \end{pmatrix} \right\Vert_\infty,
\end{align}
where $\partial_1 E$ denotes the partial derivative with respect to the first variable of $E$, the matrix ${\M{L}^T}^\dagger$ denotes the pseudoinverse of $\M{L}^T$, and $\M{L}$ is defined as in~\eqref{eq:matrixL}.
\begin{enumerate}
 \setcounter{enumi}{1}
\item For any $\lambda \geq \lambda_{\text{max}}$, the solution to the discrete problem~\eqref{eq:optizlambda} is given by $\V{y}_\lambda = \opt \beta_0 \V{1} + \opt \beta_1 \x$, where $\V{1} \eqdef (1, \ldots , 1) \in \R^M$.
\item For any $\lambda \geq \lambda_{\text{max}}$, the solution to the~\eqref{eq:noisyclean} is unique and is the linear function $f_{\text{max}}$ given by $f_{\text{max}}(x) \eqdef \opt \beta_0 + \opt \beta_1 x$.
\end{enumerate}

\end{proposition}

The proof of Proposition~\ref{prop:linear_regression} is given in \ref{sec:linear_regression_proof}. %Contrary to the exact interpolation case, linear regression~\eqref{eq:linear_regression} is not a limit case of \eqref{eq:noisyclean}:
%In fact, there exists an upper-bound value for $\lambda$ above which \eqref{eq:noisyclean} amounts to linear regression, \ie its unique solution is the linear function determined by~\eqref{eq:linear_regression}.
Proposition~\ref{prop:linear_regression} guarantees that the range of $\lambda$ can be restricted to the interval $(0, \lambda_{\text{max}}]$: indeed, all values $\lambda \geq \lambda_{\text{max}}$ lead to linear regression. Moreover, the value of $\lambda_{\text{max}}$ given in \eqref{eq:lambda_max} only depends on the data $\x, \V{y} \in \R^M$ and is easy to compute numerically - the most costly step being the computation of the pseudoinverse ${\M{L}^T}^\dagger$. Note that item 2 in Proposition~\ref{prop:linear_regression}, which stems from duality theory, is a generalization of a well-known result for the LASSO problem~\cite[Proposition 1.3]{bach2011optimization}, which plays a crucial role in the homotopy method~\cite{osborne2000lasso}. The difference here is the presence of a non-invertible regularization matrix $\M{L}$ in problem~\eqref{eq:optizlambda}, which requires additional arguments in the proof.

%% file: sections/sec-experiments.tex
% !TEX root = ../main.tex
%
\section{Experiments}
\label{sec:experiments}

In this section, we describe the implementation of our two-step algorithm presented in Section \ref{sec:algonoisy} and show our experimental results. The first step of our algorithm - which consists in solving problem \eqref{eq:optizlambda} with ADMM - is implemented using GlobalBioIm, a Matlab inverse-problem library developed by the Biomedical Imaging Group at EPFL \cite{soubies2019pocket}. In all our experiments, we choose the standard quadratic data fidelity loss $E(z, y) = \frac{1}{2}(z-y)^2$. This choice leads to $\partial_1E(z, y) = z-y$, which enables the simple computation of $\lambda_{\text{max}}$ using \eqref{eq:lambda_max}.

We present an illustrative example with $M=30$ simulated data points in Figure \ref{fig:cost_sparsity}. A small number is chosen for visualization purposes; an application of our algorithm with a larger number of $M=200$ data points was shown in Figure \ref{fig:regression_vs_NN}. The sampling locations $x_m$ are generated following a uniform distribution in the $[\frac{m-1}{M}, \frac{m}{M}]$ intervals for $m=1, \hdots, M$. Next, the ground-truth signal, a piecewise-linear spline $f_0$ in the sense of Definition~\ref{def:splines} with 2 knots, is generated, with random knot locations $\tau_m$ within the interval $[0,1]$, and i.i.d. Gaussian amplitudes $a_m$ ($\sigma_a^2 = 1$). We then have $y_m = f_0(x_m) + n_m$ for $m=1, \hdots, M$, where $\V{n} \in \R^M$ is i.i.d. Gaussian noise ($\sigma_n^2 = 4 \times 10^{-4}$).

% \texorpdfstring{tex}{pdfbookmark} removes the warning about the maths symbol in the title that will not show up in in the PDF bookmarks.
\subsection[\texorpdfstring{tex}{pdfbookmark}]{Extreme Values of $\lambda$}
The reconstructions using our algorithm for extreme values of $\lambda$ - \ie $\lambda \to 0$ which leads to exact interpolation of the data, and $\lambda = \lambda_{\text{max}}$ which leads to linear regression - are shown in Figure \ref{fig:cost_sparsity_extremes}. Clearly, none of these solutions are satisfactory: on one hand, linear regression is too simple to model the data adequately. On the other hand, the exact interpolator suffers from overfitting. Although thanks to the sparsification procedure in Algorithm~\ref{alg:algo_constrained}, its sparsity $\sparsity{\x}{\obsr} = 20$ is smaller than the theoretical bound $M-2 = 28$ given by Theorem \ref{theo:RTweneed}, it is still clearly much larger than the desired outcome.

\subsection{Sparsity \vs Data Fidelity Loss Trade-Off}

Next, we show the sparsity $\sparsity{\x}{\obsr}$ \vs error $\Vert \V{y} - \obsr \Vert$ trade-off curve in Figure \ref{fig:cost_sparsity_tradeoff}. The latter was obtained by applying our algorithm with 20 values of $\lambda$ (equispaced on a logarithmic scale) within the range $[\lambda_{\min}, \lambda_{\text{max}}]$, with $\lambda_{\text{max}} = 0.1713$ (as defined in~\eqref{eq:lambda_max}) and $\lambda_{\min} \eqdef 10^{-5} \times \lambda_{\text{max}} $. We thus observe the evolution from exact interpolation to linear regression as $\lambda$ increases.

Ideally, one would like to choose to value of $\lambda$ that minimizes $\Vert \V{y}_0 - \obsr \Vert$, \ie the error with respect to the noiseless data $\V{y}_0$. However, in practice, the noiseless data is unknown, and one must use the noisy data $\V{y}$. Depending on the noise level, solely minimizing $\Vert \V{y} - \obsr \Vert$ might not be a desirable objective, since it leads to overfitting. Hence, we consider the trade-off between data fidelity loss and sparsity as a proxy for the standard universality \vs simplicity trade-off in machine learning. Note that we choose the data fidelity loss $\Vert \V{y} - \obsr \Vert$ instead of $\lambda$ as the $x$-axis metric, since it is an increasing function of the latter, and the former is easier to interpret.

This trade-off curve does not specify a single optimal value of the regularization parameter $\lambda$. Instead, it helps the user choose an appropriate balance by giving quantitative, interpretable data about the possible trade-offs. A key observation is that this curve is not necessarily monotonous: the sparsity can increase as $\Vert \V{y} - \obsr \Vert$ increases, as shown in Figure \ref{fig:cost_sparsity_tradeoff}. This lack of monotonicity is rather counter-intuitive, since the overall trend as $\lambda$ increases is to go from sparsity $\sparsity{\x}{\V{y}} = 20$ to $\sparsity{\x}{\V{y}_{\lambda_{\text{max}}}} = 0$. Note that a similar behavior has been known to occur in the context of the homotopy method \cite{bach2011optimization}, although it is far from being systematic. However, the interesting feature is that, in the sparsity \vs error trade-off, some values of $\lambda$ are sometimes strictly better than others for both metrics, such as the star point over the square point in Figure~\ref{fig:cost_sparsity_tradeoff}. Having access to the full trade-off curve such as Figure \ref{fig:cost_sparsity_tradeoff} is very helpful to judiciously select a suitable value of $\lambda$. This holds true as well when the curve is monotonic: indeed, the user should select the value of $\lambda$ such that the data fidelity is lowest for the desired level of sparsity, \ie the leftmost point of every plateau.

\subsection{Example Reconstructions}
To illustrate the non-monotonicity of the sparsity \vs error curve, examples of reconstructions for two specific values of $\lambda$ are shown in Figures \ref{fig:cost_sparsity_sparse} and \ref{fig:cost_sparsity_nonsparse}. Indeed, the former reconstruction has a lower value of $\lambda$, and thus lower data-fidelity loss. Nevertheless, the reconstruction in Figure \ref{fig:cost_sparsity_sparse} is sparser, with $\sparsity{\x}{\obsr} = 3$ \vs 6 in Figure \ref{fig:cost_sparsity_nonsparse}. Note that this gap is not a numerical artefact, since the magnitude of the weights $\tilde{a}_k$ associated to the knots in Figure \ref{fig:cost_sparsity_nonsparse} is much greater than numerical precision. This indicates that the value of $\lambda$ for Figure \ref{fig:cost_sparsity_sparse} should be preferred to that of \ref{fig:cost_sparsity_nonsparse}.

\begin{figure}[t]
\centering
\subfloat[Extreme cases.]{\includegraphics[width=0.49\linewidth, valign=t]{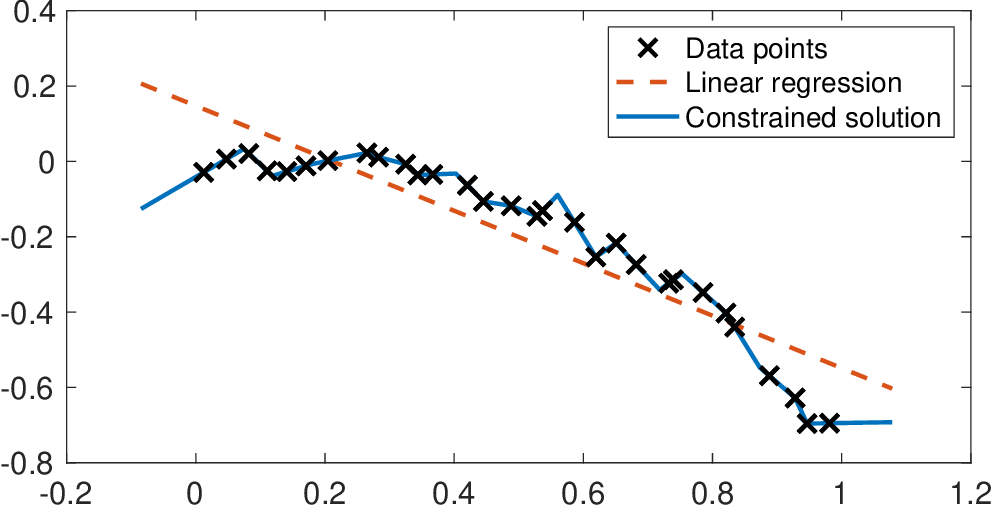} \label{fig:cost_sparsity_extremes}}
\subfloat[Sparsity \vs error trade-off. The reconstruction corresponding to the star point is shown in Figure~\ref{fig:cost_sparsity_sparse}, and the one corresponding to the square point in Figure~\ref{fig:cost_sparsity_nonsparse}.]{\includegraphics[width=0.49\linewidth, valign=t]{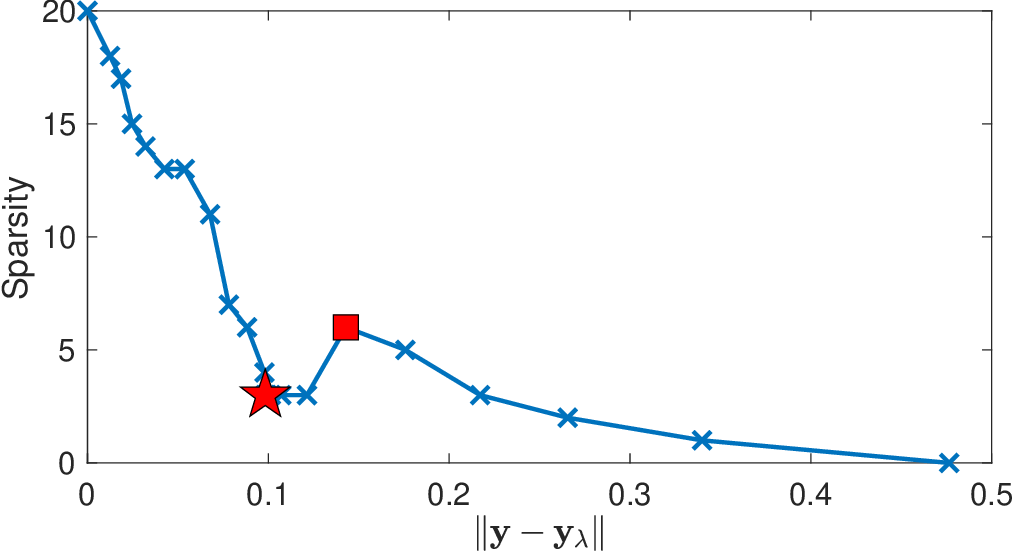}\label{fig:cost_sparsity_tradeoff}} \\
\subfloat[$\lambda = 1.7 \times 10^{-3}$, loss $\Vert \V{y} - \obsr \Vert = 0.0983$, sparsity $\sparsity{\x}{\obsr} = 3$.]{\includegraphics[width=0.49\linewidth, valign=t]{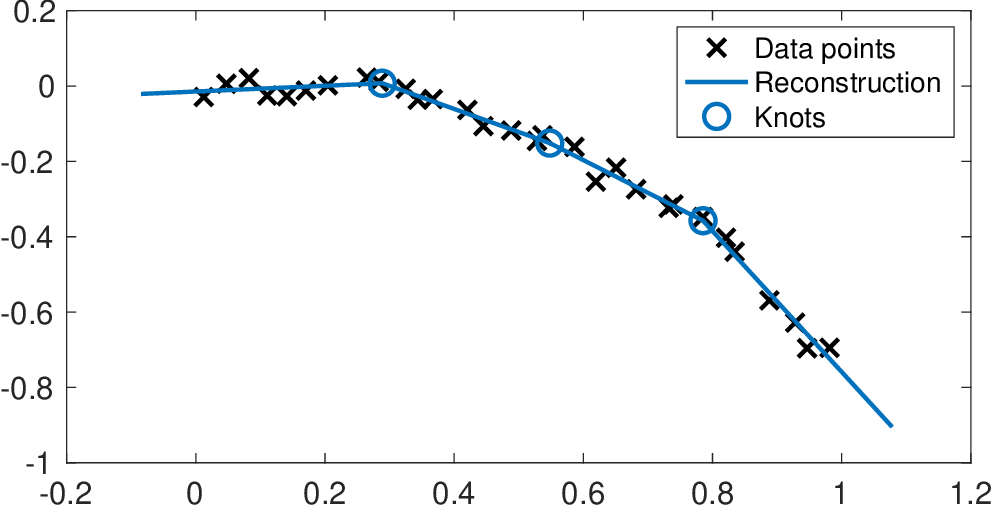}\label{fig:cost_sparsity_sparse}}
\subfloat[$\lambda = 1.71 \times 10^{-2}$, loss $\Vert \V{y} - \obsr \Vert = 0.1429$, sparsity $\sparsity{\x}{\obsr} = 6$.]{\includegraphics[width=0.49\linewidth, valign=t]{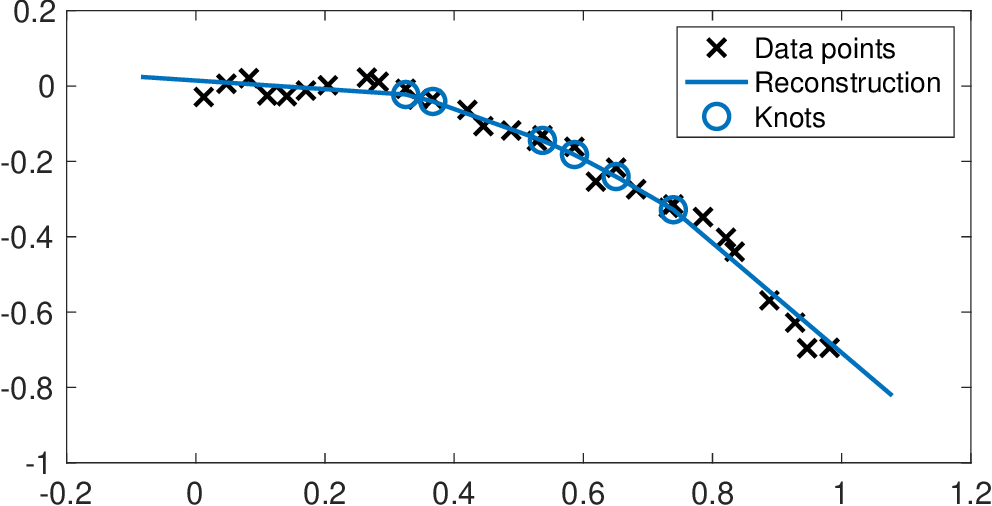}\label{fig:cost_sparsity_nonsparse}}
\caption{Example of reconstruction for varying regularization $0 \leq \lambda \leq \lambda_\text{max} = 0.1713$ with $M=30$ simulated data points. }
\label{fig:cost_sparsity}
\end{figure}

%% file: sections/sec-conclusion.tex
\section{Conclusion}
In this paper, we fully described the solution set of the~\eqref{eq:noiselessclean}, which consists in interpolating data points by minimizing the TV norm of the second derivative. More precisely, we specified the cases in which it has a unique solution, the form of all the solutions, and the subset of sparsest solutions.
We also proposed a simple and fast algorithm to reach (one of) the sparsest solution(s). We then extended these results to the~\eqref{eq:noisyclean}, by showing that it can be reformulated as a~\eqref{eq:noiselessclean} problem. Next, we introduced a two-step algorithm to solve the~\eqref{eq:noisyclean}, the first step of which consists in solving a discrete $\ell_1$-regularized problem, and the second in applying our algorithm to solve a~\eqref{eq:noiselessclean} problem. Finally, we applied our algorithm to some simulated data, and suggested plotting the sparsity \vs data fidelity error plot in order to judiciously select a suitable value of the regularization parameter. This paper paves the way for the study of supervised learning problems through the formulation of variational inverse problems with TV-based regularization, by completely describing the one-dimensional scenario. A future exciting - albeit much more challenging - prospect would be to achieve similar results in higher dimensions, \ie to reconstruct functions $f:\R^d \to \R$ with $d>1$. This would be a major milestone to better understand ReLU networks and deep learning in general, whose practical outstanding performances are yet to be fully explained.

%% file: sections/appendix.tex
 % !TEX root = ../main.tex
%
 \appendix
 \appendixpage
 
% \texorpdfstring{tex}{pdfbookmark} removes the warning about the maths symbol in the title that will not show up in in the PDF bookmarks.
\section[\texorpdfstring{tex}{pdfbookmark}]{The Space $\BV$} \label{app:1}
 
    As a complement to the characterization of the space $\BV$ in Section~\ref{sec:BVspace}, we summarize its main properties in Proposition~\ref{prop:BV2}, revealing its Banach-space structure. The construction of the native space for general spline-admissible operator $\mathrm{L}$ (we consider here the case $\mathrm{L} = \D^2$) is developed in \cite{Unser2019native}.

    \begin{proposition}[Properties of $\BV$]
        \label{prop:BV2}
        The space $\BV$ has the following properties.
        \begin{enumerate}
            \item Any function $f\in \BV$ is continuous and satisfies $f(x) = \mathcal{O}(x)$ at infinity. Affine functions $f$ such that $f(x) = a x + b$ for $a,b \in \R$ are elements of $\BV$.
            \item The linear space $\BV$ is isomorphic to $\Radon \times \R^2$ via the relation
            \begin{equation}
                \label{eq:bijection}
                f \mapsto \left( \D^2 f , (f(0), f(1)-f(0)) \right).
            \end{equation}
            \item The space $\BV$ is a Banach space for the norm
        \begin{equation} \label{eq:normBV}
        \lVert f \rVert_{\mathrm{BV}^{(2)}} \eqdef \lVert \mathrm{D}^2 f \rVert_{\mathcal{M}} + \sqrt{f(0)^2 + (f(1) - f(0))^2}.
    \end{equation}
    \item For any $w \in \Radon$, there exists a unique $f \in \BV$ such that $\D^2 f = w$ and $f(0) = f(1) = 0$.
        \end{enumerate}
%    As a Banach space, it satisfies the continuous embedding relations
 %   \begin{equation} \label{eq:embeddingsBV}
 %       \mathcal{S}(\R) \subseteq \BV \subseteq \mathcal{C}_{b,1}(\R),
 %   \end{equation}
    \end{proposition}
    
\begin{proof}
    A function in $\BV$ is the integration of a bounded-variation function, and is therefore continuous.
    If $f$ is such that $\mathrm{D}^2 f \in \Radon$, then $\mathrm{D} f$ is bounded by $\lVert \mathrm{D}^2 f \rVert_{\mathcal{M}}$. Hence,  
    \begin{equation}
        \lvert f(x) \rvert = \left\lvert f(0) + \int_{0}^x ( \mathrm{D} f ) (t) \mathrm{d} t \right\rvert \leq \lvert f(0) \rvert + \lVert \mathrm{D} f \rVert_{\infty} \lvert x \rvert,
    \end{equation}
    and $f(x) = \mathcal{O} (x)$ at infinity. 
    %The integration of a bounded Radon measures is bounded, hence its integration is bounded by $x\mapsto x$ at infinity.
    Moreover, for an affine function $f$ such that $f(x) = a + b x$, we obviously have that $\D^2 f = 0 \in \Radon$, hence $f \in \BV$.
    The relation~\eqref{eq:bijection} is clearly linear and is a bijection, since any $f \in \BV$ can be uniquely recovered from its second derivative via the specification of two boundary conditions, here the values of $f(0)$ and $f(1)$. Hence, \eqref{eq:bijection} is an isomorphism. 
    
    Due to this isomorphism, $\BV$ inherits the Banach space structure of $\Radon\times \R^2$ for the norm $\lVert (w, (\beta_0, \beta_1)) \rVert_{\mathcal{M}\times\R^2} = \lVert w \rVert + \sqrt{\beta_0^2 + \beta_1^2}$ and is hence a Banach space for the norm~\eqref{eq:normBV}.
    For the last point, by definition, any $f\in \Schp$ such that $\D^2 f = w$ is in $\BV$. 
     The space of solutions of $\mathrm{D}^2 f = w$ is then a two-dimensional space, and the solution is uniquely characterized by the specification of the two boundary conditions $f(0) = 0$ and $f(1) = 0$. 
    \end{proof}
     
     In Section~\ref{sec:BVspace}, we have introduced the operator $\Itwo$. We now summarize its main properties.
     
    \begin{proposition}[Kernel of $\Itwo$]
    \label{prop:Itwo}
     For any $w \in \Radon$, $\Itwo\{w\}$ is given by
    %        The linear and continuous operator $\Itwo : \Sch \rightarrow \Schp$ has a Schwartz kernel $g:\R^2 \rightarrow \R$ such that
            \begin{equation} \label{eq:I2integral}
                \Itwo \{ w\} (x) \eqdef \int_{\R} g(x,y)  \mathrm{d}w (y) = \langle w , g(x,\cdot) \rangle,
            \end{equation}
            where $g$ is the kernel defined over $\R^2$ as
            \begin{equation}
                \label{eq:kernelI2}
                g(x,y) \eqdef (x-y)_+ - (-y)_+ + x \left( (-y)_+ - (1-y)_+ \right),
            \end{equation}
            and is such that $g(x,\cdot)$ is a continuous and compactly supported function for any $x\in \R$.
     Then, the operator $\Itwo$ is linear and continuous from $\Radon$ to $\BV$ and satisfies the right-inverse and pseudo-left-inverse relations
   \begin{align}
        		\forall w\in\Radon,& \quad \Op D^2\{ \Itwo  \{w\} \} = w,\\
		\forall f\in\BV, \ \forall x \in \R, & \quad f(x)  = \Itwo\{ \mathrm{D}^2 \{f \} \}(x)  + f(0) + (f(1) - f(0)) x. \label{eq:Irightleft}
	\end{align}
    In particular, $\Itwo$ is a right-inverse of the second-derivative $\D^2$.
    Moreover, any $f \in \BV$ can be uniquely decomposed as 
        \begin{equation} \label{eq:fdecompose-bis}
            \forall x \in \R, \quad f(x) = \Itwo \{w\} (x) + \beta_0 + \beta_1 x,
        \end{equation}
        where $w\in \Radon$, $\beta_0, \beta_1 \in \R$ are given by
        \begin{equation} \label{eq:walphabeta}
            w = \D^2 f, \quad \beta_0 = f(0), \quad \text{and} \quad \beta_1 = f(1) - f(0) .
        \end{equation}
    \end{proposition}

   \begin{proof}
    
       We fix $x \in\R$. We easily verify that $g(x,y) = 0$ for $\lvert y \rvert \geq \max(1, \lvert x \rvert)$, hence $g(x,\cdot)$ is compactly supported. The function $g(x, \cdot)$ is continuous due to the continuity of $y \mapsto y_+$. Therefore, $g(x,\cdot) \in \Co$ and the duality product $\langle w, g(x,\cdot) \rangle$ is well defined for any $w \in \Radon$ and $x\in \R$.
       
      For $w \in \Radon$ and $x \in \R $, we set $f(x) = \langle w , g(x,\cdot ) \rangle$. We now prove that $\mathrm{D}^2 f = w$ in the distributional sense. First, we prove that $f$ is continuous and is therefore an element of the space of distributions $\mathcal{D}'(\R)$. For any $x,x_0 \in \R$, we have that $|f(x) - f(x_0)| = \left| \int_{\R} (g(x,y) - g(x_0,y) ) \mathrm{d}w (y) \right| \leq \lVert g(x,\cdot) - g(x_0,\cdot) \rVert_\infty \mnorm{w}$, and we easily see from the definition of $g$ in \eqref{eq:kernelI2} that $ \lVert g(x,\cdot) - g(x_0,\cdot) \rVert_\infty \rightarrow 0$ when $x\rightarrow x_0$.
%      To do so, we first prove that $f$ is Lipschitz continuous. As such, it is continuous and bounded by a polynomial (that can be chosen of degree $1$) and is therefore in $\mathcal{S}'(\R)$. It then suffices to show that $\langle \mathrm{D}^2 f , \varphi \rangle = \langle w , \varphi \rangle$ for any $\varphi \in \mathcal{S}(\R)$.
    It then suffices to show that $\langle \mathrm{D}^2 f , \varphi \rangle = \langle w , \varphi \rangle$ for any compactly supported and infinitely smooth test function $\varphi \in \mathcal{D}(\R)$ to deduce that $\mathrm{D}^2 f = w$ in $\mathcal{D}'(\R)$, and that this equality also holds in $\Radon$ since $w \in \Radon$.
    
       From the definition of $g$, denoting by $\partial_x$ the partial derivative with respect to the first variable, we have
       \begin{equation} \label{eq:partix2}
           \partial_x^2 \{g  \} (\cdot ,y) = \delta( \cdot -y). 
       \end{equation}
%       Then, for any $x,y \in \R$, we have that
%       $|f(x) - f(y)| = \left| \int_{\R} (g(x,z) - g(y,z) ) \mathrm{d}w (z) \right| \leq \lVert g(x,\cdot) - g(y,\cdot) \rVert_\infty \mnorm{w}$. Moreover, we easily see from the definition \eqref{eq:kernelI2} of $g$ that $\lVert  g(x,\cdot) - g(y,\cdot) \rVert_\infty \leq 2 |x-y|$, hence $|f(x) - f(y)| \leq 2 \mnorm{w} |x-y|$ and $f$ is Lipschitz continuous. As stated earlier, this proves that $f \in \mathcal{S}'(\R)$. 
       Let $\varphi \in \mathcal{D}(\R)$ and $K$ be its compact support. We have that
       \begin{equation}
           \int_{\R} \int_{\R} |g(x,y)| |\varphi''(x)| \mathrm{d} w(y) \mathrm{d}x  \leq \lVert \varphi'' \rVert_\infty \sup_{x \in K, \ y \in \R} |g(x,y)| \mathrm{Leb}(K) \mnorm{w},
       \end{equation}       
where $\mathrm{Leb}(K)$ is the Lebesgue measure of $K$. We then observe that, for any fixed $x \in \R$, we have $|g(x,y)| \leq |(x-y)_+ - (-y)_+ | + |x| |(1-y)_+ - (-y)_+| \leq |x| + |x| \leq 2|x|$, hence  $\sup_{x \in K, \ y \in \R} |g(x,y)| \leq 2 \sup_{x \in K} |x| < \infty$  and therefore
\begin{equation}
    \label{eq:useFubini}
    \int_{\R} \int_{\R} |g(x,y)| |\varphi''(x)| \mathrm{d} w(y) \mathrm{d}x < \infty.
\end{equation}
Then, we have that 
       \begin{align} \label{eq:trucmarchin}
           \langle \mathrm{D}^2 f , \varphi \rangle
           = \langle f , \mathrm{D}^2 \varphi \rangle
           = \int_{\R} \left( \int_{\R} g(x,y) \mathrm{d}w (y) \right) \varphi''(x) \mathrm{d}x
           = \int_{\R} \left( \int_{\R} \varphi''(x) g(x,y) \mathrm{d} x \right) \mathrm{d}w (y),
       \end{align}
       where the second equality follows from Fubini's theorem with the hypothesis \eqref{eq:useFubini}.
      
       Due to \eqref{eq:partix2}, we then observe that $\int_{\R} \varphi''(x) g(x,y) \mathrm{d} x = \langle \mathrm{D}^2 \varphi , g(\cdot, y) \rangle = \langle \varphi , \partial_x^2 g(\cdot, y) \rangle =   \langle \varphi , \delta(\cdot - y) \rangle = \varphi(y)$. Hence, \eqref{eq:trucmarchin} yields
       $\langle \mathrm{D}^2 f , \varphi \rangle = \int_{\R} \varphi(y) \mathrm{d}w (y) = \langle w , \varphi \rangle$, which proves that $\mathrm{D}^2 f = w$.
       
%       We therefore deduce that
%\begin{equation}
%    \mathrm{D}^2 \{f\}(x) =  \langle    \partial_x^2 \{g  \} (x ,\cdot) , w \rangle = \langle \delta(x-\cdot),   w \rangle = w(x),
%\end{equation}
 %      where we use the slight abuse of notation of keeping the variable $x$ for measures (that may not be defined pointwise) to distinguish when we operate over the first or the second variable.
    Moreover, we have that $g(0,y) = g(1,y) = 0$ for all $y\in \R$, which yields $f(0) = f(1) = 0$.
       From the definition of $\Itwo$, $\Itwo \{ w \}$ is the unique function satisfying these  properties, proving that $\Itwo \{w\} (x) = f (x) = \langle w , g(x,\cdot ) \rangle$ for every $x\in \R$ and $w \in \Radon$. This shows \eqref{eq:I2integral}.
       
        Next, it is clear that $\Itwo$ is linear from $\Radon$ to $\BV$. The continuity of $\Itwo$ follows from the fact that 
       \begin{equation}
           \lVert \Itwo \{w\} \rVert_{\mathrm{BV}^{(2)}} 
           = \lVert \D^2 \Itwo \{w\} \rVert_\mathcal{M} + \sqrt{(\Itwo \{w\} (0))^2 + ((\Itwo \{w\} (1) - \Itwo \{w\} (0))^2} = \lVert w \rVert_\mathcal{M}.
       \end{equation}
       
        The equality $\mathrm{D}^2 \Itwo \{w\} = w$ comes from the definition of $\Itwo \{w\}$. For the right-hand side of~\eqref{eq:Irightleft}, we remark that $\mathrm{D}^2 \{ \Itwo \D^2 \{f\} \} = \D^2 f$ by definition, hence  $ \Itwo \D^2 \{f\} (x) = f(x) + \beta_0 + \beta_1 x$ for every $x\in \R$ and some constants $\beta_0, \beta_1 \in \R$. The equations $ \Itwo \D^2 \{f\} (0) =  \Itwo \D^2 \{f\} (1) = 0$ then specify the constants $\beta_0$ and $\beta_1$, which proves~\eqref{eq:Irightleft}.
    Finally,~\eqref{eq:fdecompose-bis} and~\eqref{eq:walphabeta} can be seen as reformulations of the right equality in~\eqref{eq:Irightleft}. The uniqueness follows from the simple fact that $\D^2 f = w$ determines $f$ when the values of $f(0)$ and $f(1)$ are fixed. 
\end{proof}

\section{Proof of Proposition~\ref{prop:cns-sol-constrained}}\label{app:2}

The forward operator considered in this paper is a sampling operator (the functions $f\in\BV$ are sampled at the locations $x_m\in\RR$ for $m\in\{1,\ldots,M\}$). Let us denote it, for the convenience of the proof, as a linear operator $\nuf:\BV\to\RR^M$ such that
\begin{align}\label{eq:forward-op-nu}
	\forall f\in\BV, \quad \nuf(f)\eqdef(f(x_m))_{1\leq m\leq M}.
\end{align}
%In what follows, the space $\BV$ is endowed with its weak-* topology so that the linear operator $\nuf$ is weak-* continuous (see~\cite[Theorem 1]{unser2018representer}).

The proof of Proposition~\ref{prop:cns-sol-constrained} can be divided in several steps. 
First, we reformulate~\eqref{eq:noiseless} into an equivalent optimization problem thanks to the decomposition of any $f\in\BV$ given by~\eqref{eq:fdecompose}. This is stated in the next lemma.

\begin{lemma}\label{lem:equivalentOP}
The problem~\eqref{eq:noiseless} is equivalent to
\begin{align}\label{eq:constrained-equi}
	\underset{(w, (\beta_0, \beta_1))\in\Radon\times\RR^2}{\min} \  \iota_{\{\obs\}}(\nuu(w) + \beta_0 \un + \beta_1 \x) + \mnorm{w}.
\end{align}
where $\iota_{\{\obs\}}$ is the indicator of the convex set $\{\obs\}$, which is zero at $\obs$ and $+\infty$ elsewhere, and
\begin{align}\label{eq:def-nuu}
	\nuu \eqdef \nuf\circ\Itwo : \Radon \to \RR^M
\end{align}
is the modified forward operator. This equivalence is in the sense that there exists a bijection given by the unique decomposition of any $f\in\BV$ as $f=\Itwo\{w\} + \beta_0 + \beta_1(\cdot)$ with $(w,(\beta_0, \beta_1))\in\Radon\times\RR^2$ (see\eqref{eq:fdecompose}) between the solution sets of both optimization problems.
\end{lemma}

From now on, we consider the equivalent problem~\eqref{eq:constrained-equi} and analyze it using tools from duality theory. The search space $\Radon\times\RR^2$ of this optimization problem is endowed with the weak-* topology, which is defined in terms of its predual space $\Co\times\RR^2$. Using~\eqref{eq:I2integral}, the modified operator $\nuu$ can be expressed as $\nuu(w)=(\dotp{w}{g(x_m,\cdot)})_{1\leq m\leq M}$, where $g(x_m,\cdot)\in\Co$ for all $m\in\{1,\ldots,M\}$ by Proposition~\ref{prop:Itwo}. Since $\Radon$ is the dual of $\Co$, this implies that the linear functional $\nuu:\Radon\to\RR^M$ is weak-* continuous \cite[Theorem IV.20, p. 114]{Reed1980methods}. The adjoint $\nuu^* : \RR^M \to \Co$ of $\nuu$ is thus uniquely defined and is given by
\begin{align}\label{eq:nuu-star}
    \forall \dualvar\in\RR^M, \quad \nuu^*(\dualvar)=\sum_{m=1}^M \dualvarm g(x_m,\cdot),
\end{align}
since $\dotp{w}{\nuu^*(\dualvar)}=\dotp{\nuu(w)}{\dualvar}=\dotp{\pa{\dotp{w}{g(x_m,\cdot)}}_{1\leq m\leq M}}{\dualvar}=\dotp{w}{\sum_{m=1}^M \dualvarm g(x_m,\cdot)}$, for all $w\in\Radon$ and $\dualvar\in\RR^M$.

The second part of the proof consists in determining the dual problem of~\eqref{eq:constrained-equi}, proving that strong duality between the primal and dual problem holds (\ie that the optimal values of both problems are equal and finite) and then deriving the optimality conditions which characterize the solutions of problem~\eqref{eq:constrained-equi}. This is done in the next lemma.

\begin{lemma}
\label{lem:optimality-certificate}
The dual problem of~\eqref{eq:constrained-equi} is given by
\begin{align}\label{eq:constrained-dual}
	&\underset{\dualvar\in\Cc}{\sup} \dotp{\obs}{\dualvar}, \qwithq \Cc\eqdef\{\dualvar\in\RR^M:\ \dotp{\dualvar}{\un}=\dotp{\dualvar}{\x}=0,\ \normi{\nuu^*(\dualvar)}\leq1\}.
 \end{align}

Moreover, it has at least one solution and strong duality holds between problems~\eqref{eq:constrained-equi} and~\eqref{eq:constrained-dual}. Finally, for any $(w,(\beta_0, \beta_1)) \in \Radon \times \RR^2$ and $\V c\in \RR^M$, we have the equivalence between the following statements:
\begin{enumerate}
    \item $(w,(\beta_0, \beta_1))$ is a solution of~\eqref{eq:constrained-equi} and $\V c$ is a solution of~\eqref{eq:constrained-dual}.
    \item $(w,(\beta_0, \beta_1))$ and $\V c$ satisfy the following conditions:
\begin{align}
    &\nuu(w) + \beta_0 \un + \beta_1 \x = \obs,\label{eq:ic}\\
    &\dotp{\dualvar}{\un} = \dotp{\dualvar}{\x} = 0, \quad \mnorm{w} = \dotp{w}{\nuu^*(\dualvar)} \qandq \normi{\nuu^*(\dualvar)}\leq1. \label{eq:oc}
\end{align}
\end{enumerate}
\end{lemma}

\begin{proof}
Let us first obtain the dual problem~\eqref{eq:constrained-dual}. The proof follows the technique of perturbed problems detailed in~\cite[Chapter 3]{Ekeland1976convex}.
\paragraph{Dual problem}Let us write the (primal) problem~\eqref{eq:constrained-equi} as
\begin{align}\label{eq:primal-proof}
	&\underset{(w,(\beta_0, \beta_1))\in\Radon\times\RR^2}{\min} F(w,(\beta_0, \beta_1)) + G(\Lambda(w,(\beta_0, \beta_1))),
	\end{align}
	where $F(w,(\beta_0, \beta_1))\eqdef\mnorm{w}$, $G(\dualvar)\eqdef \iota_{\{\obs\}}(\dualvar)$ for all $\dualvar\in \RR^M$, and $\Lambda(w,(\beta_0, \beta_1)) \eqdef \nuu(w) + \beta_0 \un + \beta_1 \x$.

The functions $F$ and $G$ are convex, lower semi-continuous and not identically equal to $\pm\infty$. By~\cite[Equation (4.18)]{Ekeland1976convex}, the dual problem of~\eqref{eq:primal-proof} is thus given by $\underset{\dualvar\in\RR^M}{\sup} -F^*(\Lambda^*(\dualvar)) - G^*(-\dualvar)$, where $F^*$ and $G^*$ are the Fenchel conjugates of $F$ and $G$ respectively, and $\Lambda^* : \RR^M\to\Co\times\RR^2$ is the adjoint of $\Lambda$. One can check that for all $\dualvar\in\RR^M$, $G^*(\dualvar)=\dotp{\dualvar}{\obs}$, for all $\eta\in\Co$ and $\beta_0, \beta_1 \in \RR$, $F^*(\eta, (\beta_0, \beta_1))=\iota_{\normi{\cdot}\leq1}(\eta) + \iota_{\{(0,0)\}}((\beta_0, \beta_1))$ (with $\iota_{\normi{\cdot}\leq1}$ the indicator function of the closed unit ball in $\Co$ for the uniform norm), and for all $\dualvar\in\RR^M$, $\Lambda^*(\dualvar)=\pa{\nuu^*(\dualvar),(\dotp{\dualvar}{\un},\dotp{\dualvar}{\x})}$. Therefore, the dual problem can be rewritten as
\begin{align}\label{eq:dual-proof}
	&-\underset{\dualvar\in\RR^M}{\inf} \ \iota_\Cc (\dualvar) + \dotp{-\dualvar}{\obs},
\end{align}
where $\Cc\subset\RR^M$ is the convex set defined in~\eqref{eq:constrained-dual}. Problem~\eqref{eq:dual-proof} is clearly the same as problem~\eqref{eq:constrained-dual}, which proves the first statement of the lemma.
\paragraph{Strong duality}To prove strong duality between problems~\eqref{eq:constrained-equi} and~\eqref{eq:constrained-dual} (\ie they have the same optimal value), we start by showing strong duality between
\begin{align}\label{eq:dual-equi-sd}
    \underset{\dualvar\in\RR^M}{\inf} \ \iota_\Cc (\dualvar) + \dotp{-\dualvar}{\obs},
\end{align}
and its dual problem. We then conclude by observing that the optimal value of the dual problem of~\eqref{eq:dual-equi-sd} is equal to the optimal value of problem~\eqref{eq:constrained-equi} up to a sign. Indeed, this last statement proves that both problems~\eqref{eq:constrained-equi} and~\eqref{eq:constrained-dual} have the same optimal value since problem~\eqref{eq:dual-equi-sd} is, up to a sign, the dual problem~\eqref{eq:constrained-dual} (which rewrites as in~\eqref{eq:dual-proof}).

We first start by proving that strong duality holds between problem~\eqref{eq:dual-equi-sd} and its dual problem. The aim is to apply~\cite[Proposition~2.3, Chapter~3]{Ekeland1976convex}. With the notations of~\cite{Ekeland1976convex}, let us denote the map $\Phi:\RR^M\times\Co\to\RR\cup\{+\infty\}$ as
\begin{align}\label{eq:phi-perturbed}
    \forall (\dualvar,\eta)\in\RR^M\times\Co, \quad \Phi(\dualvar,\eta)\eqdef \dotp{-\dualvar}{\obs} + \iota_{\{(0,0)\}}\pa{(\dotp{\dualvar}{\un},\dotp{\dualvar}{\x})} + \iota_{\normi{\cdot}\leq1}(\nuu^*(\dualvar) - \eta).
\end{align}
This map $\Phi$ defines a perturbed problem to problem~\eqref{eq:dual-equi-sd}, since by definition, for all $\dualvar\in\RR^M$,
\eq{
    \Phi(\dualvar,0)=\iota_\Cc (\dualvar) + \dotp{-\dualvar}{\obs}
}
is the objective function of problem~\eqref{eq:dual-equi-sd}. Now let us check that the assumptions of~\cite[Proposition~2.3]{Ekeland1976convex} are satisfied for $\Phi$ and problem~\eqref{eq:dual-equi-sd}:
\begin{itemize}
\item $\Phi$ is convex,
\item the optimal value of problem~\eqref{eq:dual-equi-sd} is finite due to the weak duality (primal-dual inequality given below) between problems~\eqref{eq:primal-proof} and~\eqref{eq:dual-proof}, which yields
\eq{
    -\infty < -\underset{\dualvar\in\RR^M}{\inf} \ \iota_\Cc (\dualvar) + \dotp{-\dualvar}{\obs} \leq \underset{(w,(\beta_0, \beta_1))\in\Radon\times\RR^2}{\inf} \mnorm{w} + \iota_{\{\obs\}}(\nuu(w) + \beta_0 \un + \beta_1 \x) < +\infty,
}
\item the map $\eta\in\Co\mapsto\Phi(\V 0, \eta) = \iota_{\normi{\cdot}\leq1}(-\eta)$ is finite and continuous at $\eta = 0\in\Co$.
\end{itemize}
Therefore, we deduce that strong duality holds between problem~\eqref{eq:dual-equi-sd} and its dual problem given by
\begin{align}\label{eq:bidual}
    \underset{w\in\Radon}{\sup} \ -\Phi^*(\V{0},w),
\end{align}
and that this last optimization problem has at least one solution. By writing the map $\Phi$ as $\Phi(\dualvar,\eta)=\tilde F(\dualvar) + \tilde G(\tilde\Lambda(\dualvar)-\eta)$ with $\tilde F(\dualvar)\eqdef\dotp{-\dualvar}{\obs}+\iota_{V^\perp}(\dualvar)$, $V\eqdef\Span(\un,\x)\subset\RR^M$, $\tilde G\eqdef\iota_{\normi{\cdot}\leq1}(\cdot)$, and $\tilde\Lambda=\nuu^*$,we get that $\Phi^*(\V{c},w) = \tilde{F}^*(\tilde\Lambda^*(w) + \V c)+\tilde G^*(-w)$ for any $(\V c, w) \in \RR^M \times \Radon$, and thus that problem~\eqref{eq:bidual} becomes %$-\min_{w\in\Radon}\tilde{F}^*(\tilde\Lambda^*(w))+\tilde G^*(-w)$, \ie
\begin{align}
\label{eq:bidual-eq}
    -\min_{w\in\Radon} \iota_V(\nuu(w)+\obs)+\mnorm{w}.
\end{align}
We now verify that the optimal value of 
\begin{align}\label{eq:bidual-minus}
    \min_{w\in\Radon} \iota_V(\nuu(w)+\obs)+\mnorm{w},
\end{align}
\ie minus the optimal value of the dual problem of~\eqref{eq:dual-equi-sd} is equal to the optimal value of problem~\eqref{eq:constrained-equi}
\begin{align}
    \underset{(w,(\beta_0, \beta_1))\in\Radon\times\RR^2}{\min} \ \iota_{\{\obs\}}(\nuu(w) + \beta_0 \un + \beta_1 \x) + \mnorm{w}.
\end{align}
Let $w\in\Radon$ be a solution of problem~\eqref{eq:bidual-minus} (which we know to exist by~\cite[Proposition~2.3]{Ekeland1976convex}). Since the objective function of problem~\eqref{eq:bidual-minus} is finite at $w$, we obtain that $\nuu(w)+\obs\in V$, \ie there exists $(\beta_0, \beta_1)\in\RR^2$ such that $\obs = \nuu(-w)+\beta_0 \un + \beta_1 \x$. Assume by contradiction that there exist $(\tilde w, (\tilde \beta_0,\tilde \beta_1)) \in \Radon \times \RR^2$ that achieve a lower cost than $(w,(\beta_0, \beta_1))$ in~\eqref{eq:constrained-equi}, \ie
\eq{
    \iota_{\{\obs\}}(\nuu(-w) + \beta_0 \un + \beta_1\x) + \mnorm{-w} > \iota_{\{\obs\}}(\nuu(\tilde w) + \tilde \beta_0 \un + \tilde \beta_1 \x) + \mnorm{\tilde w}.
}
Since the left term of this inequality in finite, we must have $\obs = \nuu(\tilde w) + \tilde \beta_0 \un + \tilde \beta_1 \x$ and
\begin{align}\label{eq:ineq-mnorm}
    \mnorm{w} > \mnorm{-\tilde w}.
\end{align}
Since $\nuu(-\tilde w)+\obs = \tilde \beta_0 \un + \tilde \beta_1 \x\in V$, we deduce thanks to~\eqref{eq:ineq-mnorm} that $-\tilde w$ achieves a lower cost than $w$ for problem~\eqref{eq:bidual-minus}, which contradicts the assumption on $w$. Hence, for all $w\in\Radon$, $(\beta_0, \beta_1) \in \RR^2$, we have
\eq{
    \iota_{\{\obs\}}(\nuu(-w) + \beta_0 \un + \beta_1 \x) + \mnorm{-w} \leq \iota_{\{\obs\}}(\nuu(w) + \beta_0 \un + \beta_1 \x) + \mnorm{w},
}
\ie $(-w, (\beta_0, \beta_1)) \in \Radon\times\RR^2$ is a solution of problem~\eqref{eq:constrained-equi}. Therefore, we get that the optimal values of problems~\eqref{eq:bidual-minus} and~\eqref{eq:constrained-equi} are equal since
\eq{
    \iota_V(\nuu(w)+\obs)+\mnorm{w} = \iota_{\{\obs\}}(\nuu(-w) + \beta_0 \un + \beta_1 \x) + \mnorm{-w}.
}

%We now verify that the optimal value of problem~\eqref{eq:bidual-eq} is equal to the optimal value of problem~\eqref{eq:constrained-equi}. Let $(w^\ast, (\alpha^\ast, \beta^\ast)) \in \Spc{M}(\R) \times \R^2$ be a solution to problem~\eqref{eq:constrained-equi}. Then, since the latter has a finite optimal cost, we must have $-\nuu(w^\ast)+\obs = \alpha^\ast \V 1 + \beta^\ast \V x \in V$, which implies that $-w^\ast$ reaches a cost in problem~\eqref{eq:bidual} that is lower or equal than the optimal cost of the primal problem problem~\eqref{eq:constrained-equi}. This observation together with the weak duality property applied twice to ~\eqref{eq:constrained-equi} (since~\eqref{eq:bidual} is the bidual of ~\eqref{eq:constrained-equi}) implies that problems~\eqref{eq:dual-equi-sd} and~\eqref{eq:bidual} share the same optimal cost. This proves that strong duality holds between problems~\eqref{eq:constrained-equi} and~\eqref{eq:constrained-dual}, which concludes the proof.
%
\paragraph{Optimality conditions}To derive the optimality conditions given in~\eqref{eq:ic} and~\eqref{eq:oc}, we apply~\cite[Proposition~2.4, Chapter~3]{Ekeland1976convex}. We have already proved that strong duality holds, and that the primal problem~\eqref{eq:constrained-equi} has at least one solution. To apply the proposition, it remains to prove that the dual problem~\eqref{eq:constrained-dual} also has at least one solution. This holds true due to the following
\begin{itemize}
    \item the objective function of problem~\eqref{eq:constrained-dual} is a continuous linear form over the convex set $\Cc$,
    \item the convex set $\Cc=V^\perp\cap\Dd \subset \R^M$ is compact as the intersection of the closed set $V^\perp$ and the compact set $\Dd\eqdef\{\dualvar\in\RR^M: \normi{\nuu^*(\dualvar)}\leq1\}$. The main argument to prove the compactness of $\Dd$ is that $\mbox{Im}(\nuu^*)\subset\Co$ is finite dimensional. Let us prove it in a formal way. Consider the map $F:\RR^M\to\Ff$ given by
    \eq{
        \forall \dualvar\in\RR^M, \quad F(\dualvar)\eqdef\sum_{m=1}^M \dualvarm g(x_m,\cdot)=\nuu^*(\dualvar)
    }
  (using ~\eqref{eq:nuu-star} for the last equality), where $\Ff\eqdef \Span\pa{\{g(x_m,\cdot): 1\leq m\leq M\}}$. Then, $F$ is
    \begin{itemize}
        \item linear;
        \item injective and thus bijective due to the linear independence of the family $(g(x_m,\cdot))_{1\leq m\leq M}$. This independence can be proved by considering that $(g(x_m,\cdot))_{1\leq m\leq M}$ is a family of piecewise-linear splines with each finitely many knots, and so there exists a nonempty interval $I$ in which all the $g(x_m,\cdot)$ are linear functions;
        \item continuous with $\Ff\subset\Co$ endowed with the uniform norm $\normi{\cdot}$.
    \end{itemize}
    Therefore, by the bounded inverse theorem, $F^{-1}$ is continuous. Moreover, note that $\Ee\eqdef\{f\in\Ff: \normi{f}\leq 1\}$ is bounded and closed, and is thus compact (since $\Ff=\mbox{Im}(\nuu^*)$ is finite dimensional). This proves that $\Dd=F^{-1}(\Ee)$ is compact.
    %\item the convex set $\Cc=V^\perp\cap\{\dualvar\in\RR^M: \normi{\nuu^*(\dualvar)}\leq1\} \subset \R^M$ is compact as the intersection of the closed set $V^\perp$ and the compact set $\{\dualvar\in\RR^M: \normi{\nuu^*(\dualvar)}\leq1\}$. The latter is indeed closed since it is convex and weak-* closed as the pre-image of the weak-* compact unit ball of $\Co$ with the weak-* continuous mapping $\nuu^*:\RR^M\to\Co$. It is moreover bounded, otherwise one could construct a sequence $(\dualvar^n)_{n\in \N}$ of vectors in $\R^M$ going to infinity with $\lVert \nuu^*(\dualvar^n) \rVert_\infty = \lvert \sum_{m=1}^M c^n_m g(x_m,\cdot) \rVert_\infty \leq 1$, which is impossible by simple support considerations. Being closed and bounded, $\{\dualvar\in\RR^M: \normi{\nuu^*(\dualvar)}\leq1\}$ is compact as expected.
    
%    The last one is compact because it is closed since weak-* closed (due to the weak-* continuity of $\nuu^*:\RR^M\to\Co$) and convex, and bounded since otherwise there would be a sequence in $\RR^M$ going to infinity which would contradict the condition $\normi{\nuu^*(\dualvar)}\leq1$ because $\nuu^*(\dualvar)=\sum_{m=1}^M c_m g(x_m,\cdot)$ (where $g$ is defined in~\eqref{eq:kernelI2}),
\end{itemize}
The convexity and the compactness of $\Cc$ imply that  there is at least one extreme point of $\Cc$ that is a solution of problem~\eqref{eq:constrained-dual}. 
%the objective function of problem~\eqref{eq:constrained-dual} is a continuous linear form over the convex set $\Cc$. Moreover, the convex set $\Cc$ is compact since the set $\{\dualvar\in\RR^M: \normi{\nuu^*(\dualvar)}\leq1\}$ is weak-* closed (due to the weak-* continuity of $\nuu^*:\RR^M\to\Co$), which implies that it is closed and it is bounded since $\nuu^*(\dualvar)=\sum_{m=1}^M c_m g(x_m,\cdot)$ (where $g$ is defined in~\eqref{eq:kernelI2}).
Hence, the assumptions of~\cite[Proposition~2.4, Chapter~3]{Ekeland1976convex} are satisfied, which implies that any solution $(w,(\beta_0, \beta_1))\in\Radon\times\RR^2$ of (the primal) problem~\eqref{eq:constrained-equi} and $\dualvar\in\RR^M$ of (the dual) problem~\eqref{eq:constrained-dual} are linked by the optimality conditions
\begin{align}
    &\nuu(w) + \beta_0 \un + \beta_1 \x = \obs,\\
    &\dotp{\dualvar}{\un} = \dotp{\dualvar}{\x} = 0, \quad \mnorm{w} = \dotp{w}{\nuu^*(\dualvar)} \qandq \normi{\nuu^*(\dualvar)}\leq1.
\end{align}
Conversely, if any $(w,(\beta_0, \beta_1))\in\Radon\times\RR^2$ and $\dualvar\in\RR^M$ satisfy the optimality conditions given above, then again by~\cite[Proposition~2.4, Chapter~3]{Ekeland1976convex} we obtain that $(w,(\beta_0, \beta_1)) \in \Radon \times \RR^2$ and $\dualvar\in\RR^M$ are solutions of the primal and dual problems respectively. This proves the last statement of the lemma.
\end{proof}

The last intermediate result needed for the proof of Proposition~\ref{prop:cns-sol-constrained} is given in the next lemma, where we prove that any continuous function $\nuu^*(\dualvar)\in\Co$ with $\dualvar\in\RR^M$ satisfying the orthogonality conditions given in~\eqref{eq:oc} is a piecewise-linear spline whose knots are located at the sampling points $\x=(x_m)_{1\leq m\leq M}$.

\begin{lemma}\label{lem:struct-certif}
Let $\dualvar\in\RR^M$ such that $\dotp{\dualvar}{\un} = \dotp{\dualvar}{\x} = 0$. Then, we have
$\nuu^*(\dualvar) = \sum_{m=1}^M \dualvarm \green{x_m - \cdot}$.
\end{lemma}

\begin{proof}
We know by~\eqref{eq:nuu-star} and~\eqref{eq:kernelI2} that
\begin{align}
	\nuu^*(\dualvar) &= \dotp{\dualvar}{\pa{g(x_m,\cdot)}_{1\leq m\leq M}}, \\
	                 &= \dotp{\dualvar}{\pa{\green{x_m-x} - (-x)_+ + x_m((-x)_+ - (1 - x)_+)}_{1\leq m\leq M}}, \\
	                 &= \dotp{\dualvar}{\pa{\green{x_m-x}}_{1\leq m\leq M}} - (-x)_+ \underbrace{\dotp{\dualvar}{\un}}_{=0} + ((-x)_+ - (1 - x)_+)\underbrace{\dotp{\dualvar}{\x}}_{=0},
\end{align}
which proves that $\nuu^*(\dualvar) = \sum_{m=1}^M \dualvarm \green{x_m - \cdot}$.
\end{proof}

We can now prove Proposition~\ref{prop:cns-sol-constrained}.

%\begin{proofof}
\begin{proof}{Proposition~\ref{prop:cns-sol-constrained}}
Suppose that $\fopt\in\BV$ is a solution of~\eqref{eq:noiseless}. Then, $\fopt$ satisfies the interpolation conditions $\fopt(x_m)=y_{0,m}$ for all $m\in\{1,\ldots,M\}$, and $(w,(\beta_0, \beta_1)) \in \Radon \times \RR^2$ is a solution of problem~\eqref{eq:constrained-equi} where $\fopt=\Itwo\{w\} + \beta_0 + \beta_1 (\cdot)$. By Lemma~\ref{lem:optimality-certificate}, there exists a $\V c\in\RR^M$ solution of problem~\eqref{eq:constrained-dual} which then satisfies $\dotp{\dualvar}{\un} = \dotp{\dualvar}{\x} = 0$ with $\normi{\nuu^*(\dualvar)}\leq1$. Let us denote $\eta \eqdef \nuu^*(\dualvar)\in\Co$. By Lemma~\ref{lem:struct-certif}, we have $\eta = \sum_{m=1}^M \dualvarm \green{x_m - \cdot}$ \ie $\eta$ is a dual pre-certificate (Definition~\ref{def:dual-certif}). Moreover, again by Lemma~\ref{lem:optimality-certificate}, we know that $\mnorm{w} = \dotp{w}{\eta}$ which gives the direct implication.

For the reverse implication, the dual pre-certificate $\eta$ given by the statement satisfies $\eta=\nuu^*(\dualvar)$ by Lemma~\ref{lem:struct-certif}, and since $\fopt$ satisfies the interpolation conditions, we deduce that $\nuu(w) + \beta_0 \un + \beta_1 \x = \obs$ where $\beta_0$ and $\beta_1$ are defined thanks to the relation $\fopt=\Itwo\{w\} + \beta_0 + \beta_1 (\cdot)$. Hence, by Lemma~\ref{lem:optimality-certificate}, $(w,(\beta_0, \beta_1))\in\Radon\times\RR^2$ is a solution of problem~\eqref{eq:constrained-equi} (and $\V c$ is a solution of problem~\eqref{eq:constrained-dual}), \ie $\fopt$ is a solution of ~\eqref{eq:noiseless}.

Let us now prove that the relation $\mnorm{w} = \dotp{w}{\eta}$ is equivalent to $\ssupp w \subset \ssat \eta$ when $\eta$ is a dual pre-certificate (see Definition~\ref{def:supp-sat} for the definition of the signed support and signed saturation set). First, we have that $\mnorm{w} = \mnorm{w_{|\satp{\eta}}} + \mnorm{w_{|\satn{\eta}}} + \mnorm{w_{|S^c}}$ (see~\cite[Theorem~6.2]{rudin1986real}), where $S\eqdef\satp\eta\cup\satn\eta$, hence
\begin{align}
	\pa{\mnorm{w_{|\satp{\eta}}} - \dotp{w_{|\satp{\eta}}}{\eta}} + \pa{\mnorm{w_{|\satn{\eta}}} - \dotp{w_{|\satn{\eta}}}{\eta}} + \pa{\mnorm{w_{|S^c}} - \dotp{w_{|S^c}}{\eta}} = 0.
\end{align}
Each of the three terms in the sum is nonnegative by definition of $\mnorm{\cdot}$, and the fact that $\normi{\eta}\leq 1$, so that the equality $\mnorm{w} = \dotp{w}{\eta}$ is equivalent to
\begin{align}
	&\mnorm{w_{|\satp{\eta}}} = \dotp{w_{|\satp{\eta}}}{\eta},\label{eq:mnorm-plus}\\
	&\mnorm{w_{|\satn{\eta}}} = \dotp{w_{|\satn{\eta}}}{\eta},\label{eq:mnorm-neg}\\
	&\mnorm{w_{|S^c}} = \dotp{w_{|S^c}}{\eta}\label{eq:mnorm-comp}.
\end{align}
Consider the Jordan decomposition of $w$: $w=w_+ - w_-$. Then $\mnorm{w_{|\satp{\eta}}}=w_+\pa{\satp\eta}+w_-\pa{\satp\eta}$ and $\dotp{w_{|\satp{\eta}}}{\eta}=\int_{\satp\eta}\d w=w_+\pa{\satp\eta}-w_-\pa{\satp\eta}$, so that~\eqref{eq:mnorm-plus} is equivalent to $w_-\pa{\satp\eta} = 0$ \ie
\eq{
	\supp(w_-)\cap\satp\eta=\varnothing.
}
Similarly, we can prove that~\eqref{eq:mnorm-neg} is equivalent to
\eq{
	\supp(w_+)\cap\satn\eta=\varnothing,
}
since $\dotp{w_{|\satn{\eta}}}{\eta}=-\int_{\satn\eta}\d w$. As a result, to obtain the desired equivalence, it remains to prove that~\eqref{eq:mnorm-comp} is the same as $w_{|S^c}=0$. The arguments can be found for example in~\cite{de2012exact} (see the proof of Lemma A.1), but we reproduce the reasoning here for the sake of completeness. Consider the closed sets for all $k>0$
\eq{
	\Omega_k \eqdef \RR \setminus \pa{ S \ + \ \left( -\frac{1}{k},\frac{1}{k} \right) }\subset S^c.
}
Suppose by contradiction that there exists $k>0$ such that $\mnorm{w_{|\Omega_k}}>0$. Since $|\eta|<1$ on the closed set $\Omega_k$ (because it is true on the bigger open set $S^c$), we deduce that $\dotp{w_{|\Omega_k}}{\eta}<\mnorm{w_{|\Omega_k}}$ and then
\eq{
	\mnorm{w} = \dotp{w_{|\Omega_k}}{\eta} + \dotp{w_{|\Omega_k^c}}{\eta} < \mnorm{w_{|\Omega_k}} + \mnorm{w_{|\Omega_k^c}} = \mnorm{w},
}
which is a contradiction. Hence, we have $\mnorm{w_{|\Omega_k}}=0$ for all $k>0$, which yields $\mnorm{w_{|S^c}}=0$ since $S^c = \cup_{k>0} \Omega_k$, \ie $w_{|S^c}=0$.
\end{proof}

\section{Proof of Proposition~\ref{prop:cns-sol-constrained-fixed-dual-certif}}\label{app:2bis} 
 
The proof of Proposition~\ref{prop:cns-sol-constrained-fixed-dual-certif} is very similar to the proof of Proposition~\ref{prop:cns-sol-constrained}, and is derived from the optimality conditions given in Lemma~\ref{lem:optimality-certificate}.

%\begin{proofof}
\begin{proof}{Proposition~\ref{prop:cns-sol-constrained-fixed-dual-certif}}

Let $\eta$ be a dual certificate in the sense of Proposition~\ref{prop:cns-sol-constrained-fixed-dual-certif}. By definition of $\eta$ (it is in particular a dual pre-certificate in the sense of Definition~\ref{def:dual-certif}) and by Lemma~\ref{lem:struct-certif}, there exists $\V c\in\RR^M$ such that $\eta=\nuu^*(\dualvar)$ and $\dotp{\dualvar}{\un} = \dotp{\dualvar}{\x} = 0$. Since $\eta$ is a dual certificate, Proposition~\ref{prop:cns-sol-constrained} implies that there exists a $\tilde f\in\BV$ satisfying the interpolation conditions and such that $\mnorm{\Op D^2 \tilde f}=\dotp{\Op D^2 \tilde f}{\eta}$. This implies that $\V c$ and $(\tilde w, (\tilde \beta_0, \tilde \beta_1))\in\Radon\times\RR^2$, where $\tilde f=\Itwo\{\tilde w\}+\tilde \beta_0 + \tilde \beta_1 (\cdot)$, satisfy~\eqref{eq:ic} and~\eqref{eq:oc} \ie in particular $\V c$ is a solution of the dual problem~\eqref{eq:constrained-dual} by Lemma~\ref{lem:optimality-certificate}. Using this fixed vector $\V c\in\RR^M$ and the decomposition of any $f\in\BV$ as $f=\Itwo\{w\} + \beta_0 + \beta_1 (\cdot)$ (see~\eqref{eq:fdecompose}), the equivalence in Lemma~\ref{lem:optimality-certificate} directly yields that $\fopt$ is a solution of~\eqref{eq:noiseless} if and only if $\fopt$ satisfies the interpolation conditions $\fopt(x_m) = y_{0,m}$ and $\mnorm{\Op{D}^2 \fopt} = \dotp{\Op{D}^2 \fopt}{\eta}$, which concludes the proof.
\end{proof}

\section{Proof of Proposition~\ref{prop:general-uniqueness}}\label{app:3}

Let $\fopt\in\BV$ be a solution of problem~\eqref{eq:noiseless} given by Theorem~\ref{theo:RTweneed}. By~\eqref{eq:fdecompose}, there exist $w\in\Radon$ and $(\beta_0, \beta_1) \in \RR^2$ such that $\fopt=\Itwo\{w\} + \beta_0 + \beta_1 (\cdot)$. By the assumption of the proposition, there exists a nondegenerate dual certificate $\eta$, so that by applying Proposition~\ref{prop:cns-sol-constrained-fixed-dual-certif}, we obtain $\ssupp{w} \subset \ssat \eta$. Moreover, we have that $\ssat \eta \subset \{x_2,\ldots,x_{M-1}\}$ due to the two following facts
\begin{itemize}
    \item $\eta = \sum_{m=1}^M \dualvarm \green{x_m - \cdot}$ (as a dual pre-certificate, see Lemma~\ref{lem:struct-certif}),
    \item $\ssat \eta$ is a discrete set (as $\eta$ is nondegenerate).
\end{itemize}
This implies that $\eta$ must be equal to $\pm 1$ at the points $\{x_2,\ldots,x_{M-1}\}$, which yields
\eq{
    w = \sum_{k=2}^{M-1} a_k\dirac{x_k},
}
where the $a_k \in \R$ are (possibly zero) weights. In particular, this implies that $\fopt$ is a piecewise-linear spline with at most $(M-2)$ knots that are a subset of $\{x_2,\ldots,x_{M-1}\}$. It remains to prove that the coefficients $a_2,\ldots,a_{M-1}, \beta_0, \beta_1$ are uniquely determined to conclude that $\fopt$ is the unique solution of~\eqref{eq:noiseless}.

Since $\fopt$ is a solution of~\eqref{eq:noiseless}, we have that $\nuf(\fopt) = \obs$. This implies that
\eq{
    \sum_{k=2}^{M-1} a_k \mathbf{g}_k + \beta_0 \un + \beta_1 \x = \obs \qwithq \mathbf{g}_k\eqdef \nuu\pa{\dirac{x_k}}=\pa{g(x_m,x_k)}_{1\leq m\leq M}\in\RR^M.
}
We now prove that this equation uniquely determines the coefficients $a_2,\ldots,a_{M-1},\beta_0, \beta_1$ by showing that the family $(\un,\x,\mathbf{g}_2,\ldots,\mathbf{g}_{M-1})$ is a basis of $\RR^M$. Indeed, by definition of $g$ (see~\eqref{eq:kernelI2}), we have that
\begin{align}\label{eq:gk}
    \forall k\in\{2,\ldots,M-1\}, \quad \V g_k = \pa{(x_m - x_k)_+}_{1\leq m \leq M} - (-x_k)_+\un + \pa{(-x_k)_+ - (1-x_k)_+}\x.
\end{align}
Hence, by writing the matrix of the family $(\un,\x,\mathbf{g}_2,\ldots,\mathbf{g}_{M-1})$ in the canonical basis of $\RR^M$, subtracting thanks to~\eqref{eq:gk} appropriate linear combinations of the first two columns (given by the vectors $\un$ and $\x$) to all of the other columns and finally subtracting $x_1$ times the first column to the second one, we end up with the following matrix
\eq{
    \begin{pmatrix}
    1 & 0 & 0 & 0 & \hdots & 0\\
    1 & (x_2 - x_1) & 0 & 0 & \hdots & 0\\
    1 & (x_3 - x_1) & (x_3 - x_2) & 0 & \hdots & 0\\
    \vdots & \vdots & \vdots & \vdots & \ddots & \vdots\\
    1 & (x_M - x_1) & (x_M - x_2) & (x_M - x_3) & \hdots & (x_M - x_{M-1})\\
    \end{pmatrix}.
}
The latter is a lower triangular matrix with nonzero coefficients on the diagonal (as the sampling points $x_m$ are pairwise distinct), and is thus invertible, which proves the desired result.

\section{Proof of Theorem~\ref{thm:limitdomain}}
\label{app:limitdomain}

Let $\fopt \in\mathcal{V}_0$. 
We fix $m \in \{2, M-2\}$. First of all, as we have seen in the proof of Theorem~\ref{theo:sol_set}, if $a_m a_{m+1} \leq 0$, then $\fopt = \fcano$ on $[x_m, x_{m+1}]$, and the graph of $\fopt$ in this interval is equal to the one of $\fcano$. Assume now that $a_m a_{m+1} > 0$. We now show that $\{ (x, \fopt(x)): x\in [x_m, x_{m+1} ] \} \subset \Delta_m$. The slope condition $a_m a_{m+1} > 0$ implies that $\etacano$ is degenerate and that $\etacano = \pm 1$ is constant over $[x_m, x_{m+1}]$. Assume for instance that the value is $1$, in which case $\fopt$ is convex over $[x_{m-1},x_{m+2}]$ according to Theorem~\ref{theo:sol_set}. 

We shall use the following well-known fact on convex functions. Fix $a < b < c$ and assume that $f$ is convex over $[a,c]$. Then, $f$ is below its arc between $a$ and $b$ on $(a,b)$, that is, $f(x) \leq \frac{f(b)-f(a)}{b-a} (x-a) + f(a)$ for any $x \in (a,b)$. Moreover, $f$ is above the same arc 
over $(b,c)$, that is,  $f(x) \geq \frac{f(b)-f(a)}{b-a} (x-a) + f(a)$ for any $x \in (b,c)$.

Let $x^* \in [x_m, x_{m+1}]$. By convexity, $\fopt$ is below its arc between $x_m$ and $x_{m+1}$. %whose equation is $ y = \frac{y_{0,m+1} - y_{0,m}}{x_{m+1} - x_m} (x- x_m) + y_{0,m}$. 
Hence we have that 
\begin{equation}\label{eq:condition1fordetalm}
    \fopt(x^*) \leq \frac{y_{0,m+1} - y_{0,m}}{x_{m+1} - x_m} (x^*- x_m) + y_{0,m}.
\end{equation}
Moreover, the convexity over $[x_{m-1}, x^*]$ implies that $\fopt(x^*)$ is above the arc of $\fopt$ between $x_{m-1}$ and $x_m$.
%, whose equation is $y = \frac{y_{0,m} - y_{0,m-1}}{x_{m} - x_{m-1}} (x - x_{m-1}) + y_{0,m-1}$.
This implies that 
\begin{equation}
    \label{eq:condition2fordetalm}
    \fopt(x^*) \geq \frac{y_{0,m} - y_{0,m-1}}{x_{m} - x_{m-1}} (x^* - x_{m-1}) + y_{0,m-1}.
\end{equation}
A similar argument over $[x^*, x_{m+2}]$ implies that
\begin{equation}
    \label{eq:condition3fordetalm}
    \fopt(x^*) \geq \frac{y_{0,m+2} - y_{0,m+1}}{x_{m+2} - x_{m+1}} (x^*- x_{m+1}) + y_{0,m+1}.
\end{equation}
The conditions~\eqref{eq:condition1fordetalm},~\eqref{eq:condition2fordetalm}, and~\eqref{eq:condition3fordetalm} are precisely equivalent to $(x^*,\fopt(x^*)) \in \Delta_m$, since the three linear equations delineate this domain in this case. The same proof applies when $\etacano=-1$ over $[x_m, x_{m+1}]$ by using concavity instead of convexity. This proves that $\mathcal{G}(\fopt) \subset \mathcal{G}(\fcano) \cup \left( \cup_{m \in \mathcal{X}}  \Delta_m \right)$ for every $\fopt \in \Spc{V}_0$, and hence the direct inclusion in~\eqref{eq:domain}.\\

For the reverse inclusion, we already know that $\fcano \in \mathcal{V}_0$, therefore it suffices to show that, for any $m \in \mathcal{X}$ and any $(x^*,y^*) \in \Delta_m$, there exists a solution $\fopt \in \mathcal{V}_0$ such that $\fopt (x^*) = y^*$. As before, since $m\in \mathcal{X}$, we know that $\etacano =\pm 1$ on $[x_m,x_{m+1}]$ and we can assume without loss of generality that the value is $1$. Then, any solution is convex and satisfies the relations~\eqref{eq:condition1fordetalm},~\eqref{eq:condition2fordetalm}, and~\eqref{eq:condition3fordetalm}. By convexity of $\mathcal{V}_0$, it suffices to show the result for $(x^*,y^*)$ in the boundary of $\Delta_m$, which is delimited by the relations
\begin{align}
    &\frac{y_{0,m+1} - y_{0,m}}{x_{m+1} - x_m} (x^*- x_m) + y_{0,m} = y^* \text{, or}  \label{eq:line1}\\ 
    &\frac{y_{0,m} - y_{0,m-1}}{x_{m} - x_{m-1}} (x^*- x_{m-1}) + y_{0,m-1} = y^* \text{, or} \label{eq:line2} \\
   & \frac{y_{0,m+2} - y_{0,m+1}}{x_{m+2} - x_{m+1}} (x^*- x_{m+1}) + y_{0,m+1} = y^*. \label{eq:line3}
\end{align}
The solution $\fcano$ is such that $\fcano(x^*) = \frac{y_{0,m+1} - y_{0,m}}{x_{m+1} - x_m} (x^*- x_m) + y_{0,m}  = y^*$, hence any $(x^* , y^*)$ satisfying~\eqref{eq:line1} is attained by a solution (the canonical one) in $\mathcal{V}_0$.
Assume that $(x^*,y^*)$ satisfies~\eqref{eq:line2} (the case of~\eqref{eq:line3} follows the same argument). We construct $\fopt$ as follows. First, $\fopt (x)= \fcano(x)$ for any $x \notin (x_m, x_{m+1})$. Then,  we set
\begin{equation}
\fopt(x) = \frac{y_{0,m} - y_{0,m-1}}{x_{m} - x_{m-1}} (x- x_{m-1}) + y_{0,m-1}
\end{equation} 
for $x \in (x_m, x^*]$.
In particular, $f(x^*) = y^*$, and $\fopt$ is linear on $[x_{m}, x^*]$. Finally,  we impose that
$\fopt$ is linear on $[x^*,x_{m+1}]$, which is equivalent to the relation
\begin{equation}
\fopt(x) = \frac{y_{0,m+1}-  y^*}{x_{m+1} - x^*} (x- x^*) + y^*   
\end{equation}
for any $x \in [x^*,x_{m+1}]$. We then claim that $\fopt \in \mathcal{V}_0$, the argument being very similar to the one of Lemma \ref{lem:mountain}. Indeed, to show this, it suffices to remark that $\fopt$, which is piecewise-constant and coincides with $\fcano$ outside of $(x_m,x_{m+1})$, is convex on $[x_{m-1}, x_{m+2}]$ (this is guaranteed by the slope condition $a_m a_{m+1} > 0$ and the construction of $\fopt$). According to Theorem~\ref{theo:sol_set}, this implies that $\fopt \in \mathcal{V}_0$, with $\fopt (x^*) = y^*$. This finally shows that $(x^*,y^*) \in \cup_{\fopt \in \mathcal{V}_0} \mathcal{G}(\fopt)$, which proves~\eqref{eq:domain}.

\section{Proof of Theorem~\ref{thm:sparsest_sol}}
\label{sec:sparsest_solutions_proof}

Using Theorem~\ref{theo:sol_set}, for any $\fopt \in \Spc{V}_0$, we have $\fopt (x) = \fcano(x)$ for any $x$ such that $\etacano(x) \neq \pm 1$. We now focus on regions where $\etacano(x) = \pm 1$. For all $n \in \{1, \ldots, N_s\}$, $\fcano$ has $\alpha_n + 1$ knots in the interval $[x_{s_n}, x_{s_n + \alpha_n}]$. In order to construct one of the sparsest solutions, we must therefore replace these $\alpha_n+1$ knots with as little knots as possible in each saturation region, since all solutions must coincide with $\fcano$ outside these regions. In order to lighten the notations, in what follows, we focus on a single saturation region determined by a fixed $n \in \{1, \ldots, N_s\}$ and we write $\alpha \eqdef \alpha_n$ and $s \eqdef s_n$.

Similarly to the proof of Proposition~\ref{prop:uniqueness}, a piecewise-linear spline $f$ that coincides with $\fcano$ outside the interval $[x_{s}, x_{s + \alpha}]$ must be of the form
\begin{align}
\label{eq:sparsest_sol}
    f(x) = \fcano(x) - \sum_{n'=0}^{\alpha} a_{s + n'} (x - x_{s + n'})_+ + \sum_{p=1}^{P} \tilde{a}_{p} (x - \tilde{\tau}_{p})_+,
\end{align}
where $\tilde{a}_{p}\in \R$, $\tilde{\tau}_p \in [x_{s}, x_{s + \alpha}]$ such that $\tilde{\tau}_1 < \cdots < \tilde{\tau}_P$ and $P$ is the number of knots of $f$ in this interval. We then prove the following lemma.
\begin{lemma}
\label{lem:min_sparsity}
If $f$ in~\eqref{eq:sparsest_sol} satisfies the constraints $f(x_m)=y_{0, m}$ for all $m\in\{1,\ldots,M\}$, then the number of knots $P$ in $[x_{s}, x_{s + \alpha}]$ satisfies $P \geq \lceil\frac{\alpha+1}{2} \rceil$.
\end{lemma}
\begin{proof}
Lemma~\ref{lem:min_sparsity} is trivially true for $\alpha = 0$, since we must have $f = \fcano$ and thus $P=1$. Assume now that $\alpha >0$. Firstly, we show that we must have $\tilde{\tau}_1 \in [x_{s}, x_{s + 1})$. Assume by contradiction that $\tilde{\tau}_1 \geq x_{s+1}$: then, $f$ has no knots in the interval $(x_{s-1}, x_{s + 1})$. Yet $f$ must satisfy the interpolation constraints $f(x_m)=y_{0,m}$ for all $m\in\{1,\ldots,M\}$, which implies that the points $\P0{s-1}$, $\P0{s}$, and $\P0{s+1}$ are aligned. Therefore, $\fcano$ has a weight $a_s = 0$ (defined in~\eqref{eq:a_coefs}) which implies that $\etacano(x_s)=0$, which contradicts the assumption $\etacano(x_s)= \pm 1$. We can then prove in a similar fashion that $\tilde{\tau}_P \in (x_{s+ \alpha - 1}, x_{s + \alpha}]$ when $\alpha > 1$.

Next, we show that for $\alpha \geq 2$, we have
\begin{align}
\label{eq:min_knots_central_sat}
 \forall n' \in \{1 , \ldots, \alpha-1\},\ \exists p \in \{1, \ldots, P\} \text{ such that }  \tilde{\tau}_p \in (x_{s+n'-1}, x_{s+n'+1}),
\end{align}
\ie there must be a knot in all blocks of two consecutive saturation intervals. We assume by contradiction that this is not the case. Similarly to above, this implies that $\P0{s+n'-1}$, $\P0{s+n'}$, and $\P0{s+n'+1}$ are aligned and thus that $\etacano(x_{s+n'})=0$, which yields a contradiction.

Lemma~\ref{lem:min_sparsity} immediately follows from the constraints $\tilde{\tau}_1 \in [x_{s}, x_{s + 1})$ and $\tilde{\tau}_P \in [x_{s+ \alpha - 1}, x_{s + \alpha}]$ for $\alpha \leq 2$. For $\alpha >2$, by the two aforementioned constraints, $f$ must have at least two knots in the first and last saturation intervals $[x_{s}, x_{s + 1})$ and $(x_{s+ \alpha - 1} x_{s + \alpha}]$ respectively. Next, consider the interval $[x_{s+1}, x_{s +\alpha-1}]$, which consists of the central $\alpha-2$ consecutive saturations. Using~\eqref{eq:min_knots_central_sat}, this interval must contain at least $\lfloor \frac{\alpha-2}{2} \rfloor$ knots, which yields the lower bound $P \geq 2 + \lfloor \frac{\alpha-2}{2} \rfloor = \lceil\frac{\alpha+1}{2} \rceil$ (the last equality can easily be verified for every $\alpha \in \N$).
\end{proof}
The following Lemma then states that the bound in Lemma~\ref{lem:min_sparsity} is tight.
\begin{lemma}
\label{lem:construction_sparsest_sol}
The lower bound in Lemma~\ref{lem:min_sparsity} is always reached, \ie there exists a piecewise-linear spline $\fopt \in \Spc{V}_0$ of the form~\eqref{eq:sparsest_sol} with $P = \lceil\frac{\alpha+1}{2} \rceil$ knots in $[x_s, x_{s+\alpha}]$. If $\alpha$ is odd or $\alpha=0$, then $\fopt $ is unique. If $\alpha > 0$ is even, then there are uncountably many such functions $\fopt$.
\end{lemma}
\begin{proof}
Lemma~\ref{lem:construction_sparsest_sol} is trivially true for $\alpha = 0$, \ie when no saturation occurs. Indeed, the saturation interval is then reduced to the point $\{ x_{s} \}$, and the only solution $\fopt \in \Spc{V}_0$ of the form~\eqref{eq:sparsest_sol} is $\fopt = \fcano$ for which $P=1$.

Assume now that $\alpha = 2k + 1$ is odd. The bound in Lemma~\ref{lem:min_sparsity} then reads $P\geq k+1$. Similarly to the proof of Proposition~\ref{prop:uniqueness}, we construct a function $\fopt$ of the form~\eqref{eq:sparsest_sol} with $P=k+1$ and
\begin{align}
\label{eq:sparsest_sol_odd}
\begin{cases}
      \tilde{a}_1 \eqdef a_{s} + a_{s+1} \text{ and } \tilde{\tau}_1 \eqdef \frac{a_{s}x_{s} + a_{s+1}x_{s+1}}{\tilde{a}_1}; \\
      \tilde{a}_2 \eqdef a_{s+2} + a_{s+3} \text{ and } \tilde{\tau}_2 \eqdef \frac{a_{s+2}x_{s+2} + a_{s+3}x_{s+3}}{\tilde{a}_2}; \\
      \vdots \\
      \tilde{a}_{k+1} \eqdef a_{s+2k} + a_{s+2k+1} \text{ and } \tilde{\tau}_k \eqdef \frac{a_{s+2k}x_{s+2k} + a_{s+2k+1}x_{s+2k+1}}{\tilde{a}_{k+1}}.
\end{cases}
\end{align}
Since the $a_s, \ldots, a_{s+\alpha}$ all have the same (nonzero) sign, the $\tilde{\tau}_{i}$, $i=1, \ldots , k+1$, are all barycenters with positive weights, which implies that $\tilde{\tau}_{i} \in (x_{s+2i}, x_{s+2i+1})$. Then, as in the proof of Proposition~\ref{prop:uniqueness}, replacing the knots at $x_{s+2i}$ and $x_{s+2i+1}$ in $\fcano$ by a single knot at $\tilde{\tau}_{i}$ does not change the expression of $\fopt $ outside the interval $(x_{s+2i}, x_{s+2i+1})$, which implies that all the constraints $\fopt(x_m) = y_{0, m}$ for all $m\in\{1,\ldots,M\}$ are satisfied.

Next, let $I_s = \{1, \ldots M\} \setminus \{s, \ldots, s+\alpha\}$ be the set of indices outside our interval of interest. Since $a_s, \ldots, a_{s+\alpha}$ and thus $\tilde{a}_1, \ldots, \tilde{a}_{k+1}$ all have the same sign, we have $\Vert \Op{D}^2 \fopt \Vert_{\Spc{M}} = \sum_{m \in I_s} \vert a_m \vert + \vert \sum_{i=1}^{k+1} \tilde{a}_{i} \vert= \sum_{m \in I_s} \vert a_m \vert + \vert\sum_{n=0}^{\alpha} a_{s+n}\vert = \Vert \Op{D}^2 \fcano \Vert_{\Spc{M}}$, which together with the interpolation constraints implies that $\fopt  \in \Spc{V}_0$.

To show the uniqueness, consider once again a function $\fopt $ of the form~\eqref{eq:sparsest_sol} with $P=k+1$ and $\tilde{\tau}_1 < \cdots < \tilde{\tau}_{k+1}$. We then invoke Lemma~\ref{lem:min_sparsity}, which stipulates that there must be knots in the first and last saturation intervals as well as every two consecutive saturation intervals. The only way to achieve this is to have $\tilde{\tau}_i \in (x_{s+2i}, x_{s+2i+1})$, $i=0, \ldots ,k$. The intervals $(x_{s+2i-1}, x_{s+2i})$ for all $i\in\{1, \ldots, k\}$ thus have no knots, which implies that in these intervals, $\fopt$ must follow the line $(\P0{s+2i-1}, \P0{s+2i})$. The knots are then necessarily the intersection of these lines, which yields the solution given in~\eqref{eq:sparsest_sol_odd}. The latter is therefore the unique function in $\Spc{V}_0$ with $P=k+1$ knots in the interval $[x_s, x_{s+\alpha}]$. An example of such a sparsest solution is shown in Figure~\ref{fig:3_sat} with $M=6$ and $\alpha = 3$ consective saturation intervals.

Assume now that $\alpha = 2k$ is even, with $k>0$. The bound in Lemma~\ref{lem:min_sparsity} then reads $P\geq k+1$. By Lemma~\ref{lem:mountain}, the intersection $\widetilde{\mathrm{P}} = \Point{\tilde{\tau}}{\tilde{y}}$ between the lines $(\P0{s-1}, \P0{s})$ and $(\P0{s+1}, \P0{s+2})$ exists and satisfies $\tilde{\tau} \in (x_s, x_{s+1})$. Then, let $\widetilde{\mathrm{P}}_1 = \Point{\tilde{\tau}_1}{\tilde{y}_1}$ be any point on the line segment $[\P0{s}, \widetilde{\mathrm{P}}]$, \ie with $\tilde{\tau}_1 \in [x_{s}, \tilde{\tau}]$. Then, we define $\widetilde{\mathrm{P}}_{2}$ as the intersection between the lines $(\widetilde{\mathrm{P}}_{1}, \P0{s+1})$ and $(\P0{s+2}, \P0{s+3})$. Similarly, if $\alpha \geq 4$, for every $i\in \{3, \ldots , k+1\}$, we define $\widetilde{\mathrm{P}}_{i} = \Point{\tilde{\tau}_i}{\tilde{y}_i}$ as the intersection between the lines $(\P0{s+2i-4}, \P0{s+2i-3})$ and $(\P0{s+2i-2}, \P0{s+2i-1})$. Due to a similar barycenter argument as in~\eqref{eq:sparsest_sol_odd}, these intersections are well defined and satisfy $\tilde{\tau}_{i} \in (x_{s+2i-3}, x_{s+2i-2})$. Let $\fopt$ be the piecewise-linear spline that coincides with $\fcano$ outside the interval $(x_{s}, x_{s+\alpha})$, and that connects the points $\P0{s-1}$, $\widetilde{\mathrm{P}}_1$, $\ldots$, $\widetilde{\mathrm{P}}_{k+1}$, and $\P0{s+\alpha}$ in that interval. By construction, $\fopt$ satisfies the constraints $\fopt(x_m) = y_{0, m}$, $m\in\{1,\ldots,M\}$. Moreover, once again in a similar manner to~\eqref{eq:sparsest_sol_odd}, we have that $\Vert \fopt \Vert_\Spc{M} = \Vert \fcano \Vert_\Spc{M}$, which implies that $\fopt \in \Spc{V}_0$. Finally, $\fopt$ is of the form~\eqref{eq:sparsest_sol} with the lowest possible sparsity $P = k+1$ in the interval $[x_s, x_{s+\alpha}]$ (by Lemma~\ref{lem:min_sparsity}). Yet there are uncountably many possible choices of $\widetilde{\mathrm{P}}_1$ (it can be any point on a non-singleton line segment). All of these choices lead to a different solution $\fopt \in \Spc{V}_0$ that is uniquely defined, since the choice of $\widetilde{\mathrm{P}}_1$ specifies $\widetilde{\mathrm{P}}_2, \ldots, \widetilde{\mathrm{P}}_{k+1}$. This proves that there are uncountably many solutions of the~\eqref{eq:noiselessclean} with sparsity $k+1$ in $[x_s, x_{s+\alpha}]$, and that there is a single degree of freedom for the choice of these $k+1$ knots. An example of such a sparsest solution is shown in Figure~\ref{fig:2_sat} with $M=5$ and $\alpha = 2$ consecutive saturation intervals. In our algorithm, we simply choose $\widetilde{\mathrm{P}}_1 = \P0{s}$, which yields a function $\fopt$ of the form~\eqref{eq:sparsest_sol} with
\begin{align}
\label{eq:sparsest_sol_even}
\begin{cases}
      \tilde{a}_1 \eqdef a_{s} \text{ and } \tilde{\tau}_1 \eqdef x_{s}; \\
      \tilde{a}_2 \eqdef a_{s+1} + a_{s+2} \text{ and } \tilde{\tau}_2 \eqdef \frac{a_{s+1}x_{s+1} + a_{s+2}x_{s+2}}{\tilde{a}_2}; \\
      \vdots \\
      \tilde{a}_{k+1} \eqdef a_{s+2k-1} + a_{s+2k} \text{ and } \tilde{\tau}_k \eqdef \frac{a_{s+2k-1}x_{s+2k-1} + a_{s+2k}x_{s+2k}}{\tilde{a}_{k+1}}.
\end{cases}
\end{align}

\end{proof}
Theorem~\ref{thm:sparsest_sol} then directly derives from Lemma~\ref{lem:construction_sparsest_sol} applied independently to each saturation interval $[x_{s_n}, x_{s_n + \alpha_n}]$ for $n\in\{1, \ldots, N_s\}$. Note that Lemma~\ref{lem:construction_sparsest_sol} also applies when no saturation occurs, \ie $\alpha_n = 0$. A sparsest solution of the~\eqref{eq:noiselessclean} thus coincides with a function of the form~\eqref{eq:sparsest_sol} constructed in Lemma~\ref{lem:construction_sparsest_sol} in each of these intervals, and with $\fcano$ outside these intervals. Finally, since the behavior of a solution in each saturation interval does not affect its behavior outside of it, the number of degrees of freedom in the set of sparsest solutions of the~\eqref{eq:noiselessclean} is simply the sum of the number of degrees of freedom in each saturation interval. Yet by Lemma~\ref{lem:construction_sparsest_sol}, there are no degrees of freedom in intervals such that $\alpha_n$ is odd (a sparsest solution is uniquely determined on that interval), and there is one when $\alpha_n$ is even. Therefore, the total number of degrees of freedom of the set of sparsest solutions of the~\eqref{eq:noiselessclean} is equal to the number of even values of $\alpha_n$ for $n\in\{1, \ldots, N_s\}$.

\section{Proof of Proposition~\ref{prop:penalized_to_constrained}}
\label{sec:penalized_to_constrained}

   Assume by contradiction that there exist $f_1, f_2 \in \Spc{V}_\lambda$ and $m_0 \in \{1, \ldots M\}$ such that $f_1(x_{m_0})\neq f_2(x_{m_0})$, and let $f_\gamma = \gamma f_1 + (1-\gamma) f_2$, where $0 < \gamma < 1$. We then have
   
   \begin{align}
       & \sum_{m=1}^M E(f_\gamma(x_m), y_{m}) + \lambda \Vert \Op{D}^2 f_\gamma \Vert_{\Spc{M}} \nonumber \\
       &< \gamma \sum_{m=1}^M E(f_1(x_m), y_{m}) + (1-\gamma) \sum_{m=1}^M E(f_2(x_m), y_{m}) + \lambda \Big(\gamma \Vert \Op{D}^2 f_1 \Vert_{\Spc{M}} + (1-\gamma) \Vert \Op{D}^2 f_2 \Vert_{\Spc{M}} \Big) \nonumber \\
       &=  \gamma \Spc{J}_\lambda + (1-\gamma) \Spc{J}_\lambda = \Spc{J}_\lambda,
   \end{align}
   where $\Spc{J}_\lambda$ is the optimal cost of the~\eqref{eq:noisyclean}. The inequality is due to the convexity of the $\Vert \cdot \Vert_\Spc{M}$ norm and of $E(\cdot, y)$ for any $y \in \R$. The fact that it is strict is due to the strict convexity of $E(\cdot, y_{m_0})$ and the fact that $f_1(x_{m_0})\neq f_2(x_{m_0})$. Yet since $\Spc{V}_\lambda$ is a convex set, we have $f_\gamma \in \Spc{V}_\lambda$: this implies that $\Spc{J}_\lambda = \sum_{m=1}^M E(f_\gamma(x_m), y_m) + \lambda \Vert \Op{D}^2 f_\gamma \Vert_{\Spc{M}} < \Spc{J}_\lambda $, which yields a contradiction.
   
   Therefore, there exists a unique vector $\V{y}_\lambda \in \R^M$ such that for any $\fopt \in \mathcal{V}_\lambda$, $\fopt(x_m) = y_{\lambda, m}$ for all $m\in\{1,\ldots,M\}$. This implies that $\Spc{V}_\lambda \subset \{f \in \BV: f(x_m) = y_{\lambda, m}, \ 1\leq m\leq M \}$. Moreover, we have that for any $\fopt \in \mathcal{V}_\lambda$, $E(\fopt(x_m), y_{m}) = E(y_{\lambda, m}, y_m)$, and thus that the data fidelity is constant in the constrained space $\{f \in \BV: f(x_m) = y_{\lambda, m}, \ 1\leq m\leq M \}$. This proves the equality between the solution sets of the~\eqref{eq:noisyclean} and~\eqref{eq:noisy_constrained}.

\section{ Proof of Proposition~\ref{prop:linear_regression}}
\label{sec:linear_regression_proof}
\paragraph{Item 1}Let $J(\beta_0, \beta_1) = \sum_{m=1}^M E(\beta_0 + \beta_1 x_m, y_m)$ be the objective function of problem~\eqref{eq:linear_regression}. We show that problem~\eqref{eq:linear_regression} indeed has a unique solution by proving that $J$ is strictly convex and coercive when $M \geq 2$ and the $x_m$ are pairwise distinct.

Concerning the coercivity, let $\Vert (\beta_0, \beta_1) \Vert_2 \to +\infty$. Assume by contradiction that $\beta_0 + \beta_1 x_m$ is bounded for every $m\in \{1, \ldots, M \}$. Then, since $M \geq 2$, $\beta_0 + \beta_1 x_1- (\beta_0 + \beta_1 x_2) = \beta_1(x_1-x_2)$ must also be bounded, which implies that $\beta_1$ is bounded since the $x_m$ are pairwise distinct. Therefore, we must have $\vert \beta_0 \vert \to +\infty$, which implies that $\vert \beta_0 + \beta_1 x_1 \vert \to +\infty$ which yields a contradiction. Therefore, there exists a $m_0 \in \{ 1, \ldots, M \}$ such that  $\vert \beta_0 + \beta_1 x_{m_0} \vert \to +\infty$. The coercivity of $J$ then directly follows from that of $E(\cdot, y_{m_0})$.

Next, to prove the strict convexity of $J$, let $(\beta_0, \beta_1), (\beta_0', \beta_1') \in \R^2$ with $(\beta_0, \beta_1) \neq (\beta_0', \beta_1')$, and $0 < s < 1$. For any $m$, we have $s \beta_0 + (1-s) \beta_0' + (s \beta_1 + (1-s) \beta_1') x_m = s (\beta_0 + \beta_1 x_m) + (1-s) (\beta_0' + \beta_1' x_m)$. Since $(\beta_0, \beta_1) \neq (\beta_0', \beta_1')$ and the $x_m$ are distinct, the equation $\beta_0 + \beta_1 x_m= \beta_0' + \beta_1' x_m$ can only be satisfied for at most a single $m \in \{1, \ldots, M \}$. Yet $M \geq 2$, which implies that $\exists m_0, \ \beta_0 + \beta_1 x_{m_0} \neq \beta_0' + \beta_1' x_{m_0}$. Therefore, due to the strict convexity of $E(\cdot, y_{m_0})$, we have  
\begin{align}
E((s \beta_0 + (1-s) \beta_0') + (s \beta_1 + (1-s) \beta_1')x_{m_0}), y_{m_0}) < s E(\beta_0 + \beta_1 x_{m_0}, y_{m_0}) + (1-s) E(\beta_0' + \beta_1' x_{m_0}, y_{m_0}).
\end{align}
It then follows from the convexity of $E(\cdot, y_{m})$ for all $m$ that $J \big( s (\beta_0, \beta_1) + (1-s) (\beta_0', \beta_1') \big) < s J(\beta_0, \beta_1) + (1-s) J(\beta_0', \beta_1')$, which proves the strict convexity of $J$. Together with the fact that $J$ is coercive, this proves that~\eqref{eq:linear_regression} has a unique solution.

\paragraph{Item 2}Assume that $\lambda \geq \lambda_{\text{max}}$. By Fermat's rule, a vector $\zopt$ is a solution of problem~\eqref{eq:optizlambda} if and only if the zero vector belongs to the subdifferential of the objective function evaluated at $\zopt$. We thus have $\zopt = \V{y}_\lambda$ if and only if
\begin{align}
\label{eq:optimality_condition_discrete}
    \V{0} \in \underbrace{\begin{pmatrix}\partial_1 E(z_{\mathrm{opt}, 1}, y_1) \\ \vdots \\ \partial_1 E(z_{\mathrm{opt}, M}, y_M) \end{pmatrix}}_{\eqdef \V{v}(\zopt)} + \lambda \partial  \Vert \M{L} \cdot \Vert_1(\zopt),
\end{align}
where $\partial_1$ denotes the partial derivative with respect to the first variable, and $\partial$ the subdifferential. The chain rule for subdifferentials \cite[Theorem 23.9.]{rockafellar1970convex} yields $\partial  \Vert \M{L} \cdot \Vert_1(\V{z}) = \{ \M{L}^T \V{g}: \V{g} \in \partial \Vert \cdot \Vert_1(\M{L} \V{z}) \subset \R^{M-2}  \}$, where $\partial \Vert \cdot \Vert_1(\V{a}) = \{ \V{g} \in \R^{M-2}: \Vert \V{g} \Vert_\infty \leq 1, \ \V{a}^T \V{g} = \Vert \V{a} \Vert_1 \}$. The vector $\M{L} \zopt$ lists the weights $a_m$ associated to the knots of the canonical solution $f_{\zopt}$ (see the proof of Proposition~\ref{prop:optizlambda}). Therefore, the linear regression case (in which $f_{\zopt}$ has no knot) corresponds to $\M{L} \zopt = \V{0}$. In this case, since $\partial \Vert \cdot \Vert_1(\V{0}) = \{ \V{g} \in \R^{M-2}: \Vert \V{g} \Vert_\infty \leq 1 \}$, the optimality condition~\eqref{eq:optimality_condition_discrete} now reads
\eq{
    \exists \V{g} \in \R^{M-2}, \quad \Vert \V{g} \Vert_\infty \leq 1, \quad \text{s.t.} \quad \V{v}(\zopt) + \lambda \M{L}^T \V{g} = \V{0}.
}
We now prove that $\zopt = \opt \beta_0 \V{1} + \opt \beta_1 \x $ satisfies the optimality conditions~\eqref{eq:optimality_condition_discrete}, and thus that $\V{y}_\lambda = \opt \beta_0 \V{1} + \opt \beta_1 \x$. To achieve this, we prove that $\V{g} = - \frac{1}{\lambda} {\M{L}^T}^\dagger \V{v}(\zopt)$ satisfies $\V{v}(\zopt) + \lambda \M{L}^T \V{g} = \V{0}$. Firstly, since $\lambda \geq \lambda_{\text{max}}$, we have that $\Vert \V{g} \Vert_\infty \leq 1$ by definition of $\lambda_{\text{max}}$. Next, let $V$ be the orthogonal complement of $\ker \M{L} \subset \R^M$. A known property of the pseudoinverse operator~\cite[Corollary 7]{benisrael2003generalized} is that $\M{L}^T {\M{L}^T}^\dagger$ is the orthogonal projection operator onto $V$. By decomposing $\V{v}(\zopt)= \V{v}_1 + \V{v}_2$, where $\V{v}_1 \in V$ and $\V{v}_2 \in \ker \M{L}$, we thus get $\V{v}(\zopt) + \lambda \M{L}^T \V{g} = \V{v}_2$. Yet $\ker \M{L} = \mathrm{span} \{ \V{1}, \x \}$, since the canonical solutions $f_{\V{1}}$ and $f_{\x}$ (that satisfy $f_{\V{1}}(x_m) = 1$ and $f_{\V{x}}(x_m) = x_m$ for every $m\in \{ 1, \ldots, M\}$ respectively) are linear functions that are thus not penalized by the regularization. The optimality conditions of problem~\eqref{eq:linear_regression} (\ie setting the gradient to zero) then yield $\V{v}(\zopt) \perp \ker \M{L}$, which implies that $\V{v}_2 = \V{0}$ and thus that $\V{v}(\zopt) + \lambda \M{L}^T \V{g} = \V{0}$. This proves that $\zopt$ satisfies the optimality condition of problem~\eqref{eq:optizlambda}, and thus that $\zopt = \V{y}_\lambda = \opt \beta_0 \V{1}+ \opt \beta_1 \x$.

\paragraph{Item 3} Due to item 2, we have $\V{y}_\lambda = \opt \beta_0 \V{1}+ \opt \beta_1 \x$ which implies that the points $\Point{x_m}{y_{\lambda, m}}$ are aligned. Hence, the canonical dual certificate of the constrained problem~\eqref{eq:noisy_constrained} is $\etacano = 0$, which is nondegenerate. By Proposition~\ref{prop:uniqueness}, this implies that the unique solution to problem~\eqref{eq:noisy_constrained} is the canonical solution $f_{\zopt} = f_{\text{max}} = \opt \beta_0 + \opt \beta_1 (\cdot)$. Due to the equivalence between problems~\eqref{eq:noisy_constrained} and the~\eqref{eq:noisyclean} proved in Proposition~\ref{prop:penalized_to_constrained}, this concludes the proof.